\documentclass[a4paper, 10pt, twoside, notitlepage]{amsart}

\usepackage[utf8]{inputenc}
\usepackage{color}
\usepackage{amsmath} 
\usepackage{amssymb} 
\usepackage{amsthm}
\usepackage{geometry}
\usepackage{graphicx}
\usepackage{esint}
\usepackage{mathtools}
\usepackage[colorlinks=true,linkcolor=blue]{hyperref}

\DeclarePairedDelimiter\floor{\lfloor}{\rfloor}

\theoremstyle{plain}
\newtheorem{thm}{Theorem}
\newtheorem{prop}{Proposition}[section]
\newtheorem{lem}[prop]{Lemma}
\newtheorem{cor}[prop]{Corollary}

\newtheorem{rmk}[prop]{Remark}

\newcommand {\R} {\mathbb{R}} 
 \newcommand {\N} {\mathbb{N}}
 
\newcommand {\p} {\partial}

\newcommand {\D} {\Delta}

\newcommand {\supp} {\text{supp}}

\DeclareMathOperator{\Dom}{Dom}

\DeclareMathOperator{\di}{div}

\DeclareMathOperator {\dist} {dist}

\DeclareMathOperator{\F} {\mathcal{F}}

\DeclareMathOperator{\spec}{spec}

\title{Strong unique continuation for the higher order fractional Laplacian}

\author{Mar\'ia-\'Angeles Garc\'ia-Ferrero}
\address{Max-Planck-Institute for Mathematics in the Sciences, Inselstr. 22, 04103 Leipzig}
\email{garcia@mis.mpg.de}

\author{Angkana Rüland}
\address{Max-Planck-Institute for Mathematics in the Sciences, Inselstr. 22, 04103 Leipzig}
\email{rueland@mis.mpg.de}

\begin{document}

\begin{abstract}
In this article we study the strong unique continuation property for solutions of higher order (variable coefficient) fractional Schrödinger operators. We deduce the strong unique continuation property in the presence of subcritical and critical Hardy type potentials. In the same setting, we address the unique continuation property from measurable sets of positive Lebesgue measure. As applications we prove the antilocality of the higher order fractional Laplacian and Runge type approximation theorems which have recently been exploited in the context of nonlocal Calder\'on type problems. As our main tools, we rely on the characterisation of the higher order fractional Laplacian through a generalised Caffarelli-Silvestre type extension problem and on adapted, iterated Carleman estimates.
\end{abstract}

\maketitle

\section{Introduction}
\label{sec:intro}

Higher order local and nonlocal elliptic equations arise naturally in problems from conformal geometry and scattering theory \cite{CdMG11, GZ03}, (nonlinear) higher order elliptic PDEs and free boundary value problems \cite{S84,CF79}, control theory \cite{AKW18, BHS17} and inverse problems \cite{GSU16,GRSU18,RS17}. Motivated by these applications, in this article we study the \emph{strong unique continuation property} for \emph{higher order fractional Schrödinger equations}. More precisely, here we are concerned with equations of the form
\begin{align}
\label{eq:frac}
(-\D)^{\gamma} u + q u  = 0 \mbox{ in } \R^n,
\end{align}
where $\gamma \in \R_+ \setminus \N$ with suitable, possibly singular (critical and subcritical) potentials $q$. Here we say that a solution $u$ to \eqref{eq:frac} satisfies the strong unique continuation property if the condition that $u$ vanishes of infinite order at a point $x_0 \in \R^n$, i.e. if for all $m\in \N$
\begin{align*}
\lim\limits_{r\rightarrow 0}r^{-m} \|u\|_{L^2(B_r(x_0))}^2 = 0,
\end{align*}
already implies that $u\equiv 0$ in $\R^n$. The strong unique continuation property can hence be viewed as a generalisation of analyticity to rougher equations.

Apart from dealing with the model equation \eqref{eq:frac}, we also address the corresponding differential inequalities and the setting of variable coefficient fractional Schrödinger operators with coefficients which might be only of low regularity. This extends the results from \cite{Rue15} and \cite{Yu16}, where $C^2$ and $C^{1,\alpha}$ regular coefficients had been treated in the case $\gamma \in (0,1)$, to the setting of Lipschitz coefficients. Further, we discuss possible applications of the unique continuation results to inverse and control theoretic problems.

\subsection{Main results on the strong unique continuation property}

Let us outline our main results: As a model situation we deal with the strong unique continuation property (SUCP) for Schrödinger equations of the form \eqref{eq:frac}. Without loss of generality, here we normalise our set-up such that $x_0=0$.

\begin{thm}[SUCP]
\label{prop:SUCP}
Let $\gamma \in \R_+ \setminus \N $ and let $u\in H^{2 \gamma}(\R^n)$ be a solution to \eqref{eq:frac},
where 
the potential $q$ satisfies the following bounds
\begin{align*}
|q(x)|
\leq \left\{ 
\begin{array}{ll}
C_q|x|^{-2\gamma}, \mbox{ if } \gamma> \frac{1}{2},\\
c_0|x|^{-2\gamma}, \mbox{ if } \gamma = \frac{1}{2},\\
C_q |x|^{-2\gamma+\epsilon}, \mbox{ if } \gamma \in (\frac{1}{4},\frac{1}{2}).
\end{array}
\right.
\end{align*}
Here $c_0>0$ is a sufficiently small constant, and $C_q>0$ is an arbitrarily large, finite constant.
Assume that $u$ vanishes of infinite order at $x_0=0$, i.e. for all $m\in \N$
\begin{align*}
\lim\limits_{r\rightarrow 0}r^{-m} \|u\|_{L^2(B_r(0))}^2 = 0.
\end{align*}
Then $u\equiv 0$ in $\R^n$.
\end{thm}

\begin{rmk}
\label{rmk:small_gamma}
We remark that the limitation of the result to $\gamma>\frac{1}{4}$ arises naturally and was already present in \cite{Rue15}: Relying on Carleman estimates with weights which only have a radial dependence, we do not directly obtain positive boundary contributions but have to derive these through boundary-bulk interpolation estimates (see Lemma ~\ref{lem:bdry_bulk} and Corollary ~\ref{cor:boundary_bulk}). With respect to bulk estimates, $L^2$ Carleman estimates are however subelliptic in the large parameter $\tau$. Through the boundary-bulk estimates this is propagated to the boundary which is then reflected in the loss of a quarter derivative in $\tau$ on the boundary. In the case that one only considers radial Carleman weights this loss seems unavoidable. In order to extend the unique continuation results to the regime to $\gamma \in (0,\frac{1}{4})$ in a setting where only radial Carleman weights are used, the loss in $\tau$ has thus to be compensated by regularity of the potential (see \cite{Rue15} for corresponding results). In this case the lower order contributions would be included in the main part of the operator in the Carleman estimates (this then allows one to treat exact Hardy type potentials, but any type of perturbation of such potentials will need to obey regularity assumptions). 

One could hope to avoid this loss of derivatives by considering Carleman weights which are not only of a radial structure but also depend on the normal directions. However, due to the weighted form of the inequalities, the exact Lopatinskii type conditions necessary for this are not immediate. We do not pursue this further in this article but postpone this to future work.
\end{rmk}

\begin{rmk}
\label{rmk:nonlinearity}
It would also have been possible to treat additional nonlinear terms in the the equations. As we are mainly interested in the associated differential inequalities, we do not consider them here.
\end{rmk}

Motivated by the work on the strong unique continuation properties on higher order elliptic equations (see \cite{CK10} and the references therein), it is natural to wonder whether it is possible to extend the unique continuation property to (Hardy type higher) gradient potentials. Using iterative applications of our main Carleman estimate, we note that this is indeed the case:

\begin{thm}[SUCP with gradient potentials]
\label{prop:SUCP_higher}
Let $\gamma \in \R_+ \setminus \N $ and let $u\in H^{2 \gamma}(\R^n)$ be a weak solution to the differential inequality
\begin{align*}
|(-\D)^{\gamma}u(x)| \leq \sum\limits_{j=0}^{\floor{\gamma}}|q_j(x)||\nabla^j u(x)|  \mbox{ in } \R^n,
\end{align*}
where 
the potentials $q_j$ satisfy the following bounds
\begin{align*}
|q_j(x)|&\leq C_{q_j}|x|^{-2\gamma + j} \mbox{ if } j<\floor{\gamma},\\
|q_{\floor{\gamma}}(x)|&
\leq \left\{ 
\begin{array}{ll}
C_{q_{\floor{\gamma}}}|x|^{-2\gamma+\floor{\gamma}}, \mbox{ if } \gamma-\floor{\gamma}> \frac{1}{2},\\
c_0|x|^{-2\gamma+j}, \mbox{ if } \gamma-\floor{\gamma} = \frac{1}{2},\\
C_{q_{\floor{\gamma}}} |x|^{-2\gamma+\floor{\gamma}+\epsilon}, \mbox{ if } 
\left\{
\begin{array}{l}
\gamma\in(\frac 1 4,\frac 1 2),\\
\floor{\gamma}\geq 1 \mbox{ and }\gamma - \floor{\gamma} \in (0,\frac{1}{2}),
\end{array}
\right.
\end{array}
\right.
\end{align*}
Here $c_0>0$ is a sufficiently small constant, and $C_{q_j}>0$ are arbitrarily large, finite constants.
Assume that $u$ vanishes of infinite order at $x_0=0$.
Then $u\equiv 0$ in $\R^n$.
\end{thm}

\begin{rmk}
\label{rmk:improve}
As in \cite{CK10} it would have been possible to extend this result even further to (slightly subcritical) gradient potentials  involving also contributions $|q_j(x)||\nabla^j u(x)|$ with $j \in (\floor{\gamma},\frac{3}{2}\floor{\gamma})$. As this requires some extra care, we do not present the details of this here but refer to the ideas in \cite{CK10}.
\end{rmk}

Further, relying on the methods from \cite{Rue15, KT01} as well as the spliting argument from \cite{GRSU18}, we also address the case with variable coefficient metrics and study the unique continuation properties at a point $x_0 \in \R^n$. In the sequel, without loss of generality, we will normalise our set-up such that $x_0=0$. Under this assumption, we consider the operator
\begin{align*}
L=-\nabla \cdot \tilde a \nabla
\end{align*}
for Lipschitz metrics with the following structural conditions:
\begin{itemize}
\item[(A1)]\label{cond:a1}  $\tilde{a}: \R^n \rightarrow \R^{n\times n}$ is symmetric, (strictly) positive definite, bounded.
\item[(A2)]\label{cond:a2} $\tilde{a}\in C^{\mu,1}_{loc}(\R^n, \R^{n\times n}_{sym})$ with $[\tilde{a}^{ij}]_{\dot{C}^{\mu,1}(B_4')} + [\tilde{a}^{ij}]_{\dot{C}^{0,1}(B_4')} \ll \delta$, where ${B_4':=\{x\in \R^n: \ |x|\leq 4\}}$, for some small parameter $\delta>0$. The constant $\mu>0$ is specified below.
\item[(A3)]\label{cond:a3} We assume that $\tilde{a}^{ij}(0)=\delta^{ij}$.
\end{itemize}

For this class of coefficients, we can prove the analogue of Theorem ~\ref{prop:SUCP}:
\begin{thm}[SUCP with variable coefficients]
\label{prop:SUCP_var}
Let $\gamma \in \R_+ \setminus \N $ and let $u\in \Dom(L^{\gamma})$ be a solution to 
\begin{align}
\label{eq:frac1}
|(-\nabla \cdot \tilde{a} \nabla)^{\gamma} u(x) | \leq \sum\limits_{j=0}^{\floor{\gamma}}|q_j(x)||\nabla^j u(x)|  \mbox{ in } \R^n, 
\end{align}
where
\begin{itemize}
\item the metric $\tilde{a}$ satisfies the conditions \hyperref[cond:a1]{(A1)}-\hyperref[cond:a3]{(A3)} with $\mu = 2\floor{\gamma}$,
\item the potentials $q_j$ satisfy the following bounds
\begin{align*}
|q_j(x)|&\leq C_{q_j}|x|^{-2\gamma + j} \mbox{ if } j<\floor{\gamma},\\
|q_{\floor{\gamma}}(x)|&
\leq \left\{ 
\begin{array}{ll}
C_{q_{\floor{\gamma}}}|x|^{-2\gamma+\floor{\gamma}}, \mbox{ if } \gamma-\floor{\gamma}> \frac{1}{2},\\
c_0|x|^{-2\gamma+j}, \mbox{ if } \gamma-\floor{\gamma} = \frac{1}{2},\\
C_{q_{\floor{\gamma}}} |x|^{-2\gamma+\floor{\gamma}+\epsilon}, \mbox{ if } 
\left\{
\begin{array}{l}
\gamma\in(\frac 1 4,\frac 1 2),\\
\floor{\gamma}\geq 1 \mbox{ and }\gamma - \floor{\gamma} \in (0,\frac{1}{2}),
\end{array}
\right.
\end{array}
\right.
\end{align*}
where $c_0>0$ is a sufficiently small constant, and $C_{q_j}>0$ are arbitrarily large, finite constants.
\end{itemize}
Then the strong unique continuation property holds at $x_0=0$, i.e. if $u$ vanishes of infinite order at $x_0 =0$, then $u\equiv 0$ in $\R^n$.
\end{thm}

\begin{rmk}
\label{rmk:reg}
As explained in the Appendix (Section \ref{sec:var_coef}), we interpret the variable coefficient fractional Laplacian through its spectral decomposition as directly related to a generalised Caffarelli-Silvestre extension. Relying on spectral theory in order to establish this, we restrict to functions ${u\in \Dom(L^{\gamma})}$. Using ideas as outlined in \cite{CaffStinga16} and \cite{Y13}, it would also have been possible to lower the required regularity of $u$ in this discussion. 
\end{rmk}

\begin{rmk}
\label{rmk:optimal_reg}
Based on counterexamples to the weak unique continuation property with metrics of any $C^{0,\alpha}$ Hölder regularity with $\alpha \in (0,1)$ due to Miller \cite{M74} and Mandache \cite{M98}, it is expected that the coefficient regularity stated in condition \hyperref[cond:a2]{(A2)} in our variable coefficient strong unique continuation result is optimal in the case $\floor{\gamma}=0$. This strengthens the results from Section 7 in \cite{Rue15} and \cite{Yu16}. 

The condition \hyperref[cond:a3]{(A3)} is to be read as a normalisation condition which can be assumed without loss of generality. We remark that condition \hyperref[cond:a2]{(A2)} together with interpolation estimates implies that $[a^{ij}]_{\dot{C}^{\ell,1}(B_4')}\leq \tilde{C} \delta$ for any $\ell \in \{1,\dots,\mu\}$.
\end{rmk}

In both the settings of Theorems ~\ref{prop:SUCP} and ~\ref{prop:SUCP_var} also the \emph{unique continuation property from measurable sets} (MUCP) holds:

\begin{thm}[MUCP] 
\label{prop:MUCP}
Let $\gamma \in \R_+\setminus \N$ and let $u \in \Dom(L^{\gamma})$ be a solution to 
\begin{align}
\label{eq:diff_ineq_frac}
|(-\nabla \cdot \tilde{a} \nabla)^{\gamma} u(x)| \leq |q(x)|| u(x)| \mbox{ in } \R^n,
\end{align}
where $\tilde{a}^{ij}$ satisfies the conditions \hyperref[cond:a1]{(A1)}-\hyperref[cond:a3]{(A3)}  with $\mu = 2\floor{\gamma}$ 
and $q\in L^\infty(\R^n)$.
If there exists a measurable set $E\subset\R^n$ with $|E|>0$ and density one at $x_0=0$  such that $u|_{E}=0$,  then    $u\equiv 0$ in $\R^n$.
\end{thm}

Let us comment on the results of Theorems ~\ref{prop:SUCP}-\ref{prop:MUCP} in the context of the literature on fractional Schrödinger equations: The \emph{weak unique continuation property}, i.e. the question whether for solutions $u$ of \eqref{eq:frac} the condition that $u=0$ on an open set in $\R^n$ already implies that ${u\equiv 0}$ in the whole of $\R^n$ is rather well understood
for constant coefficient fractional Schrödinger equations, even for potentials in very rough, non-$L^2$-based function spaces (see \cite{Seo}). In contrast, the understanding of the \emph{strong unique continuation properties} of solutions to (higher order) fractional Schrödinger equations is still much less developed (for non-fractional higher order Schrödinger operators we refer to \cite{CK10} and the references therein). The main known results are here given in the regime $\gamma \in(0,1)$ and can be summarised as the following statements:
\begin{itemize}
\item \emph{Strong unique continuation for constant coefficient fractional Schrödinger equations with  essentially $L^{\infty}$ potentials.} In the regime $\gamma\in(0,1)$ the articles \cite{FF14, Rue15} deal with the strong unique continuation property for $L^{\infty}$ as well as ``Hardy type" critical and subcritical potentials. Both results crucially exploit the possibility of rephrasing the fractional Schrödinger operator in terms of the Caffarelli-Silvestre extension (see \cite{CS07}), i.e. in terms of a (degenerate) Dirichlet-to-Neumann map associated with a (degenerate) elliptic, local equation in the upper half-plane. Technically, this allowed the authors of \cite{FF14} to rely on frequency function methods for local equations, while, similarly, the key tool in \cite{Rue15} consisted of several Carleman inequalities for the Caffarelli-Silvestre extension.
\item \emph{Unique continuation property from measurable sets.} Relying on the arguments from \cite{FF14, Rue15} also unique continuation results from measureable sets can be proved in the regime $\gamma\in(0,1)$. Indeed, in \cite{FF14} this is formulated as one of the main results. For rougher equations this is deduced in \cite{GRSU18} based on variants of the Carleman estimates from \cite{Rue15}.
\item \emph{Variable coefficient operators.} Using more refined frequency function or Carleman estimates, also the case of variable coefficient fractional Schrödinger equations has been treated in \cite{Rue15} ($C^2$ regular coefficients, see Section 7) and in \cite{Yu16} ($C^{1,\alpha}$ regular coefficients).
\end{itemize}
In contrast, the situation for higher order fractional Schrödinger operators is much less studied. Here the main known properties are:
\begin{itemize}
\item \emph{Representation of the equation through a Caffarelli-Silvestre type extension problem.} In \cite{Y13} and in \cite{CdMG11} and later also in \cite{RonSti16} it was observed that the higher order fractional Laplacian can be realised as a Dirichlet-to-Neumann map of a Caffarelli-Silvestre type extension problem. This can either take the form of a scalar equation (however with a weight which is no longer in the Muckenhoupt class) or a system of Caffarelli-Silvestre extensions.
\item \emph{Strong unique continuation for fractional harmonic functions.} Exploiting the systems characterisation from \cite{Y13}, Yang also sketches the proof of the strong unique continuation property for fractional harmonic functions based on frequency function methods for systems of equations. In the regime $\gamma \in(1,2)$ this was further detailed in \cite{FelliFerrero18a}, where the authors also obtained precise asymptotics of the solutions under consideration.
\item \emph{Strong unique continuation for fractional Schrödinger equations.} In the case $\gamma=\frac{3}{2}$ the strong unique continuation property for Schrödinger equations with Hardy type potentials was recently proved in \cite{FelliFerrero18b}. In this context, the regime $\gamma=\frac{3}{2}$ is special, since the degeneracy of the problem disappears and the result can be viewed as a boundary unique continuation result for the Bilaplacian. The authors again rely on frequency function methods.
\end{itemize}

In contrast, in this article we study Carleman estimates to deduce the desired unique continuation property. Also relying on the systems Caffarelli-Silvestre extension of the higher order fractional Laplacian (which is recalled in the Appendix), we view the various unique continuation properties from above as \emph{boundary unique continuation results}. In order to deduce these, we hence derive Carleman estimates for the associated systems. This is inspired by the work in \cite{CK10} in which unique continuation in the interior is discussed for higher order equations (or equivalently systems). As in \cite{Rue15} we combine these Carleman estimates with careful compactness and blow-up arguments.

We emphasise that the results from Theorems ~\ref{prop:SUCP}-~\ref{prop:MUCP} improve significantly  on the known strong unique continuation results for the fractional Laplacian by for instance including (Hardy type) potentials and variable coefficients of low regularity. The strength of these aspects are even novel for the regime $\gamma\in(0,1)$.

\begin{rmk}
Using a characterisation of the higher order fractional Laplacian through a system and a bootstrap argument, for $a^{ij}=\delta^{ij}$ we here impose $H^{2\gamma}$ regularity on the solutions to the equations at hand. We remark that even in the setting of fractional Schrödinger equations in bounded domains, it would be possible to apply our arguments: Indeed, starting from $H^{\gamma}$ solutions, it would be possible to bootstrap the regularity properties of the solutions by means of the estimates in \cite{G15} (away from the boundary).

We however emphasise that by pseudolocality of the fractional Laplacian our results could also be formulated under only local regularity assumptions: Assuming that we had a weak notion of a generalised Caffarelli-Silvestre extension (which is the case for any $H^{r}(\R^n)$, $r\in \R$, boundary datum, see Section \ref{sec:const_coef}) as well as only local $H^{2\gamma}(B_1')$ regularity for $u$, it would have been possible to invoke our unique continuation arguments. We refer to Proposition \ref{prop:antiloc} and Remark \ref{rmk:low_reg} for more on this.
\end{rmk}

\subsection{Main ideas}

Approaching the problem by means of the generalised systems Caffarelli-Silvestre extension, our arguments for unique continuation rely on several Carleman estimates for the localised equation in combination with a careful blow-up analysis. To this end, we rely on similar ideas as in \cite{Rue15, GRSU18}. 

A crucial technical tool thus is the derivation of higher order Carleman estimates, which we address by iteration of second order estimates. As a consequence, we obtain the following bounds:

\begin{prop}
\label{prop:syst_Carl}
Let $m\in \N\cup \{0\}$ and $b\in (-1,1)$. 
Let $a\in C^{2m,1}(B_4^+,\R^{(n+1)\times(n+1)})\cap C^{0,1}(B_4^+, \R^{(n+1)\times (n+1)})$ be of a block form
\begin{align}\label{eq:block}
a(x) = \begin{pmatrix} \tilde{a}(x') & 0 \\ 0 & 1 \end{pmatrix},
\end{align}
where the metric $\tilde{a}$ satisfies the conditions \hyperref[cond:a1]{(A1)}-\hyperref[cond:a3]{(A3)} from above with $\mu =2m$.
Assume that $\tilde{u}_0,\dots,\tilde{u}_m \in H^{1}(B_4^+, x_{n+1}^b)$ with $\supp(\tilde{u}_j) \subset B_{4}^+\setminus\{0\}$, $j\in \{0,\dots,m\}$, are solutions to the bulk system
\begin{align}
\label{eq:syst_err}
\begin{split}
x_{n+1}^{-b} \nabla \cdot x_{n+1}^{b}a \nabla \tilde{u}_0 & = \tilde{u}_1 + f_0,\\
x_{n+1}^{-b} \nabla \cdot x_{n+1}^{b} a \nabla \tilde{u}_1 & = \tilde{u}_2 + f_1,\\
\vdots\\
x_{n+1}^{-b} \nabla \cdot x_{n+1}^{b} a \nabla \tilde{u}_m & = f_{m},
\end{split}
\end{align}
in $B_4^+$ with $f_{0},\dots, f_m \in H^{1}(B_4^+, x_{n+1}^b)$. Further suppose that on $B_4'$ we have
\begin{align}
\label{eq:syst_err_boundary}
\begin{split}
\lim\limits_{x_{n+1}\rightarrow 0} x_{n+1}^b \p_{n+1} \tilde{u}_j & = 0 \mbox{ for } j \in \{0,\dots,m-1\},
\\
\lim\limits_{x_{n+1} \rightarrow 0} x_{n+1}^b \p_{n+1} \tilde{u}_{m} & = g , \\
\lim\limits_{x_{n+1}\rightarrow 0} \tilde{u}_0 &= u ,
\end{split}
\end{align}
where all limits are considered with respect to the $L^2$ topology.
Then, there exists $\tau_0>1$ such that for all $\tau>\tau_0$ there is a weight $h$ such that 
\begin{align}
\label{eq:Carl_system}
\begin{split}
&\sum\limits_{j=0}^{m} \left( \tau^{m+1-j} \|e^{h(-\ln(|x|))} x_{n+1}^{\frac{b}{2}}(1+\overline{h})^{\frac{m+1-j}{2}} |x|^{-2m-1+2j} \tilde{u}_j \|_{L^2(B_4^+)} \right.\\
&  \qquad +  \left.  \tau^{m-j} \|e^{h(-\ln(|x|))} x_{n+1}^{\frac{b}{2}} (1+ \overline{h})^{\frac{m+1-j}{2}} |x|^{-2m+2j} \nabla \tilde{u}_j \|_{L^2(B_4^+)}  \right)\\
&\leq C\Big( \sum\limits_{j=0}^{m} \tau^{m-j}\|e^{h(-\ln(|x|))}  x_{n+1}^{\frac{b}{2}}(1+\overline{h})^{\frac{m-j}{2}} |x|^{-2m+1+2j}f_j\|_{L^2(B_4^+)}\\
& \qquad \quad+  \tau^{\frac{1+b}{2}} \|e^{h(-\ln(|x|))}|x|^{\frac{1-b}{2}} g\|_{L^2(B_4')}\Big).
\end{split}
\end{align}
Here $\overline{h}(x):=  h''(t)|_{t=-\ln(|x|)}$.
\end{prop}

This estimate improves previous results even in the case $m=0$ by allowing for only Lipschitz continuous metrics $a$ and the derivation of bounds which exploit the spectral gap of the fractional Laplacian on the sphere with Neumann (or Dirichlet) conditions (see the Appendix A, Section 8.3 in \cite{KRSIV} for these spectral gap properties). A relevant ingredient in the derivation of this Carleman estimate involves the use of a splitting technique in a similar way as in \cite{GRSU18}.

Estimating the commutators, the bounds from Proposition \ref{prop:syst_Carl} can be further improved:

\begin{prop}
\label{prop:syst_Carl_1}
Let $m\in \N\cup \{0\}$, $b\in (-1,1)$ and $a:\R^{n+1}_+ \rightarrow \R^{(n+1)\times (n+1)}$ as in Proposition ~\ref{prop:syst_Carl}. 
Assume that $\tilde{u}_0,\dots,\tilde{u}_m \in H^{1}(B_4^+, x_{n+1}^b)$ with $\supp(\tilde{u}_j) \subset B_{4}^+\setminus\{0\}$, $j\in \{0,\dots,m\}$, are solutions to the bulk system \eqref{eq:syst_err}
in $B_4^+$ with $f_{j}\in{H^{m-j}(B_4^+, x_{n+1}^b)}$ for $j\in \{0,\dots,m\}$. Further suppose that on $B_4'$ the boundary conditions \eqref{eq:syst_err_boundary} hold,
where all limits are considered with respect to the $L^2$ topology.
Then, there exists $\tau_0>1$ such that for all $\tau>\tau_0$ there is a weight $h$ such that 
\begin{align}
\label{eq:Carl_system_1}
\begin{split}
&\sum\limits_{j=0}^{m} \tau^{m-j+\frac{1-b}{2}} \|e^{h(-\ln(|x|))} x_{n+1}^{\frac{b}{2}}(1+\overline{h})^{\frac{m+1}{2}} |x|^{-2m+j+\frac{b-1}{2}} (\nabla')^j \tilde{u}_0 \|_{L^2(B_4')}\\
& \qquad +\sum\limits_{j=0}^{m}\sum_{k=0}^{j} \left( \tau^{m+1-j} \|e^{h(-\ln(|x|))} x_{n+1}^{\frac{b}{2}}(1+\overline{h})^{\frac{m+1-k}{2}} |x|^{-2m-1+j+k} (\nabla')^{j-k} \tilde{u}_{k} \|_{L^2(B_4^+)} \right.\\
&   \qquad\qquad+  \left.  \tau^{m-j} \|e^{h(-\ln(|x|))} x_{n+1}^{\frac{b}{2}} (1+ \overline{h})^{\frac{m+1-k}{2}} |x|^{-2m+j+k} \nabla (\nabla')^{j-k} \tilde{u}_{k} \|_{L^2(B_4^+)}  \right)\\
&\leq C \Big( \sum\limits_{j=0}^{m} \sum\limits_{k=0}^{j}\tau^{m-j}\|e^{h(-\ln(|x|))}  x_{n+1}^{\frac{b}{2}}(1+\overline{h})^{\frac{m-k}{2}}|x|^{-2m+1+j+k} (\nabla')^{j-k} f_{k}\|_{L^2(B_4^+)}\\
& \quad\qquad +  \tau^{\frac{1+b}{2}} \|e^{h(-\ln(|x|))}|x|^{\frac{1-b}{2}} g\|_{L^2(B_4')}\Big).
\end{split}
\end{align}
Here $\overline{h}(x):= h''(t)|_{t=-\ln(|x|)}$.
\end{prop}

It is in this form that we exploit the Carleman estimates to infer our main results.

\subsection{Applications of the unique continuation results}

Motivated by the recent introduction of the fractional Calder\'on problem \cite{GSU16, RS17, GRSU18}, we here discuss applications of our unique continuation results in inverse problems. As a first main property, we deduce the \emph{antilocality} of the higher order fractional Laplacian:

\begin{prop}[Antilocality]
\label{prop:antiloc}
Let $\gamma \in \R_+\setminus \N$ and let $L=-\nabla \cdot \tilde{a} \nabla$, where $\tilde{a}$ satisfies the conditions \hyperref[cond:a1]{(A1)}-\hyperref[cond:a3]{(A3)} from above with $\mu =2\floor{\gamma}$. Let $u\in \Dom(L^{\gamma})$. Assume that for some open set $W\subset \R^n$ containing the unit ball $B'_1$ we have
\begin{align*}
u= 0 \mbox{ and } L^{\gamma} u & = 0 \mbox{ in } W.
\end{align*}
Then, $u\equiv 0$ in $\R^n$.
\end{prop}

We emphasise that in this result the function $u$ is \emph{not} assumed to satisfy any equation globally.
This result thus provides a strong global rigidity property in which the nonlocality of the equation under consideration plays a major role. Originally, the notion of antilocality appeared in the context of quantum field theory as the Reeh-Schlieder theorem \cite{V93}, but has recently found various applications in inverse problems \cite{GSU16, GRSU18, GLX17, RS17, RS17a} and control theory \cite{AKW18, BHS17}.

\begin{rmk}
\label{rmk:low_reg}
We remark that a result of the form stated in Proposition ~\ref{prop:antiloc} also holds in a large range of less regular spaces. The only ingredient needed is the presence of a Caffarelli-Silvestre extension at the given regularity (but this holds under very weak assumptions, see Proposition ~\ref{prop:H2gamma} in the constant coefficient setting). The pseudolocality of the associated operators then allows one to locally deduce the vanishing of $u$ from which it is possible to propagate the deduced information to an arbitrary point through the upper half plane (in which the extension problem is formulated).
\end{rmk}

As a property dual to the antilocality of the higher order fractional Laplacian, we further obtain approximation properties for these operators:

\begin{prop}[Runge approximation]
\label{prop:Runge}
Let $\Omega \subset \R^n$ be open and bounded  and let $\widetilde{\Omega} \subset \R^n$ be open with $\Omega \subset \tilde\Omega$ and $B'_1\subset\tilde\Omega\setminus\Omega$.
Let  $v\in H^{\gamma}(\Omega)$ with $\gamma \in \R_+ \setminus \N$. Assume that $L=-\nabla \cdot \tilde{a} \nabla$, where $\tilde{a}$ satisfies the conditions \hyperref[cond:a1]{(A1)}-\hyperref[cond:a3]{(A3)} from above with $\mu =2\floor{\gamma}$.  Suppose that $q\in L^{\infty}(\Omega)$ is such that zero is not a Dirichlet eigenvalue of the operator $L^{\gamma} + q$. Then, for any
 $\epsilon>0$ there exists a solution to 
 \begin{align*}
 &L^{\gamma} u +q u= 0 \mbox{ in } \Omega,\ \supp(u)  \subset \widetilde{\Omega},
 \end{align*}
 such that $\|v-u\|_{L^2(\Omega)}\leq \epsilon$.
\end{prop}

\begin{rmk}
The assumptions on the sets $W$ and $\tilde\Omega\setminus\Omega$ with respect to the unit ball $B'_1=\{x\in\R^n: |x|<1\}$ are taken without loss of generality after normalising the set-up in such a way that the assumptions \hyperref[cond:a2]{(A2)}, \hyperref[cond:a3]{(A3)} on $\tilde a$ are satisfied.
\end{rmk}

These type of approximation properties again crucially exploit the nonlocality of the operator. They were first observed in \cite{DSV14} and generalised to larger classes of equations in \cite{CDV18, CDV18a, DSV16, Krylov18, RS17a, Rue17}. As first highlighted in \cite{GSU16} in the context of nonlocal inverse problems they play an important role in deducing injectivity.
In \cite{RS17} these properties were strengthend to \emph{quantitative} estimates. These were applied in proving stability of the associated inverse problem.

In addition to the rigidity and flexibility properties from Propositions \ref{prop:antiloc} and \ref{prop:Runge}, it is possible to make use of our unique continuation results in many further contexts. As in \cite{FelliFerrero18a, FelliFerrero18b, Rue17quant} one could for instance study the associated problems more quantitatively and derive vanishing order or nodal domain estimates. Using Carleman estimates, one could here proceed similarly as in \cite{KRS16}. Also control theoretic questions similar to for instance \cite{BHS17} could be addressed with our results.
We postpone such a discussion to future work.

\subsection{Organisation of the article}
The remainder of the article is organised as follows: After first recalling a number of auxiliary results (including the generalised Caffarelli-Silvestre extension) in Section ~\ref{sec:aux}, in Section ~\ref{sec:Carl} we then deduce the Carleman estimates which form the basis of our unique continuation results. In Section ~\ref{sec:SUCP} we derive compactness results for the systems which are associated with the SUCP for fractional Schrödinger operators. Here we reduce the SUCP to the weak unique continuation property (WUCP) for the associated systems. This is complemented by a bootstrap argument to derive the WUCP in Section ~\ref{sec:WUCP}. In Section ~\ref{sec:proofs} we combine all the previous results and deduce the statements of Theorems ~\ref{prop:SUCP}-\ref{prop:MUCP} and Propositions ~\ref{prop:antiloc}, ~\ref{prop:Runge}. Finally, in the Appendix, we present a sketch of the derivation of the generalised Caffarelli-Silvestre extension for the higher order fractional Laplacian which had been introduced in \cite{Y13} and which we here discuss at low regularity.

\section{Auxiliary Results}
\label{sec:aux}

In this section we recall several auxiliary results that will be used frequently throughout the text: On the one hand, we recall a higher order Caffarelli-Silvestre extension result. On the other hand, we discuss appropriate boundary-bulk estimates.

\subsection{Notation}

\label{sec:notation}

We summarise the notation that we will use in the sequel: 

\subsubsection{Sets}
Working in $\R^{n+1}_+:=\{x\in \R^{n+1}: \ x_{n+1}\geq 0\}$, we will always use the convention that $x=(x',x_{n+1})$ with $x'\in \R^n$ and $x_{n+1}\geq 0$. For $x_0 \in \R^n \times \{0\}$ we will denote (half) balls in $\R^{n+1}_+$ and $\R^n \times \{0\}$ by
\begin{align*}
B_{r}^+(x_0) = \{x\in \R^{n+1}_+: \ |x-x_0|\leq r\},\
B_{r}'(x_0) = \{x\in \R^n \times \{0\}: \ |x-x_0|\leq r\}. 
\end{align*}
If $x_0 =0$, we will simply write $B_r^+$ and $B_{r}'$.

\subsubsection{Operators}
We will frequently use the operator
\begin{align}
\label{eq:Lb}
L_b:=x_{n+1}^{-b}(\p_{x_{n+1}}x_{n+1}^b \p_{x_{n+1}} - x_{n+1}^b L) ,
\end{align}
where $L=-\nabla'\cdot\tilde a\nabla'$ and $\tilde a$ satisfies the conditions from \hyperref[cond:a1]{(A1)}-\hyperref[cond:a3]{(A3)} with $\mu=2\floor{\gamma}$. 
Usually, we will be thinking of $b= 1-2\floor{\gamma}+2\gamma$, where $\floor{t}=\min\{k\in \N: \ k\leq t\}$.

We define the variable coefficient fractional Laplacian through functional calculus, see Section ~\ref{sec:var_coef}. The set $\Dom(L^{\gamma})$ is then defined as the space in which this functional calculus can be directly applied.

In analogy to our convention for the space variables, we often use the notation
\begin{align*}
\nabla', \ \p_{\ell}', \ \ell \in \{1,\dots,n\},
\end{align*}
to denote the corresponding tangential gradient and partial derivatives.

\subsubsection{Function spaces}

Studying operators of the form \eqref{eq:Lb}, in the sequel, we will often work in function spaces of the form 
\begin{align*}
H^1(\Omega, x_{n+1}^b):= \{u:\Omega \rightarrow \R: \ \|x_{n+1}^{\frac{b}{2}} u\|_{L^2(\Omega)} + \|x_{n+1}^{\frac{b}{2}} \nabla u\|_{L^2(\Omega)} < \infty \},
\end{align*}
where $\Omega \subset \R^{n+1}_+$ is a relatively open set. Similarly, we also deal with the function spaces
\begin{align*}
&\dot{H}^1(\Omega, x_{n+1}^b):= \{u:\Omega \rightarrow \R: \ \|x_{n+1}^{\frac{b}{2}}\nabla u\|_{L^2(\Omega)} <\infty \},\\
&L^2(\Omega, x_{n+1}^b):= \{u:\Omega \rightarrow \R: \ \|x_{n+1}^{\frac{b}{2}} u\|_{L^2(\Omega)} < \infty \}.
\end{align*}
We denote the homogeneous Hölder norms by $[\cdot]_{\dot{C}^{k,\alpha}(\Omega)}$ for $k\in \N\cup\{0\}$, $\alpha \in (0,1)$ and $\Omega \subset \R^n$ or $\Omega \subset \R^{n+1}_+$.

For a sequence $\{a_{k}\}_{k\in \N} \subset \ell^1$, we set
\begin{align*}
\|a_k\|_{\ell^1}:= \sum\limits_{k=1}^{\infty} |a_k|.
\end{align*}

\subsection{The higher order fractional Laplacian}
\label{sec:higher order fractional Laplacian}

We first recall that the higher order fractional Laplacian can be interpreted by means of a Caffarelli-Silvestre-type extension \cite{CS07,CdMG11, Y13,RonSti16}:

\begin{prop}
\label{prop:H2gamma1a}
Let $\gamma>0$, $f\in \Dom(L^{\gamma})$ and $L:= -\nabla' \cdot \tilde{a} \nabla'$, where $\tilde{a}$ satisfies the conditions  \hyperref[cond:a1]{(A1)}-\hyperref[cond:a3]{(A3)} with $\mu = 2\floor{\gamma}$. Then, there exists an extension operator
\begin{align*}
E_{\gamma}: \Dom(L^{\gamma}) \rightarrow C^{2,1}_{loc}(\R^n \times (0,\infty))\cap H^1_{loc}(\R^{n+1}_+, x_{n+1}^b),\
f \mapsto u:=E_{\gamma}(f),
\end{align*}
such that $u=E_{\gamma}(f)$
is a weak solution to the scalar higher order problem
\begin{align}
\label{eq:higher_order_scalar}
\begin{split}
L_b^{\floor{\gamma}+1} u & = 0 \mbox{ in } \R^{n+1}_+,\\
\lim\limits_{x_{n+1}\rightarrow 0} u &= f \mbox{ on } \R^n \times \{0\}, \\
\lim\limits_{x_{n+1}\rightarrow 0}  L_b^k u &= c_{n,\gamma,k} L^k f  \mbox{ on } \R^n \times \{0\} \mbox{ for } k \in \{1,\dots,\floor{\gamma}\},
\\
\lim\limits_{x_{n+1}\rightarrow 0} x_{n+1}^{1-2\gamma + 2\floor{\gamma}} \p_{x_{n+1}} L^{\floor{\gamma}}_b u &= c_{n,\gamma}L^\gamma f \mbox{ on } \R^n \times \{0\},\\
\lim\limits_{x_{n+1}\rightarrow 0} x_{n+1}^{1-2\gamma + 2\floor{\gamma}} \p_{x_{n+1}} L^k_b u &= 0  \mbox{ on } \R^n \times \{0\} \mbox{ for } k \in \{0, \floor{\gamma}-1\}.
\end{split}
\end{align}
Here $L_b:=x_{n+1}^{-b}(\p_{x_{n+1}} x_{n+1}^b \p_{x_{n+1}} + x_{n+1}^{b} \nabla' \cdot \tilde{a}\nabla' ) $ and $b=1-2\floor{\gamma}+2\gamma$. All boundary conditions in \eqref{eq:higher_order_scalar} are attained as $L^2(\R^n)$ limits.

Moreover, setting $u_0:=u$ and $u_{j+1}=L_b u_{j}$ for $j\in\{0,\dots,\floor{\gamma}-1\}$, the functions $u_j$ are in $C^{2,1}(\R^n \times (0,1))\cap H^{1}_{loc}(\R^{n+1}_+, x_{n+1}^b)$. They are weak solutions of the following system of second order equations 
\begin{align}
\label{eq:WUCP_syst1}
\begin{split}
L_b u_{m} & = 0 \mbox{ in } \R^{n+1}_+,\\
L_b u_{j} & = u_{j+1} \mbox{ in } \R^{n+1}_+ \mbox{ for } j \in \{0,\dots,m-1\},\\
\lim\limits_{x_{n+1}\rightarrow 0} u_j & = c_{n,\gamma,j} L^{j}f \mbox{ on } \R^n \times \{0\} \mbox{ for } j \in \{0,\dots,m\},\\
\lim\limits_{x_{n+1}\rightarrow 0} x_{n+1}^b \p_{x_{n+1}} u_{m} & = c_{n,\gamma}L^{\gamma}f \mbox{ on } \R^n \times \{0\},\\
\lim\limits_{x_{n+1}\rightarrow 0} x_{n+1}^{b}\p_{x_{n+1}} u_j & = 0 \mbox{ on } \R^n \times \{0\} \mbox{ for } j \in \{0,\dots,m-1\},
\end{split}
\end{align}
where $m=\floor{\gamma}$.
All boundary conditions hold in an $L^2(\R^n)$ sense.
\end{prop}

In order to keep our presentation self-contained, we provide a short sketch of the proof of this statement  in the Appendix.

Compared to the original nonlocal equations \eqref{eq:frac}, \eqref{eq:frac1}, the equations \eqref{eq:higher_order_scalar}, \eqref{eq:WUCP_syst1} arising from the generalised Cafferelli-Silvestre extension have the advantage that they can be approached with tools from the analysis of unique continuation properties of \emph{local} elliptic equations. As in \cite{Rue15}, \cite{RW18} the price to pay for this localisation is the introduction of the additional dimension in which the solutions $u$ have to be controlled through the corresponding equations. This gives our problem and our argument the flavour of boundary unique continuation results (see for instance \cite{AEK95} and the references therein).

\subsection{Boundary-bulk interpolation estimates}
\label{sec:aux1}

We recall the boundary-bulk interpolation inequality from \cite{Rue15} for fractional Sobolev spaces on the sphere:

\begin{lem}
\label{lem:bdry_bulk}
Let $u:S^{n}_+ \rightarrow \R$ and let $b\in (-1,1)$. Then, identifying $\partial S^{n}_+$ with $S^{n-1}$, there exists a constant $C>0$ such that for any $\tau>1$ 
\begin{align*}
\|u\|_{L^2(S^{n-1})} \leq C(\tau^{\frac{1+b}{2}} \|\theta_n^{\frac b 2} u\|_{L^2(S^n_+)} 
+ \tau^{\frac{b-1}{2}}\|\theta_n^{\frac b 2} \nabla_{S^n} u \|_{L^2(S^n_+)}).
\end{align*}
\end{lem}

\begin{proof}
The proof follows as in \cite{RW18} by using the trace inequality in the associated fractional weighted Sobolev spaces. We discuss the details:

First, by the trace inequality, any function $w\in H^{1}(\R^{n+1}_+, x_{n+1}^{b})$ with $b\in (-1,1)$ satisfies
\begin{align}
\label{eq:trace}
\|w\|_{L^2(\R^n)}
\leq C (\|x_{n+1}^{\frac{b}{2}} w\|_{L^2(\R^{n+1}_+)}  + \|x_{n+1}^{\frac{b}{2}} \nabla w \|_{L^2(\R^{n+1}_+)} ).
\end{align}
Indeed, for $w \in C^{\infty}(\R^{n+1}_+)$ with compact support in the tangential directions this follows from the fundamental theorem of calculus: For $t\in (0,1)$
\begin{align*}
|w(x',0)-w(x',t)| 
&= \left| \int\limits_{0}^{t} \p_{z}w(x',z)dz \right|
\leq \int\limits_{0}^{1} |\p_{z}w(x',z)|dz 
 \leq \int\limits_{0}^{1} z^{-\frac{b}{2}} z^{\frac{b}{2}} |\p_{z}w(x',z)|dz \\
& \leq C_b \left( \int\limits_{0}^{1} z^{b} |\p_{z}w(x',z)|^2 dz \right)^{\frac{1}{2}},
\end{align*}
where we used that $b\in (-1,1)$.
Thus, applying the triangle inequality, integrating in $t\in [0,1]$ and applying the Cauchy-Schartz inequality, we infer
\begin{align*}
|w(x' ,0)|
\leq C_b  \|x_{n+1}^{\frac{b}{2}}\p_{x_{n+1}} w(x',\cdot)\|_{L^2((0,1))} + \|x_{n+1}^{-\frac{b}{2}}x_{n+1}^{\frac{b}{2}} w(x', \cdot)\|_{L^1((0,1))}\\
\leq C_b  \left(\|x_{n+1}^{\frac{b}{2}}\p_{x_{n+1}} w(x', \cdot)\|_{L^2((0,1))} + \|x_{n+1}^{\frac{b}{2}} w(x', \cdot)\|_{L^2((0,1))} \right).
\end{align*}
Taking squares and integrating in $x\in \R^n$ then yields
\begin{align*}
\|w\|_{L^2(\R^n)}^2
&\leq C_b \left( \|x_{n+1}^{\frac{b}{2}} \p_{x_{n+1}}w\|_{L^2(\R^n \times (0,1))}^2 + \|x_{n+1}^{\frac{b}{2}} w\|_{L^2(\R^n \times (0,1))}^2 \right)\\
&\leq C_b \left( \|x_{n+1}^{\frac{b}{2}} \p_{x_{n+1}}w\|_{L^2(\R^{n+1}_+)}^2 + \|x_{n+1}^{\frac{b}{2}} w\|_{L^2(\R^{n+1}_+)}^2 \right).
\end{align*}
By density considerations, this concludes the proof of the trace estimate \eqref{eq:trace}.

Rescaling \eqref{eq:trace}, we then infer 
\begin{align*}
\|w\|_{L^2(\R^n)}
\leq C (\tau^{\frac{1+b}{2}}\|x_{n+1}^{\frac{b}{2}} w\|_{L^2(\R^{n+1}_+)} + \tau^{\frac{b-1}{2}} \|x_{n+1}^{\frac{b}{2}} \nabla w \|_{L^2(\R^{n+1}_+)} ).
\end{align*}
Finally, we apply this to $w(x) = \eta(|x|)u(\frac{x}{|x|})$, where $\eta$ is a smooth, positive cut-off function supported on $B_{2}^+ \setminus B_{1/2}^+$ and which is identically one on $B_{3/2}^+ \setminus B_{3/4}^+$. This yields the desired claim.
\end{proof}

As a direct consequence of Lemma ~\ref{lem:bdry_bulk} we can augment the estimate \eqref{eq:Carl_system_1} in Proposition ~\ref{prop:syst_Carl_1} by corresponding boundary estimates:

\begin{cor}
\label{cor:boundary_bulk}
Assume that the conditions of Proposition ~\ref{prop:syst_Carl_1} hold. Then, we have
\begin{align*}&\sum\limits_{j=0}^{m} \tau^{m-j+\frac{1-b}{2}} \|e^{h(-\ln(|x|))} x_{n+1}^{\frac{b}{2}}(1+\overline{h})^{\frac{m+1}{2}} |x|^{-2m+j+\frac{b-1}{2}} (\nabla')^j \tilde{u}_0 \|_{L^2(B_4')}\\
& \leq C\sum\limits_{j=0}^{m}\sum_{k=0}^{j} \left( \tau^{m+1-j} \|e^{h(-\ln(|x|))} x_{n+1}^{\frac{b}{2}}(1+\overline{h})^{\frac{m+1-k}{2}} |x|^{-2m-1+j+k} (\nabla')^{j-k} \tilde{u}_{k} \|_{L^2(B_4^+)} \right.\\
&   \qquad\qquad\qquad+  \left.  \tau^{m-j} \|e^{h(-\ln(|x|))} x_{n+1}^{\frac{b}{2}} (1+ \overline{h})^{\frac{m+1-k}{2}} |x|^{-2m+j+k} \nabla (\nabla')^{j-k} \tilde{u}_{k} \|_{L^2(B_4^+)}  \right)\end{align*}
Here $\overline{h}(x):= h''(t)|_{t=-\ln(|x|)}$.
\end{cor}

\begin{proof}
Proposition ~\ref{prop:syst_Carl} provides exponentially weighted bulk estimates. In order to deduce the claimed boundary estimates, we apply Lemma ~\ref{lem:bdry_bulk} to the functions 
$$e^{h(-\ln(|x|))}(1+\overline{h})^{\frac{m+1-j}{2}}|x|^{-2m+j+\frac{b-1}{2}}(\nabla')^j\tilde{u}_0(x)$$ on each sphere $|x|=r$. Recalling that $h$ and $\overline{h}$ are independent of the spherical variables and integrating with respect to the radial directions, then implies the claimed boundary estimates. 
\end{proof}

\subsection{Caccioppoli inequality}  
We derive a Caccioppoli type inequality for tangential derivatives of a solution to a variable coefficient equation associated with the operator $L_b$ from above.

\begin{lem}\label{lem:Cacciop}
Let $b\in(-1,1)$ and $a:B_4^+ \rightarrow \R^{(n+1)\times(n+1)}$ be of a block form \eqref{eq:block} with $\tilde a$ satisfying the conditions \hyperref[cond:a1]{(A1)}-\hyperref[cond:a3]{(A3)} with $\mu = 2m$.
Let $u\in H^1(B^+_4, x_{n+1}^b)$ be a solution to 
\begin{align*}
x_{n+1}^{-b}\nabla \cdot a x_{n+1}^b\nabla u&=f \mbox{ in } B^+_4,\\
\lim_{x_{n+1}\to 0} x_{n+1}^b \p_{x_{n+1}}u&= g \mbox{ on } B'_4,
\end{align*}
with $f, |\nabla' f|, \cdots, |(\nabla')^k f| \in L^2(B_4^+, x_{n+1}^b)$ and $g\in H^k(B_4')$.
Then there exists a constant $C>0$ such that for any $J\in \N\cup \{0\}$ with $ J\leq\min\{2m, k\}$
\begin{align}
\label{eq:Cacciop}
\begin{split}
\sum_{j=0}^J r^{j} \| x_{n+1}^{\frac{b}{2}} \nabla(\nabla')^{j} {u}\|_{L^2(B^+_{r/2})}
&\leq C\sum\limits_{j=0}^{J} \Big(r^{j-1} \| x_{n+1}^{\frac{b}{2}} (\nabla')^{j} {u} \|_{L^2(B^+_{r})}+ r^{j+1} \| x_{n+1}^{\frac{b}{2}} (\nabla')^{j} {f} \|_{L^2(B^+_{r})}\\
&\qquad \qquad +r^j\Big(\int\limits_{B'_r} | (\nabla')^jg||(\nabla')^ju| dx'\Big)^{\frac 1 2}\Big).
\end{split}
\end{align}
\end{lem}

\begin{proof}
First of all we note that by scaling it suffices to prove the estimate for $r=1$.
Next, we observe that by the block structure of the metric $a$ for any $j\in\{1,\dots,2m\}$
\begin{align*}
x_{n+1}^{-b}\nabla \cdot a x_{n+1}^b\nabla (\nabla')^ju= (\nabla')^j \big(x_{n+1}^{-b}\nabla \cdot a x_{n+1}^b\nabla  u\big)-\sum_{k=0}^{j-1}\nabla'\cdot((\nabla')^{j-k}\tilde a)\nabla' (\nabla')^k u,
\end{align*}
and hence  $(\nabla')^j u$ is a weak solution to
\begin{align*}
L_b (\nabla')^ju&= (\nabla')^j f-\sum_{k=0}^{j-1}\nabla'\cdot((\nabla')^{j-k}\tilde a)\nabla' (\nabla')^k u
\mbox{ in } B^+_4,\\
\lim_{x_{n+1}\to 0} x_{n+1}^b \p_{x_{n+1}}(\nabla')^ju&=(\nabla')^j g \mbox{ on } B'_4.
\end{align*}
This means that for any test function $\varphi$ it holds
\begin{align*}
\int_{B_1^+} x_{n+1}^b  \nabla \varphi \cdot a \nabla (\nabla')^j u dx=&-\int_{B_1^+} x_{n+1}^b \varphi (\nabla')^j fdx -\sum_{k=0}^{j-1}\int_{B_1^+} x_{n+1}^b \nabla'\varphi\cdot((\nabla')^{j-k}\tilde a)\nabla' (\nabla')^k u dx\\
&+\int_{B'_1} \varphi (\nabla')^j g dx'.
\end{align*}

Let $\eta$ be a radial cut-off function equal to one on $B^+_{1/2}$ which vanishes outside of $B^+_{1}$ and satisfies $|\nabla \eta|\leq C$. 
We remark that the function $\varphi=\eta^2 (\nabla')^j u$ is an admissible test function, as the equation may be differentiated in tangential directions and leads to corresponding tangential regularity estimates.
Using that by condition \hyperref[cond:a2]{(A2)} it holds $|(\nabla')^\alpha\tilde a(x)| \leq C_{\alpha} $ for $x\in B_{4}^+$ and $|\alpha|\leq 2m$, we obtain the following estimate
\begin{align*}
\|x_{n+1}^{\frac b 2} \eta \nabla(\nabla')^ju\|^2_{L^2(B_1^+)}\leq C \Big(&
\int_{B_1^+} x_{n+1}^b \eta |\nabla (\nabla')^j u||\nabla \eta|  |(\nabla')^j u| dx +\int_{B_1^+} x_{n+1}^b \eta^2  |(\nabla')^j u||(\nabla')^j f|dx\\
&+\sum_{k=0}^{j-1}\int_{B_1^+} x_{n+1}^b \big(\eta^2 |\nabla (\nabla')^j u|+\eta|\nabla \eta||(\nabla')^ju|\big)  | (\nabla')^{k+1} u|dx\\
&+\int_{B'_1} \eta^2 |(\nabla')^j g| |(\nabla')^ju| dx'\Big).
\end{align*}
By virtue of Young's inequality and absorbing terms into the left hand side we infer
\begin{align*}
\|x_{n+1}^{\frac b 2} \eta \nabla(\nabla')^ju\|_{L^2(B_1^+)}
\leq
 C\Big(& \|x_{n+1}^{\frac b 2} |\nabla \eta| (\nabla')^ju\|_{L^2(B_1^+)}+\|x_{n+1}^{\frac b 2} \eta (\nabla')^ju\|_{L^2(B_1^+)}\\
 &+ \|x_{n+1}^{\frac b 2} \eta (\nabla')^jf\| _{L^2(B_1^+)}+ \sum_{k=0}^{j-1}\|x_{n+1}^{\frac b 2} \eta (\nabla')^{k+1}u\|_{L^2(B_1^+)}
\\
&+\int_{B'_1} \eta^2 |(\nabla')^j g| |(\nabla')^ju| dx'\Big).
\end{align*}
Lastly, we use the bounds of $\eta$ to obtain
\begin{align*}
\|x_{n+1}^{\frac b 2}\nabla(\nabla')^ju\|_{L^2(B_{1/2}^+)}\leq
 C\Big(
 &\sum_{k=1}^j\|x_{n+1}^{\frac b 2} (\nabla')^ku\|_{L^2(B_1^+)}+ \|x_{n+1}^{\frac b 2}  (\nabla')^jf\|_{L^2(B_1^+)}\\
&+\Big(\int_{B'_1}  |(\nabla')^j g| |(\nabla')^ju| dx'\Big)^{\frac 1 2}
\Big).
\end{align*}

The estimate for $j=0$ is straightforward.
Summing over $j\in \{0,\dots, J\}$ with the corresponding factors, yields the desired estimate.
\end{proof}

\section{Carleman Estimates for Systems in the Upper Half-Plane}
\label{sec:Carl}

In order to deduce the Carleman estimate from Propositions ~\ref{prop:syst_Carl} and ~\ref{prop:syst_Carl_1}, we first prove Carleman estimates for second order (degenerate elliptic) equations. 
Here we proceed in two steps: First, we discuss the situation for constant coefficient metrics but in the presence of divergence form right hand side contributions, and then, in a second step, we deduce the estimates for variable coefficient metrics.

\subsection{Constant coefficient Carleman estimates}

As a main ingredient in our argument we make use of the following (constant coefficient) Carleman estimate:

\begin{prop}
\label{prop:Carl}
Let $b \in (-1,1)$ and let $u\in H^{1}(B_4^+, x_{n+1}^b)$ with $\supp(u)\subset B_4^+ \setminus \{0\}$ be a solution to 
\begin{align*}
\nabla \cdot x_{n+1}^{b} \nabla u & = f + \sum\limits_{j=1}^n \p_j x_{n+1}^{b} F^j \mbox{ in } B_4^+,\\
\lim\limits_{x_{n+1}\rightarrow 0} x_{n+1}^{b} \p_{n+1} u & = g \mbox{ on } B_4',
\end{align*}
where $f \in L^{2}(B_4^+, x_{n+1}^{-b})$, $g\in L^2(B_4')$ and $F = (F^1,\dots, F^n) \in L^2(B_4^+, x_{n+1}^b)^n$ with $\supp(f)$,
$\supp(F) \subset B_4^+ \setminus \{0\}$ and $\supp(g)\subset B_4' \setminus \{0\}$.
Then, for each $\tau>\tau_0\gg 1$ there is a weight function $h(-\ln(|x|))$ such that there exists a constant $C>0$ which is independent of $\tau$ such that it holds
\begin{align*}
&\tau
 \|e^{h(-\ln(|x|))}  x_{n+1}^{\frac{b}{2}} (1+\overline{h})^{\frac{1}{2}} |x|^{-1} u\|_{L^2(\R^{n+1}_+)}
+  \|e^{h(-\ln(|x|))}x_{n+1}^{\frac{b}{2}} (1+\overline{h})^{\frac{1}{2}}  \nabla u\|_{L^2(\R^{n+1}_+)}\\
& +\tau^{\frac{1-b}{2}}\|e^{h(-\ln(|x|))}(1+\overline{h})^{\frac 1 2}|x|^{\frac{b-1}2}u\|_{L^2(\R^n\times\{0\})}
\\
&\leq C \left( \| e^{h(-\ln(|x|))} |x| x_{n+1}^{-\frac{b}{2}} f\|_{L^2(\R^{n+1}_+)} 
+ \tau \| e^{h(-\ln(|x|))}  x_{n+1}^{\frac{b}{2}} F\|_{L^2(\R^{n+1}_+)} 
\right.\\
&\qquad \quad \left. + \tau^{\frac{1+b}{2}} \|e^{h(-\ln(|x|))} |x|^{\frac{1-b}{2}} g\|_{L^2(\R^n \times \{0\})} \right).
\end{align*}
Here $\overline{h}(x):=  h''(t)|_{t=-\ln(|x|)}$.
\end{prop}

The proof of Proposition ~\ref{prop:Carl} relies on a splitting strategy, in which all inhomogeneities (be they bulk or boundary contributions) are dealt with in an elliptic estimate. As a consequence, the subelliptic part, i.e. the actual Carleman estimate, becomes rather clean. 
In particular, as shown in the following section, the estimates are strong enough to deal with only Lipschitz regular (small) metrics in a perturbative way.

\begin{proof}
We proceed in three main steps: First, we construct an appropriate Carleman weight. Then, we deduce the desired Carleman estimate by a splitting argument in conformal polar coordinates. In a final step, we concatenate the obtained information.\\

\emph{Step 1: Construction of the weight.} We begin by constructing a family of Carleman weights $h(t):\R \rightarrow \R$. Anticipating the use of polar conformal coordinates, we require it to satisfy
\begin{align*}
&h' \in (C^{-1} \tau, C\tau) \mbox{ for all } \tau \gg 1, \\
& \frac{1}{4} \leq h'' + \dist(h', \spec(\nabla_{S^n}\cdot \theta_n^{a} \nabla_{S^n})),\\
&|h'''|, |h^{(4)}| \leq C (1+h'') \leq C \tau.
\end{align*}
Here the constant $C>1$ is independent of $\tau$ and we recall that by the results of \cite{KRSIV} the operator $\nabla_{S^n} \cdot \theta_n^{a} \nabla_{S^n}$ with $\theta_n:= \frac{x_{n+1}}{|x|}$ has a spectral gap if it is considered with Neumann (or Dirichlet) data (see Section 8.3 in the Appendix A in \cite{KRSIV}).
We follow the argument from \cite{KT01} and \cite{CK10} to obtain the desired properties for the Carleman weight. To this end, we consider a sequence $\{c_j \}_{j\in \N} \in \ell^1$, $\|c_j\|_{\ell^1}< \delta$ and define the sequence $\{a_j\}_{j\in \N}$ as the convolution of $c_j$ with $2^{-\nu j}$ for some $\nu>0$ small.
Then, the sequence $a_j$ is slowly varying (i.e., $2^{-\nu} a_{j+1}\leq a_j \leq 2^{\nu} a_{j+1}$) and obeys the bound $c_j \leq a_j$. With this preparation, we define $h(0)=0$, $h(-\infty) = \floor{\tau} + \frac{5}{4}$ and 
\begin{align*}
h'' = \sum\limits_{j} b_j \chi_{[j,j+1]},
\end{align*}
with 
\begin{align*}
b_j 
\left\{
\begin{array}{ll}
= 0 &\mbox{ if } a_j \leq \nu \tau^{-1},\\
\in [\tau a_j, 2 \tau a_j] &\mbox{ if } a_j \geq \nu \tau^{-1}.
\end{array} \right.
\end{align*}
In order to also obtain the desired regularity properties for $h$, we regularise this by convolution (on the scale one).\\

\emph{Step 2: Conformal polar coordinates and splitting argument.}
We proceed by a splitting argument. In order to obtain more transparent expressions, we pass to conformal coordinates by setting $t = -\ln(|x|)$, $\theta = \frac{x}{|x|}$. We further pass from the function $u$ to the function $\tilde{u}(t,\theta) = e^{\frac{1-b-n}{2}t} u(e^{-t}\theta)$. In these coordinates
we consider (the weak form) of the equation
\begin{align}
\label{eq:conf_coord}
\begin{split}
(\theta_n^{b}\p_t^2 + \nabla_{S^n} \cdot \theta_n^{b} \nabla_{S^n} + c_{n,b} \theta_n^{b}) \tilde{u} & = \tilde{f} - \theta_n^b \p_t \tilde{F}^t  + \di_{S^n_+} \theta_n^{b} \tilde{F}' \mbox{ in } S^n_+ \times \R,\\
\lim\limits_{\theta_n \rightarrow 0} \theta_n^{b} \nu \cdot \nabla_{S^n} \tilde{u} & = \tilde{g} \mbox{ on } \partial S^n_+ \times \R,
\end{split}
\end{align}
where 
\begin{align*}
\tilde{f}(t,\theta) &= e^{-\frac{n+3-b}{2}t}f(e^{-t}\theta) + \frac{n+b-1}{2}\theta_n^b \tilde F^t(t,\theta),\\
 \tilde{g}(t,\theta)&=e^{(-1+b)t} e^{\frac{1-b-n}{2}t} g(e^{-t} \theta),\\
 \tilde{F}^t(t, \theta) & =  e^{-\frac{n+1+b}{2}t}\left(\sum\limits_{i=2}^{n} \theta_{i-1}F^i (e^{-t}\theta)  \pm \sqrt{\left(1- \sum\limits_{i=1}^{n}\theta_i^2\right)} F^1(e^{-t}\theta) \right), \\  
 \tilde{F}^j(t,\theta) &= e^{-\frac{n+1+b}{2}t}\sum\limits_{i=1}^{n}\delta_{i,j+1}F^i(e^{-t}\theta)-\theta_{j}  \tilde F^t(t,\theta), \ j \in\{1,\dots,n\}, 
 \\ \tilde{F}'(t,\theta) &= (\tilde{F}^1(t,\theta),\dots,\tilde{F}^n(t,\theta)),\\
 c_{n,b}&=-\left(\frac{n+b-1}{2}\right)^2, \ \theta_j = \frac{x_{j+1}}{|x|}, \ j\in\{1,\dots,n\}. 
\end{align*}
Here $\di_{S^n_+}$ denotes the divergence with respect to the standard metric on the sphere and the choice of the sign in the expression for $\tilde{F}^t(t,\theta)$ depends on the specific chart.

We split the problem into two parts $\tilde{u}=u_1+u_2$, where $u_1$ is a solution to the following elliptic problem
\begin{align}
\label{eq:u1}
\begin{split}
(\theta_n^{b}\p_t^2 + \nabla_{S^n}\cdot \theta_n^{b} \nabla_{S^n} + \theta_n^{b} c_{n,b}-\theta_n^{b} \tau^2 K^2)u_1 & = \tilde{f} + \theta_n^b \p_t \tilde{F}^1 +\di_{S^n_+} \theta_n^{b} \tilde{F}^i \mbox{ in } S^n_+ \times \R,\\
\lim\limits_{\theta_n \rightarrow 0} \theta_n^{b} \nu \cdot \nabla_{S^n} u_1 & = \tilde{g} \mbox{ on } \partial S^n \times \R.
\end{split}
\end{align}
Here $K\gg 1$ is a sufficiently large parameter (to be specified later).
The function $u_2$ thus solves a corresponding problem. In order to derive the desired estimate, we discuss the bounds for $u_1$ and $u_2$ separately.\\

\emph{Step 2a: Bounds for $u_1$.} By virtue of the positivity of $K\gg 1$, the estimates for $u_1$ are elliptic energy estimates. Indeed, by the Lax-Milgram theorem, we obtain that a solution to \eqref{eq:u1} exists in the energy space $H^{1}(S^n_+ \times \R, \theta_n^{b})$. We test the weak form of \eqref{eq:u1} with the test function $\tau^2 e^{2h_{M,\delta}(t)} u_1$, where
\begin{align*}
h_{M,\delta}(t):= \min\{M, \max\{h(t),-M \} \}\ast \eta_{\delta}(t),
\end{align*}
for $M \in \N$ and with $\eta_{\delta}$ denoting a standard mollifier.

This leads to the following identity
\begin{align}
\label{eq:test}
\begin{split}
&\tau^2(\theta_n^{b} \p_t u_1, e^{2 h_{M,\delta}(t)}\p_t u_1) + 2\tau^2 (\theta_n^{b} \p_t u_1 , h_{M,\delta}' e^{2 h_{M,\delta}(t)} u_1)
+ \tau^2(\theta_n^{b} e^{2 h_{M,\delta}(t)}  \nabla_{S^n}u_1, \nabla_{S^n} u_1) \\
& - (c_{n,a} \tau^2-K^2\tau^4) (\theta_n^{b}e^{2h_{M,\delta}(t)} u_1,u_1)\\
&= -\tau^2 (\tilde{f}, e^{2h_{M,\delta}(t)}u_1) -\tau^2 (\p_i \theta_n^{b} \tilde{F}^i, e^{2h_{M,\delta}(t)}u_1) + \tau^2 (\tilde{g}, e^{2 h_{M,\delta}(t)}u_1)_0.
\end{split}
\end{align}
Here $(\cdot,\cdot):=(\cdot,\cdot)_{L^2(S^n_+ \times \R)}$ and $(\cdot,\cdot)_0:=(\cdot,\cdot)_{L^2(\partial S^n_+ \times \R)}$.
Estimating (by using the properties of $h_{M,\delta}$)
\begin{align*}
2\tau^2 |(\theta_n^{b} \p_t u_1 , h_{M,\delta}' e^{2 h_{M,\delta}(t)} u_1)|
&\leq \frac{1}{4}\tau^2 |(\theta_n^{b} \p_t u_1, e^{2 h_{M,\delta}(t)} \p_t u_1)| + 16 \tau^4 |(\theta_n^{b}u_1, e^{2 h_{M,\delta}(t)}u_1)|,\\
\tau^2 |(\tilde{f}, e^{2 h_{M,\delta}(t)} u_1)|
&\leq \frac{1}{4}\tau^4|(\theta_n^{b} u_1, e^{2 h_{M,\delta}(t)} u_1)| + C |(\theta_n^{-b} \tilde{f}, e^{2 h_{M,\delta}(t)}\tilde{f})|,\\
\tau^2 |(\p_i \theta_n^{b} \tilde{F}^i, e^{2 h_{M,\delta}(t)} u_1)|
&\leq \frac{1}{4} \tau^2|(e^{2h_{M,\delta}(t)} \theta_n^{b} \nabla u_1, \nabla u_1)| + C \tau^2|(\theta_n^{b} \tilde{F}^i, e^{2h_{M,\delta}(t)}\tilde{F}^i)|,
\end{align*}
and absorbing the contributions involving $u_1$ as well as the other non-positive terms from \eqref{eq:test} into either the positive derivative contributions or (for $K^2\gg 1$ sufficiently large) into the coercive term involving $K^2$ in \eqref{eq:test}, then yields
\begin{align*}
&\tau \|\theta_n^{\frac{b}{2}} e^{h_{M,\delta}(t)}\p_t u_1\| + \tau \|\theta_n^{\frac{b}{2}} e^{h_{M,\delta}(t)} \nabla_{S^n} u_1\|
+ \frac{K}{2} \tau^2 \|\theta_n^{\frac{b}{2}} e^{h_{M,\delta}(t)} u_1\|\\
&\leq C \left( \|\theta_n^{-\frac{b}{2}} e^{h_{M,\delta}(t)} \tilde{f}\| + \tau \|\theta_n^{\frac{b}{2}}e^{h_{M,\delta}(t)}\tilde{F}\| + \epsilon \tau^{\frac{3-b}{2}} \|e^{h_{M,\delta}(t)} u_1\|_0 + C_{\epsilon} \tau^{\frac{1+b}{2}} \|e^{h_{M,\delta}(t)} \tilde{g}\|_0 \right).
\end{align*}
As above, here and in the sequel, we use the notation $\|\cdot \|:= \|\cdot \|_{L^2(S^n_+ \times \R)}$ and $\|\cdot \|_0 := \|\cdot \|_{L^2(\partial S^n_+ \times \R)}$.
Applying the boundary-bulk interpolation estimate from Lemma ~\ref{lem:bdry_bulk}, allows us to absorb the first boundary contribution into the left hand side, which then results in
\begin{align}
\label{eq:u_1_est_a}
\begin{split}
&\tau \|\theta_n^{\frac{b}{2}} e^{h_{M,\delta}(t)}\p_t u_1\| + \tau \|\theta_n^{\frac{b}{2}} e^{h_{M,\delta}(t)} \nabla_{S^n} u_1\|
+ \frac{K}{2} \tau^2 \|\theta_n^{\frac{b}{2}} e^{h_{M,\delta}(t)} u_1\|\\
&\leq C \left( \|\theta_n^{-\frac{b}{2}} e^{h_{M,\delta}(t)} \tilde{f}\| +\tau \|\theta_n^{\frac{b}{2}} e^{h_{M,\delta}(t)} \tilde{F}\| + C_{\epsilon} \tau^{\frac{1+b}{2}} \|e^{h_{M,\delta}(t)} \tilde{g}\|_0 \right).
\end{split}
\end{align}
Using the compact support of $\tilde{f}$, $\tilde{F}$ and $\tilde{g}$, by dominated convergence, we may pass to the limits $M \rightarrow \infty$ and $\delta \rightarrow 0$ which leads to 
\begin{align}
\label{eq:u_1_est}
\begin{split}
&\tau \|\theta_n^{\frac{b}{2}} e^{h(t)}\p_t u_1\| + \tau \|\theta_n^{\frac{b}{2}} e^{h(t)} \nabla_{S^n} u_1\|
+ \frac{K}{2} \tau^2 \|\theta_n^{\frac{b}{2}} e^{h(t)} u_1\|\\
&\leq C \left( \|\theta_n^{-\frac{b}{2}} e^{h(t)} \tilde{f}\| +\tau \|\theta_n^{\frac{b}{2}} e^{h(t)} \tilde{F}\| + C_{\epsilon} \tau^{\frac{1+b}{2}} \|e^{h(t)} \tilde{g}\|_0 \right).
\end{split}
\end{align}
We remark that this estimate not only contains the right weighted bounds for $u_1$ but also implies that $u_1$ has (quantitative) fast decay as $t \rightarrow \infty$ (which corresponds to $|x|\rightarrow 0$). By the compact support assumption on $\tilde{u}$ a similarly fast decay then also holds for $u_2$. \\

\emph{Step 2b: Bounds for $u_2$.}
The estimates for $u_2$ will be sub-elliptic (in $\tau$) Carleman estimates. We recall that by construction, $u_2$ is a weak solution of 
\begin{align}
\label{eq:u_2}
\begin{split}
(\theta_n^{b}\p_t^2 + \nabla_{S^n}\cdot \theta_n^{b} \nabla_{S^n} + \theta_n^{b} c_{n,b}) u_2 
& = - K^2 \tau^2 \theta_n^{b} u_1 \mbox{ in } S^n \times \R,\\
\lim\limits_{\theta_n \rightarrow 0} \theta_n^{b} \nu \cdot \nabla_{S^n}u_2 & = 0 \mbox{ on } \partial S^n \times \R.
\end{split}
\end{align}
We test this with a Neumann eigenfunction to the spherical operator, i.e. with a function $\psi_{\lambda}$ which satisfies
\begin{align*}
\nabla_{S^n} \cdot \theta_n^{b} \nabla_{S^n} \psi_{\lambda} & = - \lambda^2 \theta_n^{b} \psi_{\lambda} \mbox{ in } S^n_+,\\
\lim\limits_{\theta_n \rightarrow 0} \theta_n^{b} \nu \cdot \nabla_{S^n} \psi_{\lambda} & = 0 \mbox{ on } \partial S^n_+ \times \R.
\end{align*}
Since the set $\{\psi_{\lambda}\}$ forms an orthonormal set in $H^{1}(S^n_+, \theta_n^b)$, we may expand the function $u_2$ into this basis. As a result, we obtain an equation for each individual mode, i.e. for $\alpha_{\lambda}(t):=(u_2, \theta_n^{b} \psi_{\lambda})$ and $\beta_{\lambda}(t) = (u_1,\theta_n^{b}\psi_{\lambda})$, we obtain the mode-wise equation
\begin{align}
\label{eq:modes}
\begin{split}
\alpha_{\lambda}'' - \lambda^2 \alpha_{\lambda} + c_{n,b} \alpha_{\lambda} & = -K^2 \tau^2 \beta_{\lambda}.
\end{split}
\end{align}
Conjugating this with the weight $e^{h(t)}$ yields the equation
\begin{align*}
\tilde{\alpha}_{\lambda}'' - \lambda^2 \tilde{\alpha}_{\lambda}
+  |h'|^2 \tilde{\alpha}_{\lambda}
+c_{n,b}\tilde\alpha_\lambda
- 2 h' \tilde{\alpha}_{\lambda}' -  h'' \tilde{\alpha}_{\lambda} = K^2 \tau^2 \tilde{\beta}_{\lambda},
\end{align*}
where $\tilde{\alpha}_{\lambda} = e^{h(t)} \alpha_{\lambda}$ and 
$\tilde{\beta}_{\lambda} = e^{h(t)} \beta_{\lambda}$. Noting that the symmetric and antisymmetric parts of the conjugated operator turn into
\begin{align*}
S \tilde{\alpha}_\lambda &= \tilde{\alpha}_{\lambda}'' - \lambda^2 \tilde{\alpha}_{\lambda} +  |h'|^2 \tilde{\alpha}_{\lambda}  + c_{n,b} \tilde{\alpha}_{\lambda}, \\
A \tilde{\alpha}_\lambda&= -2  h' \tilde{\alpha}_{\lambda}' - h'' \tilde{\alpha}_{\lambda},
\end{align*}
we expand the conjugated operator to infer
\begin{align}
\label{eq:bound_u2}
\begin{split}
K^4 \tau^4 \|\tilde{\beta}_{\lambda}\|_{L^2(\R)}^2 = \|S \tilde{\alpha}_{\lambda}\|_{L^2(\R)}^2 + \|A \tilde{\alpha}_{\lambda}\|_{L^2(\R)}^2 + ([S,A] \tilde{\alpha}_{\lambda}, \tilde{\alpha}_{\lambda})_{L^2(\R)},
\end{split}
\end{align}
where 
\begin{align*}
([S,A]\tilde{\alpha}_{\lambda},\tilde{\alpha}_{\lambda}) = 4 \int\limits_{\R} h' h'' h' \tilde{\alpha}_{\lambda}^2 + h'' (\tilde{\alpha}_{\lambda}')^2 dt - \int\limits_{\R} h^{(4)} \tilde{\alpha}_{\lambda}^2 dt.
\end{align*}
We observe that the first two contributions in the expansion of the commutator are non-negative. The last term which does not necessarily carry a sign can be absorbed into these positive contributions and can hence be neglected for $\tau \geq \tau_0>1$ sufficiently large. Noting that the spectral gap of the Neumann data version of the operator $\nabla_{S^n}\cdot \theta_n^{b} \nabla_{S^n}$ (see \cite{KRSIV}) yields that 
\begin{align*}
\|S \tilde{\alpha}_{\lambda}\|_{L^2(\R)}^2
\geq c\dist(h', \spec(\nabla_{S^n}\cdot \theta_n^{b} \nabla_{S^n})) \|\max\{|h'|, |\lambda|\} \tilde{\alpha}_{\lambda}\|_{L^2(\R)}^2,
\end{align*}
and using the antisymmetric part of the operator to deduce a bound on the gradient, then turns \eqref{eq:bound_u2} into
\begin{align}
\label{eq:bound_u2_2a}
\begin{split}
K^4 \tau^4 \|\tilde{\beta}_{\lambda}\|_{L^2(\R)}^2 
&= \|S \tilde{\alpha}_{\lambda}\|_{L^2(\R)}^2 + \|A \tilde{\alpha}_{\lambda}\|_{L^2(\R)}^2  + ([S,A]\tilde{\alpha}_{\lambda}, \tilde{\alpha}_{\lambda})_{L^2(\R)}\\
&\geq c \|\max\{|h'|, |\lambda|\} \tilde{\alpha}_{\lambda}\|_{L^2(\R)}^2 + \|h' (h'')^{1/2} \tilde{\alpha}_{\lambda}\|_{L^2(\R)}^2 + \|(h'')^{1/2} \tilde{\alpha}_{\lambda}'\|_{L^2(\R)}^2\\
& \quad + \tau^{-2}\|h' \tilde{\alpha}_{\lambda}'\|_{L^2(\R)}^2.
\end{split}
\end{align}
We remark that we have given up a factor $\tau^2$ in the antisymmetric estimate. This is due to the fact, that in undoing the conjugation with the weight $e^{h(t)}$, we obtain a term originating from the $t$ derivative falling onto the weight. Without the loss of the factor $\tau^2$ this would carry a weight $\tau^4$. We would not be able to absorb this into the $L^2$ contributions on the left hand side of the estimates.

We further complement the estimate \eqref{eq:bound_u2_2a} by a bound on the spherical part of the gradient. To this end, we make use of the symmetric part of the operator. Indeed, we have
\begin{align*}
(S \tilde{\alpha}_{\lambda}, h'' \tilde{\alpha}_{\lambda})
&= - (\tilde{\alpha}_{\lambda}', h'' \tilde{\alpha}_{\lambda}')
+ \frac{1}{2}(\tilde{\alpha}_{\lambda}, h''' \tilde{\alpha}_{\lambda})
- \lambda^2 (\tilde{\alpha}_{\lambda}, h'' \tilde{\alpha}_{\lambda}) \\
& \quad + (|h'|^2 h'' \tilde{\alpha}_{\lambda}, \tilde{\alpha}_{\lambda})
+ c_{n,b}(\tilde{\alpha}_{\lambda}, h'' \tilde{\alpha}_{\lambda}).
\end{align*}
As a consequence, if $c_0>0$ is sufficiently small,
\begin{align*}
c_0 \lambda^2 (\tilde{\alpha}_{\lambda}, h'' \tilde{\alpha}_{\lambda})
&\leq c_0 |(S \tilde{\alpha}_{\lambda}, h'' \tilde{\alpha}_{\lambda})|
+ c_0|(\tilde{\alpha}_{\lambda}', h'' \tilde{\alpha}_{\lambda}')| + \frac{c_0}{2}|(\tilde{\alpha}_{\lambda}, h''' \tilde{\alpha}_{\lambda})| + c_0|(|h'|^2 h'' \tilde{\alpha}_{\lambda}, \tilde{\alpha}_{\lambda})| \\
& \quad + c_0 c_{n,b}|(\tilde{\alpha}_{\lambda},h'' \tilde{\alpha}_{\lambda})|\\
& \leq \frac{c_0^2}{2} \|S \alpha_{\lambda}\|^2 + \frac{c_0^2}{2} \|h'' \tilde{\alpha}_{\lambda}\|^2 + c_0|(\tilde{\alpha}_{\lambda}', h'' \tilde{\alpha}_{\lambda}')| 
+ \frac{c_0}{2} |(\tilde{\alpha}_{\lambda}, h''' \tilde{\alpha}_{\lambda})| \\
& \quad + c_0 C \tau^2 \|(h'')^{\frac{1}{2}} \tilde{\alpha}_{\lambda}\|^2\\
& \leq K^4 \tau^4 \|\tilde{\beta}_{\lambda}\|^2.
\end{align*}
Here the last estimate follows from the previously deduced bounds from \eqref{eq:bound_u2_2a}. Hence, we conclude
\begin{align}
\label{eq:bound_u2_2}
\begin{split}
K^4 \tau^4 \|\tilde{\beta}_{\lambda}\|_{L^2(\R)}^2 
&= \|S \tilde{\alpha}_{\lambda}\|_{L^2(\R)}^2 + \|A \tilde{\alpha}_{\lambda}\|_{L^2(\R)}^2  + ([S,A]\tilde{\alpha}_{\lambda}, \tilde{\alpha}_{\lambda})_{L^2(\R)}\\
&\geq c\|\max\{|h'|, |\lambda|\} \tilde{\alpha}_{\lambda}\|_{L^2(\R)}^2 + \|h' (h'')^{1/2} \tilde{\alpha}_{\lambda}\|_{L^2(\R)}^2 + \|(h'')^{1/2} \tilde{\alpha}_{\lambda}'\|_{L^2(\R)}^2\\
& \quad + \tau^{-2}\|h' \tilde{\alpha}_{\lambda}'\|_{L^2(\R)}^2 + c_0 \lambda^2 \|(h'')^{1/2}\tilde{\alpha}_{\lambda}\|_{L^2(\R)}^2.
\end{split}
\end{align}

By orthogonality, summing the estimate \eqref{eq:bound_u2_2} over $\lambda$, integrating over $S^n_+$, using the properties of $h$ and undoing the conjugation, we thus obtain
\begin{align}
\label{eq:bound_u2_3}
\tau \|e^{h} (1+h'')^{1/2} \theta_n^{\frac{b}{2}} u_{2}\|
+  \|e^{h} (1+h'')^{1/2} \theta_n^{\frac{b}{2}} \nabla u_2 \|
\leq K^2 \tau^2 \|\theta_n^{\frac{b}{2}} e^{h} u_1\|.
\end{align}

\emph{Step 3: Conclusion.} Last but not least, we combine the estimates from Steps 1 and 2 and deduce the Carleman estimate from Proposition ~\ref{prop:Carl} from this. We obtain
\begin{align*}
&\tau \|e^{h} (1+ h'')^{\frac{1}{2}} \theta_n^{\frac{b}{2}}\tilde{u}\| 
+ \|e^{h} (1+ h'')^{\frac{1}{2}} \theta_n^{\frac{b}{2}}\nabla \tilde{u}\|\\
&\leq 
\tau \|e^{h} (1+ h'')^{\frac{1}{2}} \theta_n^{\frac{b}{2}}u_1\| 
+ \|e^{h} (1+ h'')^{\frac{1}{2}} \theta_n^{\frac{b}{2}}\nabla u_1\|
+ \tau \|e^{h} (1+ h'')^{\frac{1}{2}} \theta_n^{\frac{b}{2}}u_2\| 
+ \|e^{h} (1+ h'')^{\frac{1}{2}} \theta_n^{\frac{b}{2}}\nabla u_2\|\\
& \leq C(\|e^{h}\theta_n^{-\frac{b}{2}}  \tilde{f}\| + \|e^{h}\theta_n^{\frac{b}{2}}  \tilde{F}\| + \tau^{\frac{1+b}{2}} \|e^{h} \tilde{g}\|_0).
\end{align*}
Using the bulk-boundary interpolation estimate from Lemma ~\ref{lem:bdry_bulk}, this can further be strengthened by a boundary contribution on the left hand side:
\begin{align*}
&\tau^{\frac{1-b}{2}}\|e^{h} (1+h'')^{\frac{1}{2}} \tilde u\|_0+ 
\tau \|e^{h} \theta_n^{\frac{b}{2}}(1+h'')^{\frac{1}{2}}\tilde{u}\| + \|e^{h} \theta_n^{\frac{b}{2}}(1+h'')^{\frac{1}{2}} \nabla \tilde{u}\|\\
& \leq C(\|e^{h}\theta_n^{-\frac{b}{2}}  \tilde{f}\| + \|e^{h}\theta_n^{\frac{b}{2}}  \tilde{F}\|+ \tau^{\frac{1+b}{2}} \|e^{h} \tilde{g}\|_0).
\end{align*}
Transforming back into Cartesian coordinates yields the desired estimate.
\end{proof}

\subsection{Variable coefficient metrics}
Considering second order equations of the form
\begin{align}
\label{eq:var_coef}
\begin{split}
\nabla\cdot x_{n+1}^ba\nabla u & = f \mbox{ in } B_4^+,\\
\lim\limits_{x_{n+1}\rightarrow 0} x_{n+1}^{b} \p_{n+1} u & = g \mbox{ on } B_4',
\end{split}
\end{align}
in the sequel, we seek to introduce \emph{variable coefficients} in the Carleman estimate of Proposition ~\ref{prop:Carl}.
Throughout this section, the metric $a$ is assumed to be of a block form \eqref{eq:block}
where the metric $\tilde{a}$ satisfies the conditions \hyperref[cond:a1]{(A1)}-\hyperref[cond:a2]{(A3)} from the introduction with $\mu =0$.

We first note that the estimate in Proposition ~\ref{prop:Carl} remains valid for a \emph{constant} coefficient metric in block form \eqref{eq:block}. This follows immediately from a change of coordinates (only involving the tangential variables). 
In order to extend Proposition \ref{prop:Carl} to \emph{variable} coefficient problems, we exploit the presence of the divergence contribution and in conformal coordinates localise the problem to scales of the size $C(a_j \tau)^{-\frac{1}{2}}$ or of size one, respectively (depending on the size of the metric). Here $\{a_j\}_{j\in \N}$ denotes the sequence that was used in the definition of the Carleman weight $h$ (see Step 1 in the proof of Proposition ~\ref{prop:Carl}).

We follow the argument in \cite{KT01} and argue in two steps: First, in the regime in which $h$ is convex, we  localise to very small scales (Lemma ~\ref{lem:local_est}). In the regime in which no convexity is present anymore, we localise to scales of order one in conformal coordinates (Lemma ~\ref{lem:local_est_1}). Finally, we patch these estimates together to derive the desired global bound of Proposition ~\ref{prop:syst_Carl}.

\begin{lem}
\label{lem:local_est}
Let $\tau \geq 1$ and $\epsilon>0$. Assume that $h$ is convex and
\begin{align*}
h' \in [\tau, 2 \tau], \ h'' \in [\epsilon\tau, 2\epsilon\tau], \ |h'''|\leq \tau.
\end{align*}
Assume that 
\begin{align*}
|x||\nabla a^{ij}(x)|\leq \delta \epsilon \mbox{ in } I_\ell:=\{x\in \R^{n+1}_+: |x|\in [e^{-\ell-1},e^{-\ell}]\}
\end{align*}
for some sufficiently small constant $\delta>0$.
Then, for all $u$ with $\supp (u) \subset I_\ell$ and all $\tau \geq \tau_0\geq 1$ we have
\begin{align*}
&\tau
 \|e^{h(-\ln(|x|))}  x_{n+1}^{\frac{b}{2}} (1+\epsilon \tau)^{\frac{1}{2}} |x|^{-1} u\|_{L^2(\R^{n+1}_+)}
+  \|e^{h(-\ln(|x|))}x_{n+1}^{\frac{b}{2}} (1+\epsilon \tau)^{\frac{1}{2}}  \nabla u\|_{L^2(\R^{n+1}_+)}\\
& +\tau^{\frac{1-b}{2}}\|e^{h(-\ln(|x|))}(1+\epsilon \tau)^{\frac 1 2}|x|^{\frac{b-1}2}u\|_{L^2(\R^n\times\{0\})}
\\
&\leq C \left( \| e^{h(-\ln(|x|))} |x| x_{n+1}^{-\frac{b}{2}} \p_i x_{n+1}^{b}a^{ij} \p_j u\|_{L^2(\R^{n+1}_+)} 
\right.\\
& \quad \left. + \tau^{\frac{1+b}{2}} \|e^{h(-\ln(|x|))} |x|^{\frac{1-b}{2}} \lim\limits_{x_{n+1}\rightarrow 0} x_{n+1}^{b} \p_{n+1} u\|_{L^2(\R^n \times \{0\})} \right).
\end{align*}
\end{lem}

\begin{proof}
\emph{Step 1: Restricted support.} We first assume that $u$ is supported on a ball $B_{C_0|x_0|(\epsilon \tau)^{-\frac{1}{2}}}(x_0)$ for some $x_0 \in I_\ell$ or in some half ball $B_{C_0|x_0|(\epsilon \tau)^{-\frac{1}{2}}}^+(x_0)$ for some $x_0 \in I_\ell\cap (\R^n \times \{0\})$ and where the constant $C_0>0$ is still to be determined (see Step 2). As the arguments are similar in both cases, we only discuss the case of the full ball in detail in the sequel.
We note that for $x\in B_{C_0|x_0|(\epsilon \tau)^{-\frac{1}{2}}}(x_0)\cap I_\ell$ we have  
\begin{align}
\label{eq:freeze}
\begin{split}
|a^{ij}(x)-a^{ij}(x_0)|
&\leq \sup\limits_{x \in B_{C_0 |x_0|(\epsilon \tau)^{-\frac{1}{2}}}(x_0)}|\nabla a^{ij}(x)||x-x_0|\\
&\leq C\delta \epsilon \frac{C_0}{|x_0|}(\epsilon \tau)^{-\frac{1}{2}}|x_0|
\leq C C_0 \delta (\epsilon \tau^{-1})^{\frac{1}{2}}.
\end{split}
\end{align}

Next, we apply the Carleman estimate from Proposition ~\ref{prop:Carl} to the equation
\begin{align*}
\p_i x_{n+1}^b a^{ij}(x_0) \p_j u & = f + \p_i (a^{ij}(x_0)-a^{ij}(x))\p_j u \mbox{ in } B_{C_0 |x_0|(\epsilon \tau)^{-\frac{1}{2}}}(x_0),\\
\lim\limits_{x_{n+1} \rightarrow 0} x_{n+1}^{b} \p_{n+1} u & = g,
\end{align*}
where $f = \p_i x_{n+1}^b a^{ij} \p_j u$ and where we recall the block structure \eqref{eq:block} of $a^{ij}$. By virtue of Proposition ~\ref{prop:Carl} we obtain
\begin{align}
\label{eq:Carl_freeze}
\begin{split}
&\tau
 \|e^{h(-\ln(|x|))}  x_{n+1}^{\frac{b}{2}} (1+\epsilon \tau)^{\frac{1}{2}} |x|^{-1} u\|_{L^2(\R^{n+1}_+)}
+  \|e^{h(-\ln(|x|))}x_{n+1}^{\frac{b}{2}} (1+\epsilon \tau)^{\frac{1}{2}}  \nabla u\|_{L^2(\R^{n+1}_+)}\\
& +\tau^{\frac{1-b}{2}}\|e^{h(-\ln(|x|))}(1+\epsilon \tau)^{\frac 1 2}|x|^{\frac{b-1}2}u\|_{L^2(\R^n\times\{0\})}
\\
&\leq C \left( \| e^{h(-\ln(|x|))} |x| x_{n+1}^{-\frac{b}{2}} f\|_{L^2(\R^{n+1}_+)} 
+ \tau \| e^{h(-\ln(|x|))} x_{n+1}^{\frac{b}{2}} (a^{ij}-a^{ij}(x_0))\p_j u\|_{L^2(\R^{n+1}_+)}
\right.\\
& \quad \left. + \tau^{\frac{1+b}{2}} \|e^{h(-\ln(|x|))} |x|^{\frac{1-b}{2}} \lim\limits_{x_{n+1}\rightarrow 0} x_{n+1}^{b} \p_{n+1} u\|_{L^2(\R^n \times \{0\})} \right).
\end{split}
\end{align}
In order to deduce the desired estimate under the support constraint, it suffices to bound the second bulk term in the above inequality. To this end, we invoke \eqref{eq:freeze} and estimate
\begin{align*}
\| e^{h(-\ln(|x|))} x_{n+1}^{\frac{b}{2}} (a^{ij}-a^{ij}(x_0))\p_j u\|_{L^2(\R^{n+1}_+)}
\leq C C_0 \delta (\epsilon \tau^{-1})^{\frac{1}{2}} \|e^{h(-\ln(|x|))} x_{n+1}^{\frac{b}{2}} \nabla u \|_{L^2(\R^{n+1}_+)}.
\end{align*}
For $\delta>0$ sufficiently small (but independent of $u$), it is possible to absorb this contribution into the left hand side of \eqref{eq:Carl_freeze}.\\

\emph{Step 2: Localisation.}
We seek to apply the previous argument by localising a general solution $u$ with $\supp (u) \subset I_\ell$ by a partition of unity. Here commutator estimates play a crucial role and provide a natural limitation to the possible localisation scale.

We consider a partition of unity $\{\psi_k\}_{k\in \{1,\dots,K\}}$ associated with the half annulus $I_\ell$ and a finite collection $\{x_{k}\}_{k\in \{1,\dots, K\}}$ of points in $I_\ell$ such that $\supp(\psi_k) \subset B_{C_0|x_k|(\epsilon \tau)^{-\frac{1}{2}}}(x_k)$ or $\supp(\psi_m) \subset B_{C_0|x_{m}|(\epsilon \tau)^{-\frac{1}{2}}}^+(x_{m})$ (where in the latter case $x_m \in I_{\ell}\cap (\R^n \times \{0\})$) and such that the balls and half-balls $B_{C_0|x_k|(\epsilon \tau)^{-\frac{1}{2}}}(x_k)$ and $B_{C_0|x_{m}|(\epsilon \tau)^{-\frac{1}{2}}}^+(x_{m})$ cover the interval $I_\ell $ (with controlled overlap). Without loss of generality, we choose the partition of unity and the points $x_k$ such that the following estimate hold:
\begin{align}
\label{eq:bounds_psi_k}
|\nabla^{\alpha} \psi_k| \leq C^{-|\alpha|}_1|x_k|^{-|\alpha|}(\epsilon \tau)^{\frac{|\alpha|}{2}} \mbox{ for } |\alpha| \in \{0,1,2\}, 
\end{align}
and $C_1=C_1(C_0)>0$.
Further, without loss of generality, we may assume that  $${\lim\limits_{x_{n+1}\rightarrow 0} x_{n+1}^{b}\p_{n+1}\psi_k(x) = 0}.$$
We then write $u= \sum\limits_{k=1}^K \psi_k u$. 
With this in hand, we apply the triangle inequality as well as the Carleman estimate from Step 1:
\begin{align}
\label{eq:Carl_split}
\begin{split}
&\tau
 \|e^{h(-\ln(|x|))}  x_{n+1}^{\frac{b}{2}} (1+\epsilon \tau)^{\frac{1}{2}} |x|^{-1} u\|_{L^2(\R^{n+1}_+)}
+  \|e^{h(-\ln(|x|))}x_{n+1}^{\frac{b}{2}} (1+\epsilon \tau)^{\frac{1}{2}}  \nabla u\|_{L^2(\R^{n+1}_+)}\\
& \quad +\tau^{\frac{1-b}{2}}\|e^{h(-\ln(|x|))}(1+\epsilon \tau)^{\frac 1 2}|x|^{\frac{b-1}2}u\|_{L^2(\R^n\times\{0\})}
\\
&\leq \sum\limits_{k=1}^K \left(\tau
 \|e^{h(-\ln(|x|))}  x_{n+1}^{\frac{b}{2}} (1+\epsilon \tau)^{\frac{1}{2}} |x|^{-1} \psi_k u\|_{L^2(\R^{n+1}_+)}
+  \|e^{h(-\ln(|x|))}x_{n+1}^{\frac{b}{2}} (1+\epsilon \tau)^{\frac{1}{2}}  \nabla (\psi_k u)\|_{L^2(\R^{n+1}_+)}\right)\\
& \quad +\sum\limits_{k=1}^K \tau^{\frac{1-b}{2}}\|e^{h(-\ln(|x|))}(1+\epsilon \tau)^{\frac 1 2}|x|^{\frac{b-1}2}\psi_k u \|_{L^2(\R^n\times\{0\})}\\
&\leq C \sum\limits_{k=1}^K \Big( \| e^{h(-\ln(|x|))}  x_{n+1}^{-\frac{b}{2}} |x|f_k\|_{L^2(\R^{n+1}_+)} 
\\
& \qquad \quad  \qquad  + \tau^{\frac{1+b}{2}} \|e^{h(-\ln(|x|))} |x|^{\frac{1-b}{2}} \lim\limits_{x_{n+1}\rightarrow 0} x_{n+1}^{b} \p_{n+1} (\psi_k u)\|_{L^2(\R^n \times \{0\})} \Big).
\end{split}
\end{align}
We first consider the bulk contribution on the right hand side for which
\begin{align*}
f_k = \p_i x_{n+1}^b a^{ij}\p_j (u\psi_k)
= \psi_k f + 2 a^{ij}x_{n+1}^{b} \p_i \psi_k \p_j u + u a^{ij} \p_{i }(x_{n+1}^b \p_j\psi_k) + u x_{n+1}^b (\p_i a^{ij})(\p_j \psi_k).
\end{align*}
Using the bounds for $\psi_k$ from \eqref{eq:bounds_psi_k} as well as the estimate for $|\nabla a^{ij}|$, we obtain 
\begin{align}
\label{eq:inhom_Carl}
\begin{split}
\| e^{h(-\ln(|x|))} |x| x_{n+1}^{-\frac{b}{2}} f_k\|_{L^2(\R^{n+1}_+)} 
&\leq 
\| e^{h(-\ln(|x|))}|x| x_{n+1}^{-\frac{b}{2}}  f \psi_k\|_{L^2(\R^{n+1}_+)} \\
& \quad + C_1^{-1}(1+ (\epsilon\tau)^{\frac{1}{2}})\| e^{h(-\ln(|x|))} x_{n+1}^{\frac{b}{2}} \psi_k \nabla u\|_{L^2(\R^{n+1}_+)} \\
& \quad + C_1^{-1}(1+ \epsilon\tau)\| e^{h(-\ln(|x|))} x_{n+1}^{\frac{b}{2}} |x|^{-1} \psi_k u\|_{L^2(\R^{n+1}_+)} .
\end{split}
\end{align}
Choosing $C_1=C_1(C_0)>1$ sufficiently large (by choosing $C_0>0$ appropriately), then allows us to absorb the second and third contribution from \eqref{eq:inhom_Carl} into the left hand side of \eqref{eq:Carl_split}. 

For the boundary term in \eqref{eq:Carl_split}, we use the fact that by construction of the partition of unity $\lim\limits_{x_{n+1}\rightarrow 0} x_{n+1}^{b}\p_{n+1} (\psi_k u) = \psi_k \lim\limits_{x_{n+1}\rightarrow 0} x_{n+1}^{b}\p_{x_{n+1}} u$.

Using this and the finite overlap of the supports of the functions $\psi_k$, allows us to turn \eqref{eq:Carl_split} into
\begin{align}
\label{eq:Carl_split2}
\begin{split}
&\tau
 \|e^{h(-\ln(|x|))}  x_{n+1}^{\frac{b}{2}} (1+\epsilon \tau)^{\frac{1}{2}} |x|^{-1} u\|_{L^2(\R^{n+1}_+)}
+  \|e^{h(-\ln(|x|))}x_{n+1}^{\frac{b}{2}} (1+\epsilon \tau)^{\frac{1}{2}}  \nabla u\|_{L^2(\R^{n+1}_+)}\\
& +\tau^{\frac{1-b}{2}}\|e^{h(-\ln(|x|))}(1+\epsilon \tau)^{\frac 1 2}|x|^{\frac{b-1}2}u\|_{L^2(\R^n\times\{0\})}
\\
&\leq C \sum\limits_{k=1}^K \left( \| e^{h(-\ln(|x|))} |x| x_{n+1}^{-\frac{b}{2}} \psi_k f\|_{L^2(\R^{n+1}_+)} 
\right.\\
& \qquad \qquad\left. + \tau^{\frac{1+b}{2}} \|e^{h(-\ln(|x|))} |x|^{\frac{1-b}{2}} \psi_k \lim\limits_{x_{n+1}\rightarrow 0} x_{n+1}^{b} \p_{n+1} u\|_{L^2(\R^n \times \{0\})} \right)\\
&\leq C \left( \| e^{h(-\ln(|x|))} x_{n+1}^{-\frac{b}{2}} |x|  f\|_{L^2(\R^{n+1}_+)} + \tau^{\frac{1+b}{2}} \|e^{h(-\ln(|x|))} |x|^{\frac{1-b}{2}} \lim\limits_{x_{n+1}\rightarrow 0} x_{n+1}^{b} \p_{n+1} u\|_{L^2(\R^n \times \{0\})} \right),
\end{split}
\end{align}
which concludes the proof of the argument.
\end{proof}

Similarly as in Lemma ~\ref{lem:local_est}, it is also possible to deal with the situation of even smaller perturbations of the metric without invoking the convexity of the weight:

\begin{lem}
\label{lem:local_est_1}
Let $\tau \geq 1$. 
Assume that  
\begin{align*}
|x||\nabla a^{ij}|\leq \delta \tau^{-1} \mbox{ in } I_\ell:=\{x\in \R^{n+1}_+: |x|\in [e^{-\ell-1},e^{-\ell}]\}
\end{align*}
for some sufficiently small constant $\delta>0$.
Then, for all $u$ with $\supp (u) \subset I_\ell$ and all $\tau \geq \tau_0\geq 1$ we have
\begin{align*}
&\tau
 \|e^{h(-\ln(|x|))}  x_{n+1}^{\frac{b}{2}}  |x|^{-1} u\|_{L^2(\R^{n+1}_+)}
+  \|e^{h(-\ln(|x|))}x_{n+1}^{\frac{b}{2}}   \nabla u\|_{L^2(\R^{n+1}_+)}\\
& +\tau^{\frac{1-b}{2}}\|e^{h(-\ln(|x|))}|x|^{\frac{b-1}2}u\|_{L^2(\R^n\times\{0\})}
\\
&\leq C \Big( \| e^{h(-\ln(|x|))} |x| x_{n+1}^{-\frac{b}{2}} \p_i x_{n+1}^{b}a^{ij} \p_j u\|_{L^2(\R^{n+1}_+)} 
\\
& \quad \qquad + \tau^{\frac{1+b}{2}} \|e^{h(-\ln(|x|))} |x|^{\frac{1-b}{2}} \lim\limits_{x_{n+1}\rightarrow 0} x_{n+1}^{b} \p_{n+1} u\|_{L^2(\R^n \times \{0\})} \Big).
\end{align*}
\end{lem}

\begin{proof}
The argument follows along the same lines as the proof of Lemma ~\ref{lem:local_est_1}, however now we directly localise to scales of the order $C_0|x_k|$ around a finite number of points $x_k \in I_\ell$. Using an associated partition of unity then yields the desired result.
\end{proof}

Relying on the previous result, we obtain global Carleman estimates:

\begin{prop}
\label{prop:global_est}
Let the metric $a:\R^{n+1}_+ \rightarrow \R^{(n+1)\times (n+1)}$ be of a block form as in \eqref{eq:block}
where the metric $\tilde{a}$ is assumed to
satisfy the conditions \hyperref[cond:a1]{(A1)}-\hyperref[cond:a2]{(A3)} with $\mu=0$.
Let $u\in H^1(B^+_4,x_{n+1}^b)$ with ${\supp(u)\subset B^+_4\backslash\{0\}}$.
Then, for each $\tau>\tau_0\geq 1$ there exist a weight function $h$ and a constant $C>0$ independent of $\tau$ such that it holds
\begin{align*}
&\tau
 \|e^{h(-\ln(|x|))}  x_{n+1}^{\frac{b}{2}} (1+\overline{h})^{\frac{1}{2}} |x|^{-1} u\|_{L^2(\R^{n+1}_+)}
+  \|e^{h(-\ln(|x|))}x_{n+1}^{\frac{b}{2}} (1+\overline{h})^{\frac{1}{2}}  \nabla u\|_{L^2(\R^{n+1}_+)}\\
& +\tau^{\frac{1-b}{2}}\|e^{h(-\ln(|x|))}(1+\overline{h})^{\frac 1 2}|x|^{\frac{b-1}2}u\|_{L^2(\R^n\times\{0\})}
\\
&\leq C \Big( \| e^{h(-\ln(|x|))} |x| x_{n+1}^{-\frac{b}{2}} \nabla\cdot x_{n+1}^{b} a \nabla u\|_{L^2(\R^{n+1}_+)} 
\\
& \quad\qquad  + \tau^{\frac{1+b}{2}} \|e^{h(-\ln(|x|))} |x|^{\frac{1-b}{2}} \lim\limits_{x_{n+1}\rightarrow 0} x_{n+1}^{b} \p_{n+1} u\|_{L^2(\R^n \times \{0\})} \Big).
\end{align*}
Here $\overline{h}(x):=  h''(t)|_{t=-\ln(|x|)}$.
\end{prop}

\begin{proof}
In order to deduce this, we use the properties of the weight $h$. By relying on a partition of unity, we localise the set-up to dyadic intervals. Then, with the constants $a_j$ as in Step 1 in the proof of Proposition ~\ref{prop:Carl}, if $a_j \tau>1$, we apply Lemma ~\ref{lem:local_est}, while if $a_j \tau<1$, we invoke Lemma ~\ref{lem:local_est_1}.
\end{proof}

\subsection{Proof of Propositions \ref{prop:syst_Carl} and ~\ref{prop:syst_Carl_1}}
\label{sec:stronger}
Proposition ~\ref{prop:syst_Carl} arises as an iteration of the Carleman estimate from Proposition ~\ref{prop:global_est} (or directly from Proposition \ref{prop:Carl} if $a^{ij}=\delta^{ij}$).
In the sequel, we present the variable and constant coefficient proofs simultaneously.

\begin{proof}[Proof of Proposition ~\ref{prop:syst_Carl}] We seek to iterate the second order Carleman estimates in order to obtain an estimate for full the system. To this end, we set $\tilde{w}_j=(1+\overline{h})^{\frac{m-j}{2}} |x|^{-2m+2j}\tilde{u}_j$ and apply Proposition ~\ref{prop:global_est} iteratively. We argue in two steps.\\

\emph{Step 1: Building block estimate.}
For $j\in\{0,\dots, m\}$ we have
\begin{align*}
L_b \tilde{w}_j(x) 
&= (1+\overline{h})^{\frac{m-j}{2}} |x|^{-2m+2j} (L_b \tilde{u}_j)(x) + 2   \nabla \big((1+\overline{h})^{\frac{m-j}{2}} |x|^{-2m+2j}\big)\cdot a(x) \nabla \tilde{u}_j(x) \\
& \quad + \tilde{u}_j(x) L_b\big((1+\overline{h})^{\frac{m-j}{2}} |x|^{-2m+2j}\big)\\
& = (1+\overline{h})^{\frac{m-j}{2}} |x|^{-2m+2j} \big(f_j(x) + \tilde{u}_{j+1}(x)\big) + 2 \nabla \big((1+\overline{h})^{\frac{m-j}{2}} |x|^{-2m+2j}\big)\cdot  a(x) \nabla \tilde{u}_j(x) \\
& \quad + \tilde{u}_j(x) L_b\big((1+\overline{h})^{\frac{m-j}{2}} |x|^{-2m+2j}\big).
\end{align*}
Hence, as consequence of Proposition ~\ref{prop:global_est}, we deduce the estimate
\begin{align}
\label{eq:est_building_block}
\begin{split}
&\tau^{m+1-j} \| e^{h(-\ln(|x|))} x_{n+1}^{\frac{b}{2}}(1+\overline{h})^{\frac{m+1-j}{2}}|x|^{-2m-1+2j}   \tilde{u}_j \|_{L^2(B_4^+)} \\
&\quad+ \tau^{m-j}\|e^{h(-\ln(|x|))} x_{n+1}^{\frac{b}{2}}(1+\overline{h})^{\frac{m+1-j}{2}}|x|^{-2m+2j}   \nabla \tilde{u}_j \|_{L^2(B_4^+)}\\
&\leq C \left(\tau^{m-j} \|e^{h(-\ln(|x|))} x_{n+1}^{\frac{b}{2}}(1+\overline{h})^{\frac{m-j}{2}} |x|^{-2m+1+2j}  (f_j + \tilde{u}_{j+1})  \|_{L^2(B_4^+)} \right.\\
& \qquad \quad
+ \tau^{m-j+\frac{1+b}{2}} \|e^{h(-\ln(|x|))} (1+\overline{h})^{\frac{m-j}{2}} |x|^{-2m+2j+\frac{1-b}{2}} g_j\|_{L^2(B_4')} \\
& \qquad \quad + \tau^{m-j} 
\|e^{h(-\ln(|x|))} x_{n+1}^{\frac{b}{2}}|x|
 \nabla \big((1+\overline{h})^{\frac{m-j}{2}} |x|^{-2m+2j}\big)\cdot a\nabla \tilde{u}_j 
\|_{L^2(B_4^+)}\\
& \left. \qquad \quad + \tau^{m-j} 
\|e^{h(-\ln(|x|))} x_{n+1}^{\frac{b}{2}}|x| L_b\big((1+\overline{h})^{\frac{m-j}{2}} |x|^{-2m+2j}\big)\tilde{u}_j\|_{L^2(B_4^+)}
\right).
\end{split}
\end{align}
We note that by using the fact that
\begin{align*}
|\nabla^{\alpha} \big((1+\overline{h})^{\frac{m-j}{2}} |x|^{-2m+2j}) | \leq C|x|^{-|\alpha|}\big((1+\overline{h})^{\frac{m-j}{2}} |x|^{-2m+2j}\big), \ \mbox{ for } |\alpha| \in \{1,2\},
\end{align*}
the regularity of the metric $a$,
and by choosing $\tau_0\geq 1$ sufficiently large, we may absorb the last two contributions on the right hand side of \eqref{eq:est_building_block} into the left hand side of \eqref{eq:est_building_block}. As a consequence, we obtain 
\begin{align}
\label{eq:est_building_block1}
\begin{split}
&\tau^{m+1-j} \| e^{h(-\ln(|x|))} x_{n+1}^{\frac{b}{2}}(1+\overline{h})^{\frac{m+1-j}{2}}|x|^{-2m-1+2j}   \tilde{u}_j \|_{L^2(B_4^+)} \\
&\quad+ \tau^{m-j}\|e^{h(-\ln(|x|))} x_{n+1}^{\frac{b}{2}}(1+\overline{h})^{\frac{m+1-j}{2}}|x|^{-2m+2j}   \nabla \tilde{u}_j \|_{L^2(B_4^+)}\\
&\leq C \left(\tau^{m-j} \|e^{h(-\ln(|x|))} x_{n+1}^{\frac{b}{2}}(1+\overline{h})^{\frac{m-j}{2}} |x|^{-2m+1+2j}  (f_j + \tilde{u}_{j+1})  \|_{L^2(B_4^+)} \right.\\
& \qquad \quad
+ \left. \tau^{m-j+\frac{1+b}{2}} \|e^{h(-\ln(|x|))} (1+\overline{h})^{\frac{m-j}{2}} |x|^{-2m+2j+\frac{1-b}{2}} g_j\|_{L^2(B_4')} \right).
\end{split}
\end{align}

\emph{Step 2: Iteration.}
Using the $C^{2m,1}$ coefficient regularity, we iterate the building block estimate: 
\begin{align*}
& \tau^{m+1-j} \|e^{h(-\ln(|x|))}x_{n+1}^{\frac{b}{2}}(1+\overline{h})^{\frac{m+1-j}{2}}   |x|^{-2m-1+2j}\tilde{u}_j\|_{L^2(B_4^+)}\\
& \quad + \tau^{m-j} \| e^{h(-\ln(|x|))} x_{n+1}^{\frac{b}{2}} (1+ \overline{h})^{\frac{m+1-j}{2}}|x|^{-2m+2j}  \nabla \tilde{u}_j \|_{L^2(B_4^+)}\\
& \leq C \Big( \tau^{m-j} \|e^{h(-\ln(|x|))} x_{n+1}^{\frac{b}{2}} (1+\overline{h})^{\frac{m-j}{2}}  |x|^{-2m+1+2j}f_j \|_{L^2(B_4^+)}\\
& \qquad\quad +   \tau^{m-j} \|e^{h(-\ln(|x|))} x_{n+1}^{\frac{b}{2}}(1+\overline{h})^{\frac{m-j}{2}} |x|^{-2m+1+2j}   \tilde{u}_{j+1} \|_{L^2(B_4^+)} \Big)\\
& \leq C  \Big(\sum\limits_{k=j}^{j+1} \tau^{m-k} \|e^{h(-\ln(|x|))} x_{n+1}^{\frac{b}{2}}  (1+\overline{h})^{\frac{m-k}{2}} |x|^{-2m+1+2k}f_k \|_{L^2(B_4^+)}\\
& \quad\qquad +   \tau^{m-j-1} \|e^{h(-\ln(|x|))} x_{n+1}^{\frac{b}{2}} (1+\overline{h})^{\frac{m-j-1}{2}}|x|^{-2m+1+2(j+1)} \tilde{u}_{j+2} \|_{L^2(B_4^+)}\Big) \\
& \leq C \Big( \sum\limits_{k=j}^{m} \tau^{m-k} \|e^{h(-\ln(|x|))} x_{n+1}^{\frac{b}{2}} (1+\overline{h})^{\frac{m-k}{2}}|x|^{-2m+1+2k}  f_k \|_{L^2(B_4^+)}\\
& \quad\qquad  +  \tau^{\frac{1+b}{2}} \|e^{h(-\ln(|x|))} |x|^{\frac{1-b}{2}}g\|_{L^2(B_4')}\Big).
\end{align*}
Here we used that $g_j=0$ except for $g_m=g$ and applied the triangle inequality to separate the bulk terms on the right hand side.
Summing these estimates from $j=0$ to $m$, then yields the desired bound \eqref{eq:Carl_system}.
\end{proof}

\begin{proof}[Proof of Proposition \ref{prop:syst_Carl_1}.]
We split the proof into two steps: First we deal with commutation relations arising in iterations of the second order estimates of Proposition \ref{prop:global_est} and then we iterate the resulting bounds.\\

\emph{Step 1: Commutators.}
We observe that \begin{align}
\label{eq:commute_a}
\begin{split}
\p'_k L_b=L_b\p'_k+(\p'_k \tilde a^{ij})\p'_{ij}+(\p'_{ik}\tilde a^{ij}) \p'_j,
\end{split}
\end{align}
where $\p'_{k}$ denotes derivatives in tangential directions.
Due to the assumptions in condition \hyperref[cond:a2]{(A2)}, the contributions in the Carleman estimate arising from $\nabla'L_b u_{j-1}$  for $j\in\{1,\dots,m\}$ can be bounded from below by terms which are controlled by $L_b \nabla'u_{j-1}$ if $\delta>0$ is chosen sufficiently small: By Proposition \ref{prop:global_est}, the condition \hyperref[cond:a2]{(A2)} and by using that by interpolation $|x||\nabla a (x)|$, $|x|^2|\nabla^2 a (x)|\leq \tilde{C}\delta $ for $x\in B^+_4$ we estimate the commutator contributions in \eqref{eq:commute_a} by
\begin{align*}
&\|e^{h(-\ln(|x|))} x_{n+1}^{\frac b 2}(1+\overline h)^{\frac{m+1-j}{2}}|x|^{-2m-1-2j}\nabla L_b\tilde u_{j-1}\|_{L^2(B_4^+)}\\
&\geq\|e^{h(-\ln(|x|))} x_{n+1}^{\frac b 2}(1+\overline h)^{\frac{m+1-j}{2}}|x|^{-2m-1-2j}L_b\nabla'\tilde u_{j-1}\|_{L^2(B_4^+)}\\
&\quad-\|e^{h(-\ln(|x|))} x_{n+1}^{\frac b 2}(1+\overline h)^{\frac{m+1-j}{2}}|x|^{-2m-1-2j}|\nabla a||(\nabla')^2 u_{j-1}|\|_{L^2(B_4^+)}\\
&\quad-\|e^{h(-\ln(|x|))} x_{n+1}^{\frac b 2}(1+\overline h)^{\frac{m+1-j}{2}}|x|^{-2m-1-2j}|\nabla^2 a||\nabla'\tilde u_{j-1}|\|_{L^2(B_4^+)}\\
&\geq\|e^{h(-\ln(|x|))} x_{n+1}^{\frac b 2}(1+\overline h)^{\frac{m+1-j}{2}}|x|^{-2m-1-2j}L_b\nabla'\tilde u_{j-1}\|_{L^2(B_4^+)}\\
&\quad-\tilde{C}\delta \Big(\|e^{h(-\ln(|x|))} x_{n+1}^{\frac b 2}(1+\overline h)^{\frac{m+2-j}{2}}|x|^{-2m-2-2j}|\nabla\nabla' u_{j-1}|\|_{L^2(B_4^+)}\\
&\quad\quad\quad+\tau\|e^{h(-\ln(|x|))} x_{n+1}^{\frac b 2}(1+\overline h)^{\frac{m+2-j}{2}}|x|^{-2m-3-2j}\nabla'\tilde u_{j-1}\|_{L^2(B_4^+)}\Big)\\
&\geq\|e^{h(-\ln(|x|))} x_{n+1}^{\frac b 2}(1+\overline h)^{\frac{m+1-j}{2}}|x|^{-2m-1-2j}L_b\nabla'\tilde u_{j-1}\|_{L^2(B_4^+)}\\
&\quad-C \tilde{C}\delta  \|e^{h(-\ln(|x|))} x_{n+1}^{\frac b 2}(1+\overline h)^{\frac{m+1-j}{2}}|x|^{-2m-1-2j}L_b\nabla'\tilde u_{j-1}\|_{L^2(B_4^+)}.
\end{align*}
 Choosing $\delta>0$ so small that $C\tilde{C}\delta\leq \frac 1 2 $, in addition to the estimates from Proposition \ref{prop:global_est}, we also deduce that for $j\in\{1,\dots,m\}$ the following bounds holds
\begin{align}
\label{eq:grad_est_a}
\begin{split}
&\|e^{h(-\ln(|x|))} x_{n+1}^{\frac b 2}(1+\overline h)^{\frac{m+1-j}{2}}|x|^{-2m-1-2j}\nabla  \tilde{u}_{j}\|_{L^2(B_4^+)}\\
&\geq \|e^{h(-\ln(|x|))} x_{n+1}^{\frac b 2}(1+\overline h)^{\frac{m+1-j}{2}}|x|^{-2m-1-2j}\nabla'  \tilde{u}_{j}\|_{L^2(B_4^+)}\\
&= \|e^{h(-\ln(|x|))} x_{n+1}^{\frac b 2}(1+\overline h)^{\frac{m+1-j}{2}}|x|^{-2m-1-2j}\nabla'(L_b \tilde{u}_{j-1} - f_{j-1})\|_{L^2(B_4^+)}\\
&\geq \|e^{h(-\ln(|x|))} x_{n+1}^{\frac b 2}(1+\overline h)^{\frac{m+1-j}{2}}|x|^{-2m-1-2j}\nabla'L_b \tilde{u}_{j-1} \|_{L^2(B_4^+)} \\
& \quad - \|e^{h(-\ln(|x|))} x_{n+1}^{\frac b 2}(1+\overline h)^{\frac{m+1-j}{2}}|x|^{-2m-1-2j}\nabla' f_{j-1}\|_{L^2(B_4^+)}\\
&\geq \frac{C}{2}\Big(\tau \|e^{h(-\ln(|x|))} x_{n+1}^{\frac b 2}(1+\overline h)^{\frac{m+2-j}{2}}|x|^{-2m-3-2j}\nabla'\tilde u_{j-1}\|_{L^2(B_4^+)}\\
&\qquad\quad+\|e^{h(-\ln(|x|))} x_{n+1}^{\frac b 2}(1+\overline h)^{\frac{m+2-j}{2}}|x|^{-2m-2-2j}\nabla\nabla'\tilde u_{j-1}\|_{L^2(B_4^+)}\\
& \qquad \quad - \|e^{h(-\ln(|x|))} x_{n+1}^{\frac b 2}(1+\overline h)^{\frac{m+1-j}{2}}|x|^{-2m-1-2j}\nabla' f_{j-1}\|_{L^2(B_4^+)}
\Big).
\end{split}
\end{align}

\emph{Step 2: Iteration.}
With the additional bounds from \eqref{eq:grad_est_a} in hand, we can iterate the Carleman estimate. For $j\geq 1$ and $\ell\in\{1,\dots,j\}$ we infer
\begin{align*}
&\|e^{h(-\ln(|x|))} x_{n+1}^{\frac b 2}(1+\overline h)^{\frac{m+1-j}{2}}|x|^{-2m-1+2j}\tilde u_{j}\|_{L^2(B_4^+)}\\
&\geq C\Big(\sum_{k=0}^{\ell-1} \tau^{\ell-k}\|e^{h(-\ln(|x|))} x_{n+1}^{\frac b 2}(1+\overline h)^{\frac{m+1-j+\ell}{2}}|x|^{-2m-1+2j-2\ell+k}(\nabla')^k\tilde u_{j-\ell}\|_{L^2(B_4^+)}\\
&\qquad\quad+\sum_{k=0}^{\ell-1} \tau^{\ell-k-1}\|e^{h(-\ln(|x|))} x_{n+1}^{\frac b 2}(1+\overline h)^{\frac{m+1-j+\ell}{2}}|x|^{-2m+2j-2\ell+k}\nabla(\nabla')^k\tilde u_{j-\ell}\|_{L^2(B_4^+)}\\
&\qquad\quad-\sum_{i=1}^{\ell}\sum_{k=0}^{i-1} \tau^{i-k-1}\|e^{h(-\ln(|x|))} x_{n+1}^{\frac b 2}(1+\overline h)^{\frac{m-j+i}{2}}|x|^{-2m+1+2j-2i+k}(\nabla')^{k} f_{j-i}\|_{L^2(B_4^+)}\Big).
\end{align*}
A similar estimate holds for the gradient term.
Adding these estimates to the bounds from Proposition \ref{prop:syst_Carl} and summing over all $j \in \{1,\dots,m\}$ implies the desired bulk estimate. Combining this with the boundary-bulk interpolation result from Corollary ~\ref{cor:boundary_bulk} then implies the claim.
\end{proof}

\section{On the Strong Unique Continuation Property for the Extension Problem} 
\label{sec:SUCP}
In this section we study the strong unique continuation property for the system \eqref{eq:WUCP_syst1} and seek to reduce it to the weak unique continuation property. As in \cite{FF14, Rue15, GRSU18} we achieve this by careful  compactness and blow-up arguments.

From a technical point of view, the main challenge is to control solutions to our system also in the \emph{normal} direction in which we can only obtain information through the equation itself. Here two cases arise: 
\begin{itemize}
\item If we had vanishing of infinite order in the tangential \emph{and} normal directions, an immediate application of the Carleman estimate \eqref{eq:Carl_system_1} would allow us to prove the strong unique continuation property. 
\item If we are however dealing with solutions which a priori do \emph{not} vanish of infinite order in the normal directions, we have to argue more carefully, exploiting properties of our equations.
\end{itemize}

In this section, we consider solutions $u_j \in H^1_{loc}(\R^{n+1}_+, x_{n+1}^b)$ with $b\in (-1,1)$ of the system \eqref{eq:syst_err}, \eqref{eq:syst_err_boundary} with $f_0,\dots,f_m =0$ satifsying the following conditions:
\begin{itemize}
\item[(C)]
\label{cond:C} 
$a$ is of a block form \eqref{eq:block} and such that $\tilde a$ satisfies \hyperref[cond:a1]{(A1)}-\hyperref[cond:a3]{(A3)} with $\mu=2m$ and 
\begin{align*}
|g(x')|\leq \sum_{j=0}^m|q_j(x')||(\nabla')^ju_0(x',0)|
\end{align*}
with
$|q_j(x')|\leq C_{q_j}|x|^{-2m+j+b-1}$ for $j=\{0,\dots,m-1\}$ and
\begin{align*}
|q_{m}(x')|
\leq \left\{ 
\begin{array}{ll}
C_{q_{m}}|x|^{-m+b-1}, \mbox{ if } b< 0,\\
c_0|x|^{-m+b-1}, \mbox{ if } b=0,\\
C_{q_m} |x|^{-m+b-1+\epsilon}, \mbox{ if } 
\left\{
\begin{array}{l}
b\in(0,\frac 1 2)  \mbox{ and } m= 0,\\
b\in(0,1)  \mbox{ and } m\geq 1.
\end{array}
 \right. 
\end{array}
\right.
\end{align*}
Here $c_0>0$ is a sufficiently small constant (which is specified below), and $C_{q_j}>0$ are arbitrarily large, finite constants.
\end{itemize}

As the vanishing of infinite order in the normal directions is not a direct consequence of our assumptions on the infinite order vanishing in the tangential directions, we thus split the argument into two parts: 
\begin{itemize}
\item In the case of infinite order vanishing in all directions, i.e. for all $j\in\{0,\dots,m\}$
\begin{align*}
\lim\limits_{r\rightarrow 0} r^{-k} \|x_{n+1}^{\frac{b}{2}} u_j\|_{L^2(B_r^+)} = 0 \mbox{ for all } k \in \N,
\end{align*}
 we directly apply the Carleman estimate from Proposition \ref{prop:syst_Carl_1} (see Section \ref{sec:SUCP_inf_order}). 

\item If this is (a priori) not the case, 
i.e. there exist some $j\in \{0,\dots,m\}$, a subsequence $r_{\ell}\rightarrow 0$ and a constant $k_0\in \N$ such that
\begin{align}
\label{eq:van_ord}
\lim\limits_{\ell \rightarrow \infty} r_{\ell}^{-k_0} \|x_{n+1}^{\frac{b}{2}} u_j\|_{L^2(B_{r_{\ell}}^+)} \geq C_0 >0,
\end{align}
we  deduce doubling properties and then exploit these in a compactness argument to reduce the strong unique continuation property to the weak unique continuation property (see Section \ref{sec:SUCP_red}). 
\end{itemize}

\subsection{Reduction to the weak unique continuation property}
\label{sec:SUCP_red}
In the sequel, we seek to reduce the strong unique continuation property to a weak unique continuation result by a blow-up argument under the assumption that the solution vanishes to infinite order just in the \emph{tangential} directions (but \emph{not} in the normal directions, see \eqref{eq:van_ord}).

In order to deduce sufficient compactness for a blow-up argument, we 
first prove a doubling estimate for the functions $u_j$.
Here we exploit elliptic estimates and deal with the resulting boundary contributions by absorbing these into the bulk terms with finite order of vanishing (for sufficiently small radii).

\begin{prop}[Doubling]
\label{prop:doubling}
Let  $u_j\in H^1_{loc}(\R^{n+1}_+, x_{n+1}^b)$ for $j\in\{0,\dots,m\}$ be weak solutions of the system 
\eqref{eq:syst_err}, \eqref{eq:syst_err_boundary}
 with $f_0,\dots,f_m=0$  satisfying the conditions from \hyperref[cond:C]{(C)}.
Assume also that the tangential restrictions $u_j(x',0)$ vanish of infinite order at $x_0=0$ and that there exist some $j\in \{0,\dots,m\}$, a subsequence $r_{\ell}\rightarrow 0$ and a constant $k_0\in \N$ such that \eqref{eq:van_ord} holds.
Then, there exist a universal constant $C>1$ and a radius $r_{u_0}>0$ (the latter depending on $u_0$) such that for any $\ell \in \N$ and all $r\in (r_{\ell},2^{m}r_{\ell})\cap(0,r_{u_0})$ we have
\begin{align*}
\sum\limits_{j=0}^{m}r^{2j}\|x_{n+1}^{\frac{b}{2}} u_j\|_{L^2(B_{2r}^+)} 
+\sum\limits_{j=0}^{m}r^{2j+1}\|x_{n+1}^{\frac{b}{2}} \nabla u_j\|_{L^2(B_{2r}^+)} 
\leq C \sum\limits_{j=0}^{m}r^{2j}\|x_{n+1}^{\frac{b}{2}} u_j\|_{L^2(B_{r}^+)}.
\end{align*}
Here $r_{\ell}>0$ denotes the radii from \eqref{eq:van_ord}.
\end{prop}

\begin{rmk}\label{rmk:doubling_bdd}
In the case of bounded potentials, the same result holds without assuming that the functions $u_j(x',0)$ vanish of infinite order in the tangential directions. Moreover, in the setting of bounded potentials, the statement holds for all $r\in(0,r_0)$ (there is no intersection with the interval around $r_{\ell}$ here), where $r_0$ is sufficiently small but \emph{independent} of $u_0$. We refer to the proof of Proposition \ref{prop:doubling} for further details on this.
\end{rmk}

\begin{rmk}
\label{rmk:van_order}
Instead of restricting our doubling results to radii around $r_{\ell}$, we could also have argued as in Section 3 in \cite{KRS16}. This would have allowed us to conclude that the vanishing order is defined not only through a subsequence of radii but is independent of such a sequence. As a consequence, we would have obtained the statement of Proposition \ref{prop:doubling} for any choice of radius less that $r_{u_0}$. As our unique continuation argument does not rely on quantitative order of vanishing estimates, we do not further pursue this approach here.
\end{rmk}

\begin{proof}[Proof of Proposition \ref{prop:doubling}]
Consider $\tilde{u}_j := \eta u_j$, where $\eta$ is a radial cut-off function which is equal to one in $B_{3}^+ \setminus B_{r/4}^+$, vanishes outside of $B_4^+ \setminus B_{r/8}^+$ and satisfies the following bounds
\begin{align*}
  |\nabla^\alpha \eta| &\leq C \mbox{ in } B_{4}^+\setminus B_{3}^+ \mbox{ for } |\alpha|\leq m+2, \ \alpha \in \N^n,\\
  |\nabla^\alpha \eta| & \leq C r^{-|\alpha|} \mbox{ in } B_{r/4}^+\setminus B_{r/8}^+ \mbox{ for } |\alpha|\leq m+2.
\end{align*}
Here $r\in (0,r_0)$ where $r_0 \ll 1$ is chosen sufficiently small (with a choice that is explained later).
We note that the functions $\tilde{u}_j$ are solutions to the system from Proposition ~\ref{prop:syst_Carl} with $f_j := 2\nabla \eta \cdot(a\nabla u_j) + u_j L_b \eta$ and $|g(x')|\leq|\eta| \sum_{j=0}^m|q_j(x')||(\nabla')^ju_0(x',0)|$. Hence the Carleman estimates \eqref{eq:Carl_system} and \eqref{eq:Carl_system_1}  hold.\\

\textit{Step 1:  Boundary contributions.}
Let us assume that for all $j\in \{0,\dots,m\}$ it holds $|q_j(x)|\leq C_{q_j} |x|^{-2m+j+b-1+\varepsilon}$, where $\varepsilon\geq 0$.
We seek to absorb the boundary contributions from the right hand side of the estimate \eqref{eq:Carl_system_1} in Proposition \ref{prop:syst_Carl_1} into the ones on the left hand side. To this end, we compare the relevant contributions: The left hand side of the Carleman estimate controls contributions of the form
\begin{align}
\label{eq:boundary_terms}
\begin{split}
\sum\limits_{j=0}^{m} \tau^{m-j+\frac{1-b}{2}} \|e^{h(-\ln(|x|))} x_{n+1}^{\frac{b}{2}}(1+\overline{h})^{\frac{m+1}{2}} |x|^{-2m+j+\frac{b-1}{2}} (\nabla')^j \tilde{u}_0 \|_{L^2(B_4')},
\end{split}
\end{align}
while the boundary terms on the right hand side can be estimated from above by
\begin{align}
\label{eq:rhs}
\begin{split}
C\tau^{\frac{1+b}{2}}\sum\limits_{j=0}^{m} C_{q_j}\|e^{h(-\ln(|x|))} x_{n+1}^{\frac{b}{2}} |x|^{-2m+j+\frac{b-1}{2}+\varepsilon}|(\nabla')^j u_0|\|_{L^2(B_4')}.
\end{split}
\end{align}
In order to carry out the absorption argument, we distinguish three cases: 
\begin{itemize}
\item If $b<m-j$, for any $\epsilon \geq 0$, it suffices to
choose $\tau$ sufficiently large, in order to absorb the contributions from \eqref{eq:boundary_terms} into \eqref{eq:rhs}. Note that this always holds for $j\in\{0,\dots, m-1\}$ and also for $j=m$ if $b<0$.
\item If $j=m$ and $b=0$, it is still possible carry out this absorption argument in the case that $\varepsilon=0$ provided the constant $C_{q_m}$ is small enough. 
More precisely, after plugging the estimate \eqref{eq:boundary_terms} into \eqref{eq:Carl_system_1}, the relevant boundary contribution will carry the prefactor $C C_{q_m}$. Requiring that $C_{q_m}\leq c_0\leq\frac{1}{2C}$ then allows us to implement an absorption argument. 
\item Lastly, in the case of subcritical potentials, i.e. if $\varepsilon>0$, a wider range of values of $b$ is admissible by using the properties of $\overline h$.
Indeed, by the construction of $h''$
\begin{align*}
& \tau^{m-j+\frac{1-b}{2}} \|e^{h(-\ln(|x|))} x_{n+1}^{\frac{b}{2}}(1+\overline{h})^{\frac{m+1}{2}} |x|^{-2m+j+\frac{b-1}{2}} (\nabla')^j \tilde{u}_0 \|_{L^2(B_4')}
\\
&\quad \geq \tau^{m-j+\frac{1-b}{2} + \frac{m+1}{2}} \|e^{h(-\ln(|x|))} x_{n+1}^{\frac{b}{2}}|x|^{-2m+j+\frac{b-1}{2}+\nu \frac{m+1}{2}} (\nabla')^j \tilde{u}_0 \|_{L^2(B_4')}
\end{align*}
Analogous estimates hold for the gradient contributions. Hence, by an appropriate choice of $\nu>0$ (in the construction of the Carleman weight in the proof of Proposition \ref{prop:Carl}) it is possible to absorb all boundary contributions as long as $b<\frac{3m+1}{2}-j$  for any finite constant $C_{q_j}$ by choosing $\tau$ sufficiently large.
This enlarges the range of $b$ for $m=0$ to $b < \frac{1}{2}$ and for $j=m\geq 1$ to  $b< 1$.
\end{itemize}

\textit{Step 2: Bulk contributions.}
In discussing the bulk contributions, we first
deal with the bulk terms of the right hand side of the Carleman estimate which are localised on the unit scale. We will absorb these into the left hand side of the Carleman estimate. Secondly, we treat the contributions on the small scale $r>0$ for which we deduce the desired doubling estimate.

In the sequel, we use the following abbreviations for the respective half annuli
\begin{align*}
I_{1}:= B_{r/4}^+ \setminus B_{r/8}^+, \ 
I_2:= B_{2r}^+ \setminus B_{r/2}^+,\
I_3:= B_{5/2}^+ \setminus B_2^+, \ 
I_4:= B_{7/2}^+ \setminus B_3^+.
\end{align*}

For the convenience of the reader, we split the proof of the bulk estimates into two steps: In Step 2a, we deal with the case without gradient contributions. This allows us to introduce the ideas without resorting to too many technicalities. Then, in Step 2b, we deal with the full case including gradient terms.

\textit{Step 2a: Lowest order potentials.}
After having dealt with the boundary terms in Step 1,  the Carleman estimate \eqref{eq:Carl_system} turns into
\begin{align}
\label{eq:bulk_0}
\begin{split}
&\sum\limits_{j=0}^{m} \left(\tau^{m+1-j}\|e^{h(-\ln(|x|))} x_{n+1}^{\frac{b}{2}}(1+ \overline{h})^{\frac{m+1-j}{2}}|x|^{-2m-1+2j} u_j\|_{L^2(I_2\cup I_3)}\right. \\
&\qquad+ \left.\tau^{m-j}\|e^{h(-\ln(|x|))} x_{n+1}^{\frac{b}{2}}(1+ \overline{h})^{\frac{m+1-j}{2}}|x|^{-2m+2j} \nabla u_j\|_{L^2(I_2\cup I_3)}\right) \\
&\leq 
C \sum\limits_{j=0}^{m} \left(\tau^{m-j} \|e^{h(-\ln(|x|))} x_{n+1}^{\frac{b}{2}} (1+\overline{h})^{\frac{m-j}{2}}|x|^{-2m+1+2j}|\nabla \eta||\nabla u_j|\|_{L^2(I_1 \cup I_4)}\right.\\
&\qquad \qquad  +\left.  \tau^{m-j} \|e^{h(-\ln(|x|))} x_{n+1}^{\frac{b}{2}} (1+\overline{h})^{\frac{m-j}{2}}|x|^{-2m+1+2j}|L_b \eta|| u_j|\|_{L^2(I_1 \cup I_4)}\right),
\end{split}
\end{align}
with $|L_b\eta|\leq C (|\nabla^2\eta|+|x|^{-1}|\nabla\eta|)$.

As a first simplification step, we deal with the contributions on the unit scale:
Using the monotonicity of $h$, we infer
\begin{align}
\label{eq:bulk_a}
\begin{split}
&\sum\limits_{j=0}^{m} \left(\tau^{m-j} \|e^{h(-\ln(|x|))} x_{n+1}^{\frac{b}{2}} (1+\overline{h})^{\frac{m-j}{2}}|x|^{-2m+1+2j}|\nabla \eta||\nabla u_j|\|_{L^2(I_4)}\right.\\
&\qquad + \left. \tau^{m-j} \|e^{h(-\ln(|x|))} x_{n+1}^{\frac{b}{2}} (1+\overline{h})^{\frac{m-j}{2}}|x|^{-2m+1+2j}||L_b \eta| u_j|\|_{L^2( I_4)}\right)\\
&\leq C e^{ h(-\ln 3)}\tau^{\frac{3m}{2}}\sum\limits_{j=0}^{m}\Big(\|x_{n+1}^{\frac{b}{2}} u_j\|_{L^2(\tilde I_4)}+\|x_{n+1}^{\frac{b}{2}} \nabla u_j\|_{L^2(\tilde I_4)}\Big),
\end{split}
\end{align}
where $\tilde I_4=B^+_{4}\backslash B^+_{5/2}$.
Estimating the terms on the left hand side of \eqref{eq:bulk_0} from below by
\begin{align*}
Ce^{h(-\ln\frac 5 2)} \sum\limits_{j=0}^{m}  \Big(\|x_{n+1}^{\frac{b}{2}} u_j \|_{L^2(I_3)}+\|x_{n+1}^{\frac{b}{2}} \nabla u_j \|_{L^2(I_3)}\Big),
\end{align*}
and relying on the monotonicity of $h$, by choosing $\tau>\tau_0$ sufficiently large, we can absorb the contribution \eqref{eq:bulk_a} into the left hand side of \eqref{eq:bulk_0} (this yields a dependence of $\tau$ on $u$, but only on the unit scale).

As a consequence, we are left with the estimate
\begin{align}
\label{eq:bulk_b}
\begin{split}
&\sum\limits_{j=0}^{m} \left(\tau^{m+1-j}\|e^{h(-\ln(|x|))} x_{n+1}^{\frac{b}{2}}(1+ \overline{h})^{\frac{m+1-j}{2}}|x|^{-2m-1+2j} u_j\|_{L^2(I_2)}\right. \\
&\qquad+ \left.\tau^{m-j}\|e^{h(-\ln(|x|))} x_{n+1}^{\frac{b}{2}}(1+ \overline{h})^{\frac{m+1-j}{2}}|x|^{-2m+2j} \nabla u_j\|_{L^2(I_2)}\right) \\
&\leq 
C \sum\limits_{j=0}^{m} \left(\tau^{m-j} \|e^{h(-\ln(|x|))} x_{n+1}^{\frac{b}{2}} (1+\overline{h})^{\frac{m-j}{2}}|x|^{-2m+1+2j}|\nabla u_j||\nabla \eta|\|_{L^2(I_1 )}\right.\\
&\qquad \qquad +\left. \tau^{m-j} \|e^{h(-\ln(|x|))} x_{n+1}^{\frac{b}{2}} (1+\overline{h})^{\frac{m-j}{2}}|x|^{-2m+1+2j}|u_j||L_b \eta|\|_{L^2(I_1 )}\right),
\end{split}
\end{align}
where $\tau$ has been fixed in the previous step.
Using the monotonicity of $h$, the bound of $(1+\overline h)$ and the estimates on the derivatives of $\eta$ in $I_1$, we obtain
\begin{align}
\label{eq:aux_Carl_a}
\begin{split}
e^{h(-\ln 2r)}\sum_{j=0}^m&  \left(r^{-2m-1+2j}\|x_{n+1}^{\frac b 2}u_j\|_{L^2(I_2)}
+r^{-2m+2j}\|x_{n+1}^{\frac b 2}\nabla u_j\|_{L^2(I_2)}\right)\\
&\leq Ce^{h(-\ln \frac r 8)} \sum_{j=0}^m  \left(r^{-2m-1+2j}\|x_{n+1}^{\frac b 2}u_j\|_{L^2(I_1)}
+r^{-2m+2j}\|x_{n+1}^{\frac b 2}\nabla u_j\|_{L^2(I_1)}\right).
\end{split}
\end{align}
We observe that the difference $|h(-\ln \frac r 8)-h(-\ln 2r)|$ is bounded independently of $r>0$, since
\begin{align*}
|h(-\ln \frac r 8)-h(-\ln 2r)|=\left|\left(-\frac{h'(-\ln\xi r)}{\xi r}\right)\left(-\frac{15r}{8}\right)\right|,
\end{align*}
where $\xi\in(\frac 1 8,2)$, and $h'\in (C^{-1}\tau, C\tau)$.

Thus, dividing \eqref{eq:aux_Carl_a} by $e^{h(-\ln(2r))}$ and adding to both sides 
\[\sum_{j=0}^m\left(r^{-2m-1+2j}\|x_{n+1}^{\frac b 2}u_j\|_{L^2(B_{r/2}^+)}+r^{-2m+2j}\|x_{n+1}^{\frac b 2}\nabla u_j\|_{L^2(B_{r/2}^+)}\right),\]
we obtain
\begin{align*}
\sum_{j=0}^m& \left(r^{-2m-1+2j}\|x_{n+1}^{\frac b 2}u_j\|_{L^2(B^+_{2r})}
+r^{-2m+2j}\|x_{n+1}^{\frac b 2}\nabla u_j\|_{L^2(B^+_{2r})}\right)\\
&\leq C\sum_{j=0}^m\left(r^{-2m-1+2j}\|x_{n+1}^{\frac b 2}u_j\|_{L^2(B^+_{r/2})}
+r^{-2m+2j}\|x_{n+1}^{\frac b 2}\nabla u_j\|_{L^2(B^+_{r/2})}\right).
\end{align*}

\textit{Step 2b:  Gradient potentials.}
The proof is similar to the one in Step 2a but instead of the Carleman estimate from Proposition \ref{prop:syst_Carl}, we here use the one from Proposition ~\ref{prop:syst_Carl_1}. 
After Step 1b, estimate \eqref{eq:Carl_system_1} becomes
\begin{align*}
& \sum\limits_{j=0}^{m}\sum_{k=0}^{j} \left( \tau^{m+1-j} \|e^{h(-\ln(|x|))} x_{n+1}^{\frac{b}{2}}(1+\overline{h})^{\frac{m+1-k}{2}} |x|^{-2m-1+j+k} (\nabla')^{j-k} {u}_{k} \|_{L^2(I_2\cup I_3)} \right.\\
&  \qquad\qquad+  \left.  \tau^{m-j} \|e^{h(-\ln(|x|))} x_{n+1}^{\frac{b}{2}} (1+ \overline{h})^{\frac{m+1-k}{2}} |x|^{-2m+j+k} \nabla (\nabla')^{j-k} {u}_{k} \|_{L^2(I_2\cup I_3)}  \right)\\
&\leq C  \sum\limits_{j=0}^{m} \sum\limits_{k=0}^{j}\sum_{i=0}^{j-k}\Big(
\tau^{m-j}\|e^{h(-\ln(|x|))}  x_{n+1}^{\frac{b}{2}}(1+\overline{h})^{\frac{m-k}{2}}|x|^{-2m+1+j+k} |\nabla^iL_b\eta||(\nabla')^{j-k-i} u_{k}|\|_{L^2(I_2\cup I_4)}
\\
&\qquad +
\tau^{m-j}\sum_{\ell=0}^{i}\|e^{h(-\ln(|x|))}  x_{n+1}^{\frac{b}{2}}(1+\overline{h})^{\frac{m-k}{2}}|x|^{-2m+1+j+k} |\nabla^{i+1-\ell}\eta||\nabla^\ell a||\nabla(\nabla')^{j-k-i} u_{k}| \|_{L^2(I_2\cup I_4)}\Big).
\end{align*}

Considering the bounds for derivatives of $\eta$ and the metric $a$ (i.e., $|\nabla^\ell a|\leq \tilde{C}\delta |x|^{-\ell}$), and repeating the same arguments as in Step 2a, we arrive at
\begin{align*}
& \sum\limits_{j=0}^{m}\sum_{k=0}^{j} \left( r^{-2m-1+j+k} \| x_{n+1}^{\frac{b}{2}} (\nabla')^{j-k} {u}_{k} \|_{L^2(B^+_{2r})} + r^{-2m+j+k} \| x_{n+1}^{\frac{b}{2}} \nabla(\nabla')^{j-k} {u}_{k} \|_{L^2(B^+_{2r})}\right)\\
&\leq C\sum\limits_{j=0}^{m}\sum_{k=0}^{j} \left( r^{-2m-1+j+k} \| x_{n+1}^{\frac{b}{2}} (\nabla')^{j-k} {u}_{k} \|_{L^2(B^+_{r/2})} + r^{-2m+j+k} \| x_{n+1}^{\frac{b}{2}} \nabla(\nabla')^{j-k} {u}_{k} \|_{L^2(B^+_{r/2})}\right).
\end{align*}

\textit{Step 3: Caccioppoli's inequality.}
It remains to control the gradient terms of the right hand side. 
We can apply Lemma \ref{lem:Cacciop} to $u_k$ with $f=u_{k+1}$, $g=0$ if $k\in\{0,\dots,m-1\}$ and $f=0$, $|g|\leq\sum_{j=0}^m |q_j||(\nabla')^ju_0|$ if $k=m$.

If we just consider the lowest order potentials (i.e. where in the bounds for $|g|$ only $q_0$ is needed), tangential derivatives are not necessary and after summing over $k$ with suitable factors we arrive at
\begin{align}
\label{eq:Cacciop_crit}
\begin{split}
 \sum\limits_{k=0}^{m}r^{-2m+2k} \| x_{n+1}^{\frac{b}{2}} \nabla {u}_{k} \|_{L^2(B^+_{r/2})}
\leq C\Big(\sum\limits_{k=0}^{m} r^{-2m+2k-1} \| x_{n+1}^{\frac{b}{2}}  {u}_{k} \|_{L^2(B^+_{r})}+ \Big(\int\limits_{B'_r} |q_0||u_0| |u_m| dx'\Big)^{\frac 1 2}\Big).
\end{split}
\end{align}
If we also consider gradient potentials (i.e. where the full bound $|g|\leq\sum_{j=0}^m |q_j||(\nabla')^ju_0|$ is needed), a similar estimates holds after considering in \eqref{eq:Cacciop} tangential derivatives up to the order $m-k$:
\begin{align}
\label{eq:Cacciop_crit_ii}
\begin{split}
& \sum\limits_{j=0}^{m}\sum_{k=0}^{j}r^{-2m+j+k} \| x_{n+1}^{\frac{b}{2}} \nabla(\nabla')^{j-k} {u}_{k} \|_{L^2(B^+_{r/2})}\\
&\leq C\Bigg(\sum\limits_{j=0}^{m}\sum_{k=0}^{j} r^{-2m+j+k-1} \| x_{n+1}^{\frac{b}{2}} (\nabla')^{j-k} {u}_{k} \|_{L^2(B^+_{r})}+ \sum_{j=0}^m \Big(\int\limits_{B'_r} |q_j||(\nabla')^ju_0| |u_m| dx'\Big)^{\frac 1 2}\Bigg).
\end{split}
\end{align}

The boundary terms can be controlled as follows: We first notice that
\begin{align*}
\int\limits_{B'_r\backslash B'_{r/2}} |q_j| |(\nabla')^ju_0| | u_m|dx'
\leq Cr^{-2m+j+\varepsilon} \|(\nabla')^j u_0\|_{L^2(B'_r)}\|u_m\|_{L^2(B'_r)},
\end{align*} 
and
\begin{align*}
\|(\nabla')^j u_0\|_{L^2(B'_r)}
&\leq \|u_0\|_{H^{j}(B_{2r}')}
\leq C \|u_0\chi \|_{H^{j}(\R^n)}
\leq C \|u_0\chi\|_{L^2(\R^n)}^{1-\frac{j}{2\gamma}}\|u_0\chi\|_{H^{2\gamma}(\R^n)}^{\frac{j}{2\gamma}}\\
&\leq \|u_0\|_{L^2(B_{4r}')}^{1-\frac{j}{2\gamma}} \|u_0\|_{H^{2\gamma}(B_{4r}')}^{\frac{j}{2\gamma}},
\end{align*}
with $\chi$ a suitable cut-off function and $\gamma=\frac{1+2m-b}{2}$. 
Since $\lim\limits_{r\rightarrow 0} r^{-\ell} \|u_0\|_{L^2(B_{4r}')} = 0$  for any $\ell \in \N$, 
given any $\epsilon>0$ and $\ell>\ell_0$ there is a radius $r_{u_0}>0$ such that if $r\in(0,r_{u_0})$
\begin{align*}
\int\limits_{B'_r\backslash B'_{r/2}} |q_j| |(\nabla')^ju_0| |u_m|dx'
\leq C \epsilon r^\ell.
\end{align*} 
Therefore
\begin{align*}
\int\limits_{B'_r} |q_j| |(\nabla')^ju_0| | u_m|dx'&
\leq \sum_{k=0}^\infty \int\limits_{B'_{r/2^k}\backslash B'_{r/2^{k+1}}} |q_j| |(\nabla')^ju_0| |u_m|dx'\leq \sum_{k=0}^\infty C \epsilon \Big(\frac{r}{2^k}\Big)^\ell\leq C\epsilon r^\ell.
\end{align*} 
On the other hand, $\sum\limits_{j=0}^{m}r^{-2m+2j}\|x_{n+1}^{\frac b 2}u_j\|_{L^2(I_1)}$ only vanishes of finite order, so choosing $\epsilon$ sufficiently small, the boundary term can be absorbed  into the bulk terms provided $r\in (0,r_{u_0})\cap (r_{\ell}, 2^{m}r_{\ell})$ for some $\ell\in\N$.
Observe that if $q_0\in L^\infty$ and $q_j=0 $ for $j\geq 1$, then the boundary term can be absorbed directly by Young's inequality and 
the estimate is valid for $r\in(0,r_0)$ with $r_0$ independent of $u_0$ (we refer to Lemma 5.1 in \cite{RW18} for the details on this).

Hence, \eqref{eq:Cacciop_crit} becomes
\begin{align*}
\sum\limits_{j=0}^{m}r^{-2m+2j}\|x_{n+1}^{\frac b 2}\nabla u_j\|_{L^2(B^+_{r/2})}
\leq \sum\limits_{j=0}^{m}r^{-2m+2j-1}\|x_{n+1}^{\frac b 2} u_j\|_{L^2(B^+_r)}
\end{align*} 
and \eqref{eq:Cacciop_crit_ii} turns into
\begin{align}
\label{eq:grad_pot_iter}
& \sum\limits_{j=0}^{m}\sum_{k=0}^{j}r^{-2m+j+k} \| x_{n+1}^{\frac{b}{2}} \nabla(\nabla')^{j-k} {u}_{k} \|_{L^2(B^+_{r/2})}\leq \sum\limits_{j=0}^{m}\sum_{k=0}^{j} r^{-2m+j+k-1} \| x_{n+1}^{\frac{b}{2}} (\nabla')^{j-k} {u}_{k} \|_{L^2(B^+_{r})}.
\end{align}
In order to deal with the remaining derivatives on the left hand side in \eqref{eq:grad_pot_iter}, we notice that
\begin{align*}
&\sum\limits_{j=0}^{m}\sum_{k=0}^{j} r^{-2m+j+k-1} \| x_{n+1}^{\frac{b}{2}} (\nabla')^{j-k} \tilde{u}_{k} \|_{L^2(B^+_{r})}\\
&\leq \sum\limits_{j=0}^{m} r^{-2m+2j-1} \| x_{n+1}^{\frac{b}{2}}  \tilde{u}_{k} \|_{L^2(B^+_{r})}
+\sum\limits_{j=0}^{m-1}\sum_{k=0}^{j} r^{-2m+j+k} \| x_{n+1}^{\frac{b}{2}}\nabla (\nabla')^{j-k} \tilde{u}_{k} \|_{L^2(B^+_{r})}
\end{align*}
and iterate the Caccioppoli estimate
with starting radius $r=2^{-m}\tilde r$. 
\end{proof}

With the doubling property in hand, we apply a blow-up argument reducing the strong unique continuation property to the weak unique continuation property. To this end, we introduce the following rescaled functions:

\begin{align}
\label{eq:rescaled}
u_{\sigma,j}(x):= \frac{\sigma^{-2(m-j)}u_j(\sigma x)}{\sum\limits_{k=0}^{m} \sigma^{-\frac{n+1}{2}-\frac{b}{2}-2(m-k)}\|x_{n+1}^{\frac{b}{2}} u_k\|_{L^2(B_{\sigma}^+)}}.
\end{align}

We exploit the previous compactness arguments to pass to the blow-up limit $\sigma \rightarrow 0$ which leads to a boundary weak unique continuation formulation of the blown-up system:

\begin{prop}
\label{prop:limit}
Let $u_j\in H^1_{loc}(\R^{n+1}_+, x_{n+1}^b)$ for $j\in\{0,\dots,m\}$ be weak solutions of the system \eqref{eq:syst_err}, \eqref{eq:syst_err_boundary} with $f_0,\dots,f_m=0$ satisfying the conditions from \hyperref[cond:C]{(C)}.
Assume also that the tangential restrictions $u_j(x',0)$ vanish of infinite order at $x_0=0$ and that there exist some $j\in \{0,\dots,m\}$, a subsequence $r_{\ell}\rightarrow 0$ and a constant $k_0\in \N$ such that
\begin{align*}
\lim\limits_{\ell \rightarrow \infty} r_{\ell}^{-k_0} \|x_{n+1}^{\frac{b}{2}} u_j\|_{L^2(B_{r_{\ell}}^+)} \geq C_0 >0.
\end{align*}
Let $u_{\sigma,j}$ be the rescaled functions defined by \eqref{eq:rescaled} and let $\{r_{\ell}\}$ denote the sequence of radii from \eqref{eq:van_ord}. 
Then, along a subsequence $\{\sigma_{\ell}\}_{\ell \in \N} \subset \{2r_{\ell}\}_{\ell \in \N}$ with $\sigma_l \rightarrow 0$ we have $u_{\sigma_l,j} \rightarrow u_{0,j}$ strongly in $L^2(B_4^+, x_{n+1}^b)$, where the functions $u_{0,j}$ are weak solutions to the following elliptic system:
\begin{align}
\label{eq:syst_blow_up}
\begin{split}
\D_b u_{0,m} & = 0 \mbox{ in } B_1^+,\\
\D_b u_{0,j} & = u_{0,j+1} \mbox{ in } B_1^+ \mbox{ for all } j \in\{0,\dots,m-1\},\\
\lim\limits_{x_{n+1}\rightarrow 0} x_{n+1}^b \p_{x_{n+1}} u_{0,j} & = 0 \mbox{ on } B_1' \mbox{ for all } j \in\{0,\dots,m\},
\end{split}
\end{align}
with $\Delta_b=x_{n+1}^{-b}\nabla \cdot x_{n+1}^b\nabla$.
Moreover, for all $j\in\{0,\dots,m\}$ we have
\begin{align*}
u_{0,j} = 0 \mbox{ on } B_1',
\end{align*}
and
\begin{align}
\label{eq:normalized}
\sum\limits_{j=0}^{m}\|x_{n+1}^{\frac{b}{2}} u_{0,j}\|_{L^2(B_1^+)} = 1.
\end{align}
\end{prop}

\begin{proof}
We first note  that the functions $u_{\sigma,j}$ are constructed in such a way that
\begin{align*}
\sum_{j=0}^m\|x_{n+1}^{\frac b 2}u_{\sigma,j}\|_{L^2(B^+_1)}=1.
\end{align*}
From Proposition ~\ref{prop:doubling}, after rescaling and for
$\sigma\in\{2r_{\ell}\}_{\ell \in \N}$, we obtain
\begin{align*}
\sum_{j=0}^m\big(
\|x_{n+1}^{\frac b 2} u_{\sigma,j}\|_{L^2(B_4^+)}+
\|x_{n+1}^{\frac b 2} \nabla u_{\sigma,j}\|_{L^2(B_4^+)}\big)
\leq C\sum_{j=0}^m \|x_{n+1}^{\frac b 2} u_{\sigma,j}\|_{L^2(B_1^+)}=C.
\end{align*}
We stress that there is no problem with the dependence of the radius in the doubling estimate from Proposition \ref{prop:doubling} on $u_0$, as we first apply this to the fixed functions $u_j$ and then rescale (which implies that the result holds uniformly in $\sigma$ for the whole family $u_{\sigma, j}$).
By Rellich's compactness theorem, there exist a subsequence $\{\sigma_\ell\}_{\ell \in \N} \subset \{2r_{\ell}\}_{\ell \in \N}$ with $\sigma_\ell\to 0$  and functions $u_{0,j}\in H^1(x_{n+1}^b, B^+_4)$ such that $u_{\sigma_\ell,j}\to u_{0, j}$ strongly in $L^2(B_4^+, x_{n+1}^b)$ and weakly in $H^1(B_4^+, x_{n+1}^b)$ and the normalization \eqref{eq:normalized} holds.
In addition, since the embedding $H^1(B^+_4, x_{n+1}^b)\hookrightarrow L^2(B'_3)$ is compact and $H^{\frac{1-b}{2}}(B'_3)$ is compactly embedded in $L^2(B'_3)$, up to a redefinition of the subsequence, $u_{\sigma_\ell,j}\to u_{0,j}$ strongly in $L^2(B'_{3})$ and weakly in $H^{\frac{1-b}{2}}(B'_3)$.

The functions $u_{\sigma,j}$ satisfy weakly the same system \eqref{eq:syst_err}, \eqref{eq:syst_err_boundary} as the original functions $u_j$ (again with $f_0, \dots, f_m=0$) however with a rescaled metric and potentials $a_\sigma(x)=a(\sigma x)$ and
\[|g_\sigma(x)|=\sigma^{1-b}|g(\sigma x)|\leq \sum_{j=0}^m \sigma^{1-b+2m-j}|q_j(\sigma x)||(\nabla')^ju_{\sigma, 0}(x',0)|.\]
Hence,
\begin{align*}
&\int_{B^+_2}x_{n+1}^b\nabla \varphi \cdot a(\sigma x)\nabla u_{\sigma, j} =-\int_{B^+_2}x_{n+1}^b u_{\sigma, j+1}\varphi, 
 \qquad j\in\{0,\dots,m-1\},
\\&\int_{B^+_2}x_{n+1}^b\nabla \varphi\cdot a(\sigma x)\nabla u_{\sigma, m}=\int_{B'_2}g_\sigma \varphi,
\end{align*}
for any $\varphi\in C^1_c(B^+_3)\cap C_c(B'_3)$.
As a result, in the limit  $\sigma_\ell\to 0$
\begin{align*}
&\int_{B^+_2}x_{n+1}^b\nabla u_{0, j}\cdot\nabla \varphi=-\int_{B^+_2}x_{n+1}^b u_{0, j+1}\varphi, \qquad j\in\{0,\dots,m-1\},
\\
&\int_{B^+_2}x_{n+1}^b\nabla u_{0, m}\cdot\nabla \varphi=0.
\end{align*}
Here we have used that by the normalising condition \hyperref[cond:a3]{(A3)} the metric satisfies $a^{ij}(\sigma x)\to \delta^{ij}$ as $\sigma \rightarrow 0$ and that the boundary integrals vanish of infinite order as we proved in Step 3 of the proof of Proposition ~\ref{prop:doubling}.
This shows that the functions $u_{0,j}$ are indeed weak solutions to the system in the statement.

Finally, we prove that the functions $u_{0,j}$, $j\in\{0,\dots,m\}$, vanish on $B'_1$.  Indeed, 
\begin{align*}
\|u_{\sigma,j}\|_{L^2(B'_1)}=\frac{\sigma^{-2(m-j)-\frac n 2}\|u_j\|_{L^2(B'_\sigma)}}{\sum\limits_{k=0}^{m} \sigma^{-\frac{n+1}{2}-\frac{b}{2}-2(m-k)}\|x_{n+1}^{\frac{b}{2}} u_k\|_{L^2(B_{\sigma}^+)}},
\end{align*}
and whereas the numerator vanishes of infinite order, by our assumption \eqref{eq:van_ord}, the denominator vanishes of only finite order. 
\end{proof}

\subsection{Strong unique continuation} \label{sec:SUCP_inf_order}
Here we deduce the strong unique continuation property for solutions which vanish of infinite order in tangential \emph{and} normal directions.
The proof relies on  the associated Carleman estimates from Proposition ~\ref{prop:syst_Carl_1}.

\begin{prop}
\label{prop:SUCP_inf_order}
Let $u_j\in H^1_{loc}(\R^{n+1}, x_{n+1}^b)$ for $j\in\{0,\dots,m\}$ be weak solutions of the system \eqref{eq:syst_err}, \eqref{eq:syst_err_boundary} with $f_0,\dots,f_m=0$ satisfying the conditions from \hyperref[cond:C]{(C)}. Assume further that for all $j\in \{0,\dots,m\}$ and all $k\in \N$
\begin{align*}
\lim\limits_{r\rightarrow 0} r^{-k}\|x_{n+1}^{\frac{b}{2}} u_j\|_{L^2(B_r^+)} = 0.
\end{align*}
Then, $u\equiv 0$ in $B_1^+$.
\end{prop} 

\begin{proof}
Consider the functions $\tilde{u}_j:= \eta_{\epsilon} u_j$, where $\eta_{\epsilon}$ is a smooth, radial cut-off function which satisfies the following bounds:
\begin{align}
\label{eq:eta_bounds}
\begin{split}
&\eta_{\epsilon}(x) = 1 \mbox{ for } |x|\in (2\epsilon,1),\
\supp(\eta_{\epsilon}) \subset \{x\in \R^{n+1}_+: \ |x|\in(\epsilon,2)\},\
|\eta_{\epsilon}(x)| \leq 1, \\
&|\nabla^\alpha \eta_{\epsilon}(x)| \leq C\epsilon^{-|\alpha|}\mbox{ for } |x|\in (\epsilon, 2\epsilon) \mbox{ and } |\alpha|\leq m+2,\ \alpha \in \N^{n},\\
& 
|\nabla^\alpha \eta_{\epsilon}(x)| \leq C\mbox{ for } |x|\in (1, 2) \mbox{ and } |\alpha|\leq m+2,
\end{split}
\end{align} 
where $\epsilon\ll1$.
The functions $\tilde u_j$ satisfy the system from Proposition ~\ref{prop:syst_Carl} with $f_j=2\nabla \eta_\epsilon\cdot a\nabla u_j+u_jL_b\eta_\epsilon$.

We insert the functions $\tilde{u}_j$ into the Carleman estimate \eqref{eq:Carl_system_1}.
Notice that we can pass to the limit $\epsilon \rightarrow 0$ 
by virtue of the infinite rate of vanishing of $u_j$.
Arguing as in Step 1 of the proof of Proposition ~\ref{prop:doubling}, we can drop the boundary contributions.
Therefore  \eqref{eq:Carl_system_1} turns into
\begin{align*}
& \sum\limits_{j=0}^{m}\sum_{k=0}^{j} \tau^{m+1-j} \|e^{h(-\ln(|x|))} x_{n+1}^{\frac{b}{2}}(1+\overline{h})^{\frac{m+1-k}{2}} |x|^{-2m-1+j+k} (\nabla')^{j-k} {u}_{k} \|_{L^2(B_1^+)} \\
&\leq C  \sum\limits_{j=0}^{m} \sum\limits_{k=0}^{j}\sum_{i=0}^{j-k}\Big(
\tau^{m-j}\|e^{h(-\ln(|x|))}  x_{n+1}^{\frac{b}{2}}(1+\overline{h})^{\frac{m-k}{2}}|x|^{-2m+1+j+k} |\nabla^iL_b\eta_0||(\nabla')^{j-k-i} u_{k}|\|_{L^2(B^+_2)}
\\
&\qquad +
\tau^{m-j}\sum_{\ell=0}^{i}\|e^{h(-\ln(|x|))}  x_{n+1}^{\frac{b}{2}}(1+\overline{h})^{\frac{m-k}{2}}|x|^{-2m+1+j+k} |\nabla^{i+1-\ell}\eta_0||\nabla^\ell a||\nabla(\nabla')^{j-k-i} u_{k}| \|_{L^2(B^+_2)}\Big).\end{align*}
Finally, passing to the limit $\tau \rightarrow \infty$, we obtain $u_j=0$ in $B^+_1$.

\end{proof}

\section{Weak Unique Continuation}
\label{sec:WUCP}

In this section we consider the weak unique continuation property for the equations at hand. In spite of weak unique continuation results for the fractional Laplacian already existing in the literature (see in particular \cite{Seo}), both our argument and our result contain novel aspects: In contrast to the weak unique continuation results from Seo \cite{Seo} our result is a \emph{localised} unique continuation result (as we do not need the validity of the equation $L^{\gamma} u = q u$ in $\R^n$), and hence in particular it is formulated for a local equation (instead of working with the global fractional Laplacian).

\begin{prop}
\label{prop:WUCP}
Let $u_j\in H^1_{loc}(B_1^+, x_{n+1}^b)$ for $j\in\{0,\dots,m\}$  be weak solutions of the system \eqref{eq:syst_err}, \eqref{eq:syst_err_boundary} in $B^+_1$ with $f_0,\dots,f_m=0$, $g=0$ and the metric $a$ of a block form \eqref{eq:block} where $\tilde a$ satisfies the conditions \hyperref[cond:a1]{(A1)}-\hyperref[cond:a3]{(A3)}.
Assume also that for all $j\in\{0,\dots,m\}$ the tangential restrictions $u_j(x',0)$ vanish on $B'_1$. Then, $u_j \equiv 0$ in $B_1^+$.
\end{prop}

\begin{proof}
We bootstrap the system by applying the weak unique continuation property for scalar equations: Indeed, by the weak unique continuation property of solutions of the fractional Laplacian (see \cite{Rue15} and \cite{FF14}) and regularity results from \cite{KRSIV}, we first infer that $u_m \equiv 0$ in $B_1^+$. Iteratively, this then also entails that $u_{j} \equiv 0$ in $B_1^+$ since, once $u_{j+1} \equiv 0$ in $B_1^+$, then $u_j$ satisfies the Caffarelli-Silvestre  equation with zero Dirichlet and (weighted) Neumann data. We iterate this until we reach $u_0$.
\end{proof}

\begin{rmk}
We remark that an argument for the WUCP had already been given by Riesz \cite{R38} (relying on certain regularity conditions, see the discussion in Remark 4.2 in \cite{GSU16}). Using a Kelvin transform he reduced it to the situation with data vanishing in the exterior of a domain. An argument of a related flavour  for a much larger class of pseudodifferential operators was also used in \cite{RS17} (see also \cite{Isakov}).
\end{rmk}

\begin{rmk}
\label{rmk:lopatinski}
We remark that the (weak) unique continuation property requires the Lopatinskii condition to hold. If this is violated even if ``formally'' there are sufficiently many boundary conditions prescribed, one will in general not be able to infer the vanishing of $u$. 
This is for instance the case for problem
\begin{align*}
&\D^2 u = 0 \mbox{ in } B_1^+ \subset \R^2,\\
&u = 0 , \p_{y}u = 0, \ \p_{x} u = 0, \p_{xx}u=0 \mbox{ on } B_1'.
\end{align*}
By simply invoking counting arguments these boundary conditions should yield an overdetermined system. They however do not (the function $w(x,y)=y^2 x$ is a non-trivial solution), as the Lopatinskii condition is not satisfied.
\end{rmk}

As a consequence of the localised formulation of our weak unique continuation property, it for instance applies to settings which arise in inverse problems \cite{RS17, GSU16, GRSU18}.
This allows us to prove the antilocality of the fractional Laplacian for any order $\gamma>0$ with $\gamma \notin \N$, i.e. it allows us to prove Proposition ~\ref{prop:antiloc}, which we postpone to the next section.

\section{Proofs of the Unique Continuation Results for the Fractional Laplacian}

\label{sec:proofs}

In this section, we rely on the connection between the systems representations for the higher order fractional Laplacian (see Proposition ~\ref{prop:H2gamma1a} as well as Propositions ~\ref{prop:H2gamma}, ~\ref{prop:H2gamma1} in the Appendix) and -- building on the previous compactness results -- present the proofs of Theorems ~\ref{prop:SUCP}-\ref{prop:MUCP} and of Propositions ~\ref{prop:antiloc}, ~\ref{prop:Runge}.

\subsection{Proofs of strong unique continuation properties for the fractional Laplacian}
We begin by proving Theorems \ref{prop:SUCP}-\ref{prop:SUCP_var}:

\begin{proof}[Proofs of Theorems ~\ref{prop:SUCP}, ~\ref{prop:SUCP_higher} and ~\ref{prop:SUCP_var}]
We seek to reduce the strong unique continuation properties  for the fractional Laplacian to the previously deduced results on the systems case.  We invoke Proposition ~\ref{prop:H2gamma1a} and rewrite the problem as a system of the form \eqref{eq:WUCP_syst1}, where $f=u$, $m=\floor{\gamma}$ and $b=1-2\gamma+2\floor{\gamma}$.
We seek to apply a combination of Propositions ~\ref{prop:limit}, \ref{prop:SUCP_inf_order} and ~\ref{prop:WUCP}. To this end, we have to show that the functions $u_j(x',0)=L^j u(x')$ 
with $j\in\{1,\dots, \floor{\gamma}\}$ vanish of infinite order in the tangential directions on the boundary. By assumption, we have that  the function $u_0(x',0)=u(x')$ vanishes of infinite order at $x'_0=0$ in the  tangential directions. In order to obtain the desired infinite order of vanishing of $u_j$ in the tangential directions on the boundary, we use an interpolation argument: Let $\eta$ be a smooth cut-off function which is equal to one on $B_{r}'$ and which is supported in $B_{4r}'$. Then,
\begin{align*}
\|L^j u\|_{L^2(B_{r}')}
&\leq \|u\|_{H^{2j}(B_{r}')}
\leq C \|u\eta \|_{H^{2j}(\R^n)}
\leq C \|u\eta\|_{L^2(\R^n)}^{1-\frac{j}{\gamma}}\|u\eta\|_{H^{2\gamma}(\R^n)}^{\frac{j}{\gamma}}\\
&\leq \|u\|_{L^2(B_{4r}')}^{1-\frac{j}{\gamma}} \|u\|_{H^{2\gamma}(B_{4r}')}^{\frac{j}{\gamma}}.
\end{align*}
Since $\lim\limits_{r\rightarrow 0} r^{-\ell} \|u\|_{L^2(B_{4r}')} = 0$ for any $\ell \in \N$ and as $u\in H^{2\gamma}(B_1')$, this implies that the same holds for $\|L^j u\|_{L^2(B_{r}')}$ and thus for $\|u_j(\cdot,0)\|_{L^2(B_{r}')}$ with $j\in\{0,\dots, \floor{\gamma}\}$.
Moreover, due to the previous identification of $b$ and $m$ in terms of $\gamma$ and $\floor{\gamma}$, the conditions from \hyperref[cond:C]{(C)} are satisfied. 
Hence, if \eqref{eq:van_ord} holds, the blow-up argument from Proposition ~\ref{prop:limit} and subsequently the weak unique continuation result from Proposition ~\ref{prop:WUCP} are applicable. Alternatively, we invoke Proposition \ref{prop:SUCP_inf_order}.

As a consequence, the functions $u_j$ for $j=0,\dots,m$, (and thus in particular also the function $u$) vanish in $B_1^+$. Using that the equation for the generalised Caffarelli-Silvestre extension holds globally, the vanishing of $u$ on $B_1^+$ propagates through the upper half plane $\R^{n+1}_+$: Indeed, by the weak unique continuation property for uniformly elliptic equation and by \eqref{eq:WUCP_syst1} we infer $u_m\equiv 0$ in $\R^{n+1}_+$. Plugging this into the equation for $u_{m-1}$ and again using the weak unique continuation property for solutions to uniformly elliptic equations in the upper half plane implies also $u_{m-1} \equiv 0$ in $\R^{n+1}_+$. Iterating this further leads to $u_j \equiv 0$ in $\R^{n+1}_+$, whence $u \equiv 0$ in $\R^n$. This concludes the argument.
\end{proof}

\subsection{Proof of unique continuation from measurable sets}
In this section we prove Theorem ~\ref{prop:MUCP}  by reducing it to the weak unique continuation property for the generalised Caffarelli-Silvestre extension.

By the representations from Proposition ~\ref{prop:H2gamma1a} (see also Propositions ~\ref{prop:H2gamma}, ~\ref{prop:H2gamma1}) we rewrite the problem of Theorem ~\ref{prop:MUCP} as a system of the form \eqref{eq:WUCP_syst1}, where $f=u$, $m=\floor{\gamma}$ and $b=1-2\gamma+2\floor{\gamma}$. Notice that by Proposition ~\ref{prop:H2gamma1a} and the assumptions in Theorem \ref{prop:MUCP} we  have that   $|L^\gamma f(x')|\leq |q(x')||u_0(x',0)|$ with $q\in L^\infty(\R^n)$ and $u_0(\cdot,0)\in H^{2\gamma}(\R^n)$.
By assumption, $u_0(x',0)=u(x')$ vanishes on a measurable set $E\subset\R^n\times\{0\}$ of density one at $x_0=0$.

Under these assumptions and supposing that \eqref{eq:van_ord} holds, we prove an analogous blow-up result as in Proposition ~\ref{prop:limit}:

\begin{prop}
\label{prop:limit_E}
Let $u_j$  with $j\in \{0,\dots,m\}$ be the functions from above and
let $u_{\sigma,j}$ with $j\in \{0,\dots,m\}$ be the associated rescaled functions defined in \eqref{eq:rescaled}. Suppose further that \eqref{eq:van_ord} holds.
Then, along a subsequence $\{\sigma_l\}_{l \in \N}\subset\{2r_\ell\}_{l \in \N}$ with $\sigma_l \rightarrow 0$ we have $u_{\sigma_l,j} \rightarrow u_{0,j}$ strongly in $L^2(B_4^+, x_{n+1}^b)$, where $u_{0,j}$ is a weak solution to the following elliptic system
\begin{align*}
\begin{split}
\D_b u_{0,m} & = 0 \mbox{ in } B_1^+,\\
\D_b u_{0,j} & = u_{0,j+1} \mbox{ in } B_1^+ \mbox{ for all } j \in\{0,\dots,m-1\},\\
\lim\limits_{x_{n+1}\rightarrow 0} x_{n+1}^b \p_{x_{n+1}} u_{0,j} & = 0 \mbox{ on } B_1' \mbox{ for all } j \in\{0,\dots,m\}.
\end{split}
\end{align*}
Moreover, for all $j\in\{0,\dots,m\}$ we have
\begin{align*}
u_{0,j} = 0 \mbox{ on } B_1',
\end{align*}
and
\begin{align*}
\sum\limits_{j=0}^{m}\|x_{n+1}^{\frac{b}{2}} u_{0,j}\|_{L^2(B_1^+)} = 1.
\end{align*}
\end{prop}

In order to obtain the properties of the blow-up limit, we deduce a smallness condition for the (not yet blown-up) function $u_0$ in tangential directions on the boundary. By virtue of an interpolation inequality, this will be inherited to all the (not yet blown-up) functions $u_j$ with $j\in\{0,\dots,m\}$ on the boundary.

\begin{lem}
\label{lem:small_MUCP}
Let $u_j$ with $j\in\{0,\dots,m\}$ be as in Proposition ~\ref{prop:limit_E}.
For any $\epsilon>0$, there exists a radius $r_0>0$ such that if $r\in (0,r_0)$
\begin{align*}
\|u_0\|_{L^2(B'_r)}\leq \epsilon \sum_{j=0}^m r^{2j-\frac{1+b}{2}}\|x_{n+1}^{\frac b 2}u_j\|_{L^2(B^+_r)}.
\end{align*}
\end{lem}

\begin{proof}
Since $x_0=0$ is a point of density one in $E\cap B'_4$, given $\delta>0$, there exists  a radius $r_\delta>0$ such that if $r\in(0, r_\delta)$
\begin{align*}
|B'_r\cap E^c|\leq \delta |B'_r|.
\end{align*}

On the other hand, using H\"older's inequality ($n>1-b$)
\begin{align}
\label{eq:small_thin}
\|u_0\|_{L^2(B'_r)}&=\|u_0\|_{L^2(B'_r\cap E^c)}
\leq |B'_r\cap E^c|^{\frac{1-b}{2n}}\|u_0\|_{L^\frac{2n}{n+1-b}(B'_r)}.
\end{align}

By Sobolev and trace inequalities
\begin{align*}
\|u_0\|_{L^\frac{2n}{n+1-b}(B'_r)}
&\leq C\|u_0\|_{H^{\frac{1-b}{2}}(B'_r)}\\
&\leq C (r^{-1}\|x_{n+1}^{\frac b 2} u_0\|_{L^2(B^+_{2r})}+\|x_{n+1}^{\frac b 2} \nabla u_0\|_{L^2(B^+_{2r})})\\
&\leq \sum_{j=0}^m r^{2j-1} \|x_{n+1}^{\frac b 2}u_j\|_{L^2(B^+_{2r})}+\sum_{j=0}^m r^{2j} \|x_{n+1}^{\frac b 2}\nabla u_j\|_{L^2(B^+_{2r})}.
\end{align*}
Now we use the estimate from Proposition ~\ref{prop:doubling}, where according with Remark ~\ref{rmk:doubling_bdd} no assumptions on the vanishing order of $u_j$ are necessary and it holds for $r\in(0,r_0)$ with $r_0$ independent of $u_0$:
\begin{align*}
\|u_0\|_{L^\frac{2n}{n+1-b}(B'_r)}
&\leq C \sum_{j=0}^m r^{2j-1} \|x_{n+1}^{\frac b 2}u_j\|_{L^2(B^+_{r})}.
\end{align*}
Therefore, by combining this with \eqref{eq:small_thin} and recalling the definition of $\delta>0$ from above, we obtain
\begin{align*}
\|u_0\|_{L^2(B'_r)}
\leq C\delta^{\frac{1-b}{2n}}\sum_{j=0}^m r^{2j-\frac{1+b}{2}} \|x_{n+1}^{\frac b 2}u_j\|_{L^2(B^+_{r})}.
\end{align*}
Choosing $\delta$ such that $C\delta^{\frac{1-b}{2n}}=\epsilon$, the result holds.

\end{proof}

\begin{proof}[Proof of Proposition ~\ref{prop:limit_E}]
The proof of Proposition \ref{prop:limit_E} follows along the same lines as the one of Proposition ~\ref{prop:limit} until the moment of proving $u_{0,j}|_{B'_1}=0$.
Here we use Lemma \ref{lem:small_MUCP} to obtain the same result:
Indeed, after rescaling it implies
\begin{align*}
\|u_{\sigma,0}\|_{L^2(B'_1)}\leq \epsilon \sum_{j=0}^m \|x_{n+1}^{\frac b 2}u_{\sigma,j}\|_{L^2(B^+_1)}=\epsilon.
\end{align*}
Therefore, in the limit $\sigma_\ell\to 0$,
\begin{align*}
\|u_{0,0}\|_{L^2(B'_1)}\leq \epsilon.
\end{align*}
Since this holds for any $\epsilon>0$, in particular for any sequence $\epsilon_k\to 0$, we infer $u_{0,0}|_{B'_1}=0$.

The proof of $u_{0,j}|_{B'_1}=0$ for $j=1,\dots,m$ relies on an interpolation result together with the previous bound: Considering a smooth cut-off function $\eta$ with $\eta = 1$ in $B_{r}'$ and $\supp(\eta) \subset B_{4r}'$, we obtain
\begin{align*}
\|u_j\|_{L^2(B'_r)}&\leq \|L^j u_0\|_{L^2(B_{r}')}
\leq C\|u_0\|_{H^{2j}(B_{r}')}
\leq C \|u_0\eta \|_{H^{2j}(\R^n)}\\
&\leq C \|u_0\eta\|_{L^2(\R^n)}^{1-\frac{j}{\gamma}}\|u_0\eta\|_{H^{2\gamma}(\R^n)}^{\frac{j}{\gamma}}
\leq C\|u_0\|_{L^2(B'_{4r})}^{1-\frac{j}{\gamma}} \|u_0\|_{H^{2\gamma}(B_{4r}')}^{\frac{j}{\gamma}}.
\end{align*}
By rescaling, we then also infer 
\begin{align*}
\|u_{\sigma,j}\|_{L^2(B'_r)}
\leq C\|u_{\sigma,0}\|_{L^2(B'_{4r})}^{1-\frac{j}{\gamma}} \|u_{\sigma,0}\|_{H^{2\gamma}(B_{4r}')}^{\frac{j}{\gamma}}.
\end{align*}

Thus, the smallness of $u_0$ and $u_{\sigma,0}$ also entails the smallness of $u_j$ and $u_{\sigma,j}$ on the boundary. The remainder of the argument follows analogously as in the proof of Proposition ~\ref{prop:limit}.
\end{proof}

\begin{proof}[Proof of Theorem ~\ref{prop:MUCP}]
Assuming that $\eqref{eq:van_ord}$ holds, the representations from Proposition ~\ref{prop:H2gamma1a} lead to Proposition ~\ref{prop:limit_E}  which reduces the problem to the weak unique continuation statement from Proposition ~\ref{prop:WUCP} from which we infer the desired result. If \eqref{eq:van_ord} fails, we directly apply Proposition \ref{prop:SUCP_inf_order}.
\end{proof}

\subsection{Applications of the unique continuation results}
We turn to the proof of the antilocality result. As above we emphasise that in this case, we do not assume the validity of an equation on the whole space $\R^n$. Nevertheless the antilocality of the fractional Laplacian entails the claimed strong rigidity property.

\begin{proof}[Proof of Proposition ~\ref{prop:antiloc}]
By Proposition ~\ref{prop:H2gamma1a} we  can consider the extension $u$ and  the functions $u_j=L_b^j u$ for $j\in\{0,\dots,\floor{\gamma}\}$, which solve a system of the form \eqref{eq:WUCP_syst1}.
Thus, if $f=0$ and $L^\gamma f=0$ on $B_1'$, \eqref{eq:WUCP_syst1} reduces to the setting in Section ~\ref{sec:WUCP}, whence we conclude that $u_j=0$ on $B_1^+$. 
 Since the nonlocal equation is assumed to hold in $\R^n$, the vanishing of $u_j$ can be propagated to the whole upper half space $\R^{n+1}_+$, whence we conclude that $u\equiv 0$ in $\R^n$.
\end{proof}

With the global unique continuation result of Proposition ~\ref{prop:antiloc} in hand, the proof of Proposition ~\ref{prop:Runge} follows by a duality argument and the Fredholm property of the fractional Schrödinger equation (see \cite{G15}). In particular, this is of similar flavour as a number of approximation properties which had been used to treat inverse problems for nonlocal equation in \cite{GSU16,GLX17,GRSU18}. 

We consider the fractional Schrödinger equation

\begin{align}
\label{eq:weak_form}
\begin{split}
L^{\gamma} u + qu & = 0 \mbox{ in } \Omega,\\
u & = f \mbox{ in } \R^n \setminus \Omega.
\end{split}
\end{align}
where $L$ is as in Proposition ~\ref{prop:Runge}.

Considering the bilinear form
\begin{align*}
B_{q}(w,v):=(L^{\gamma/2}w, L^{\gamma/2} v)_{\R^n} + (q w,v)_{\Omega},
\end{align*}
it is possible to show the well-posedness of the problem, provided zero is not a Dirichlet eigenvalue of the problem. This follows similarly as explained  for instance in \cite{GSU16}. 
In this setting, we define the associated Poisson operator as
\begin{align}
\label{eq:Poiss}
P_q : H^{\gamma}(\R^n \setminus \overline{\Omega})
&\rightarrow H^{\gamma}(\Omega),\ f \mapsto u_f,
\end{align}
where $u_f $ is a weak solution to \eqref{eq:weak_form}.
With this preparation, we address the proof of Proposition ~\ref{prop:Runge}:

\begin{proof}[Proof of Proposition ~\ref{prop:Runge}]
It suffices to prove that the set  
\begin{align*}
\mathcal{R}:=\{u = P_q f, \ f \in C^{\infty}_0(W)\}
\end{align*}
is dense in $L^2(\Omega)$, where $W\subset\tilde{\Omega}\setminus \Omega$ is any open subset. We can suppose without loss of generality that $B_1'\subset W$ by assupmtion. As in \cite{GSU16} we rely on the Hahn-Banach theorem: Assuming that $v\in L^2(\Omega)$ is such that $(P_q f, v)_{\Omega} = 0$ for all $f \in C_0^{\infty}(W)$, it suffices to show that $v=0$. In order to infer this, we note that by the well-posedness results for the inhomogeneous fractional Schrödinger equation, we may define $w$ to be a solution to 
\begin{align*}
(L^{\gamma} + q) w &= v \mbox{ in } \Omega,\\
w & = 0 \mbox{ in } \R^n \setminus \overline{\Omega}.
\end{align*}
Then, as in \cite{GSU16},
\begin{align*}
0 &=(P_q f, v)_{\Omega} 
= (P_q f, L^{\gamma}w + qw)_{\Omega}
= -(f,L^{\gamma} w)_{\R^n}. 
\end{align*}
As a consequence, $L^{\gamma} w = 0 = w$ in $W$. Thus, by Proposition \ref{prop:antiloc}, the function $w$ vanishes identically in $\R^n$, whence also $v\equiv 0$. This concludes the argument.
\end{proof}

\appendix
\section{The Higher Order Fractional Laplacian and Degenerate Elliptic Systems}
\label{sec:appendix}

\label{sec:frac_Lapl}

In this section, in order to keep our presentation self-contained, we connect the previous discussion on systems with certain boundary conditions to the properties of the higher order fractional Laplacian. Here we mainly recall several known results from the literature and rely heavily on the observations from \cite{Y13, CdMG11,RonSti16} but also refer to \cite{StingaTorrea10, GS18, KM18} and the references therein.

We split the section into two parts: First, we derive the representation of the constant coefficient higher order fractional Laplacian operators through a generalised Caffarelli-Silvestre extension. Next, we deduce analogous results for operators with non-constant coefficients. 

\subsection{The constant coefficient operator -- characterisation through a system of degenerate elliptic equations}
\label{sec:const_coef}

The starting point of our discussion is the definition of the fractional Laplacian as a Fourier multiplier:
\begin{align*}
(-\D')^{\gamma} u (x') = \F^{-1}(|\xi|^{2\gamma} \F u)(x'), \ \gamma >0,
\end{align*}
where $\F$ denotes the Fourier transform. 
Since we seek to study the unique continuation property of the higher order fractional Laplacian by techniques which are available for \emph{local} (possibly weighted) equations, we are particularly interested in Caffarelli-Silvestre type extension properties for the higher order fractional Laplacian. These exist in different generalities, we only recall two of these and refer to the literature for more general results. As we aim at applying these characterisations of the (higher order) fractional Laplacian for our study of the unique continuation property, we limit ourselves here to showing that starting from the fractional Laplacian of a function $f$, it is possible to find a suitable and sufficiently regular extension $u$ of $f$ which obeys a corresponding equation/ a corresponding system of equations. We however do not address the full equivalence (in that we do not show that any solution to the system at hand is related to the fractional Laplacian of a suitable function). For this we refer to the literature cited above.

We begin by recalling that also the higher order fractional Laplacian can be realised as the solution to a degenerate elliptic, second order boundary value problem \cite{CdMG11, Y13}:

For $\gamma \in \R_+$ we consider the equation
\begin{align}
\label{eq:higher_frac}
\begin{split}
\left( \D + \frac{1-2\gamma}{x_{n+1}} \p_{x_{n+1}} \right) u & = 0 \mbox{ in } \R^{n+1}_+,\\
u & = f \mbox{ on } \R^{n}.
\end{split}
\end{align}
Here we are interested in solutions
\begin{itemize} 
\item[(i)] which (by elliptic regularity) are classical solutions in $\R^n \times (0,\infty)$, 
\item[(ii)] which attain the boundary data $f\in H^{\mu}(\R^n)$ with $\mu \in \R$ in an $H^{\mu}(\R^n)$ sense as $x_{n+1}\to 0$, 
\item[(iii)] and which have some decay at infinity in the sense that $u \in \dot{H}^{1}(\R^{n+1}_+, x_{n+1}^{1-2(\gamma-\floor{\gamma})})$, where for $t\in \R$ we define $\floor{t}:=\min\{k\in \N: \ k\leq t\}$.
\end{itemize}
Solutions to degenerate elliptic equations of this form have been investigated in the literature, even in the context of fully nonlinear equations \cite{LW98}. Working with extensions of a problem from $\R^n$ into $\R^{n+1}_+$, in this section, we use the notational convention that
\begin{align*}
x=(x',x_{n+1})\in \R^{n+1}_+, \ x'\in \R^n, \ x_{n+1}>0.
\end{align*}

As we view the strong unique continuation property for the fractional Laplacian as a strong boundary unique continuation property of the associated degenerate extension problem, it is one of our main goals to identify the associated boundary values for the generalised Caffarelli-Silvestre extension. In particular, we aim at showing that, as in the original Caffarelli-Silvestre characterisation of the fractional Laplacian, the formulation \eqref{eq:higher_frac} also allows one to compute the fractional Laplacian $(-\D')^{\gamma} f(x')$ as an ``iterated Neumann" map from the knowledge of its generalised Caffarelli-Silvestre extension $u(x',x_{n+1})$. 

\begin{lem}
\label{lem:frac_higher}
Let $\gamma>0$, $\mu \in \R $
and assume that $f\in H^{\mu}(\R^n)$. Let $\F_{x'}$ denote the tangential Fourier transform.
Then there exists an extension operator 
\begin{align*}
E_{\gamma}:  H^{\mu}(\R^n) & \rightarrow C^{\infty}_{loc}(\R^n \times (0,\infty)),\\
 f & \mapsto E_{\gamma}f = u = c_{\gamma,n}\F^{-1}_{x'}( \F_{x'}(f)(\xi) \phi_{\gamma}(|\xi| x_{n+1}))(x) = c_{\gamma,n} f\ast G_{\gamma}(x), \ c_{\gamma,n}\neq 0,
\end{align*}
such that $E_{\gamma}f$ is a solution to 
\begin{align*}
\D u + \frac{1-2\gamma}{x_{n+1}} \p_{x_{n+1}}u & = 0 \mbox{ in } \R^{n+1}_+,
\end{align*}
and
\begin{align}
\label{eq:boundary_data}
E_{\gamma} f(x',x_{n+1}) \rightarrow f(x') \mbox{ in } H^{\mu}(\R^n) \mbox{ as } x_{n+1}\rightarrow 0.
\end{align}
Here $\phi_{\gamma}(t) = t^{\gamma}K_{\gamma}(t)$ and $K_{\gamma}(t)$ denotes a modified Bessel function of the second kind.

Also, there exists a constant $\tilde{c}_{\gamma,n} \neq 0$ such that
\begin{align}
\label{eq:frac_higher}
\tilde{c}_{\gamma,n}  x_{n+1}^{1-2\gamma + 2\floor{\gamma}} \p_{x_{n+1}} ((x_{n+1}^{-1}\p_{x_{n+1}})^{\floor{\gamma}} u(x',x_{n+1})) \rightarrow (-\D')^{\gamma} f(x') \mbox{ in } H^{\mu-2\gamma}(\R^n) \mbox{ as } x_{n+1}\rightarrow 0.
\end{align}
\end{lem}

\begin{proof}
We first derive the desired representation of the extension operator:
To this end, we solve \eqref{eq:higher_frac} by means of a tangential Fourier transform. Fourier transforming in the tangential directions, one obtains for the partial Fourier transform $\hat{u}(\xi,x_{n+1}):= \F_{x'}u(\xi, x_{n+1})$ the following ODE
\begin{align*}
\hat{u}'' + \frac{1-2\gamma}{ x_{n+1}} \hat{u}' - |\xi|^2 \hat{u} & = 0 \mbox{ in } (0,\infty),\\
\hat{u} & = \hat{f} \mbox{ on } \{x_{n+1}=0\}.
\end{align*}
We rewrite $\hat{u}(\xi,x_{n+1}) = v(\xi,|\xi|x_{n+1})$ and deduce a similar ODE for this function, but where we can scale out the $|\xi|$ contribution:
\begin{align*}
v'' + \frac{1-2\gamma}{ x_{n+1}} v' - v & = 0 \mbox{ in } (0,\infty),\\
v & = \hat{f} \mbox{ on } \{x_{n+1}=0\}.
\end{align*}
Further, setting $v(\xi,x_{n+1})=x_{n+1}^{\gamma} g(\xi, x_{n+1})$ for some function $g$, we are lead to a modified Bessel equation for $g$ (as a function of $x_{n+1}$):
\begin{align*}
x_{n+1}^2 g'' + x_{n+1} g' - (\gamma^2 + x_{n+1}^2) g & = 0 \mbox{ in } (0,\infty),
\end{align*}
with corresponding initial conditions. Since we are looking for a function with decay at infinity, by the asymptotics of modified Bessel functions (see \cite{NIST}) we infer that $g(x_{n+1},\xi) = C(\xi) K_{\gamma}(x_{n+1})$, where $K_{\gamma}(t)$ denotes a modified Bessel function of the second kind.

Returning to our original variables and using the asymptotics of $K_{\gamma}(t)$ as $t \rightarrow 0$, we thus obtain that $\hat{u}(\xi, x_{n+1}) = c_{\gamma,n} \hat{f}(\xi) \phi_{\gamma}(|\xi| x_{n+1})$, where $\phi_{\gamma}(t) = t^{\gamma} K_{\gamma}(t)$ and $c_{\gamma,n}\neq 0$.

By the regularity of $\phi_{\gamma}(t)$ for $t>0$ the function $u(x):= \F^{-1}_{x'}(c_{\gamma,n} \hat{f}(\xi) \phi_{\gamma}(|\xi| x_{n+1}))$ is $C^{\infty}_{loc}(\R^n \times (0,\infty))$. 

In order to observe the $H^{\mu}(\R^n)$ convergence from \eqref{eq:boundary_data} we note that 
\begin{align*}
\|E_{\gamma}f(\cdot, x_{n+1}) -f(\cdot)\|_{H^{\mu}(\R^n)}
&= \|(1+|\xi|^2)^{\frac{\mu}{2}} \left( c_{\gamma,n}\F_{x'}(f)(\xi)\phi_{\gamma}(|\xi|x_{n+1}) - \F_{x'}(f)(\xi) \right)\|_{L^2(\R^n)}\\
&\leq  \|(1+|\xi|^2)^{\frac{\mu}{2}}  \F_{x'}(f)(\xi)\left(c_{\gamma,n}\phi_{\gamma}(|\xi|x_{n+1})- 1 \right)\|_{L^2(\{|\xi|<\epsilon x_{n+1}^{-1}\})}\\
& \quad + \|(1+|\xi|^2)^{\frac{\mu}{2}}\F_{x'}(f)(\xi)( c_{\gamma,n} \phi_{\gamma}(|\xi|x_{n+1})-1)\|_{L^2(\{|\xi|\geq \epsilon x_{n+1}^{-1}\})}\\
& \leq \|(1+|\xi|^2)^{\frac{\mu}{2}}\F_{x'}(f)(\xi)\|_{L^2(\R^n)} \sup\limits_{|z|\leq \epsilon} |c_{\gamma,n}\phi_{\gamma}(z)-1|\\
& \quad + \|(1+|\xi|^2)^{\frac{\mu}{2}} \F_{x'}(f)(\xi)\|_{L^2(\{|\xi|\geq \epsilon x_{n+1}^{-1}\})} \sup\limits_{|z|> \epsilon} |c_{\gamma,n}\phi_{\gamma}(z)-1|.
\end{align*}
Using that $c_{n,\gamma}\phi_{\gamma}(\epsilon) \rightarrow 1$ as $\epsilon \rightarrow 0$, the boundedness of $\phi_{\gamma}(t)$ as $t\rightarrow \infty$ and the fact that for $f\in H^{\mu}(\R^n)$, $\|(1+|\xi|^2)^{\frac{\mu}{2}} \F(f)(\xi)\|_{L^2(\{|\xi|\geq \epsilon x_{n+1}^{-1}\})} \rightarrow 0$ as $x_{n+1}\rightarrow 0$, this implies the desired limiting behaviour.

Finally, in order to obtain \eqref{eq:frac_higher}, we now use the asymptotics and recurrence relations of the modified Bessel functions \cite{NIST}. We have
\begin{align}
\begin{split}
\label{eq:Bessel}
\frac{d}{dt}(t^{s}K_s(t))& = c_s t^{s} K_{s-1}(t) ,\\
K_{s}(t) &= c_s t^{-s} \mbox{ as } t \rightarrow 0, \ s > 0,\
K_{-s}(t) = K_{s}(t) \mbox{ for } s \geq 0.
\end{split}
\end{align}
Recalling the expression for $u$ (or rather $\hat{u}$), we thus obtain
\begin{align*}
x_{n+1}^{-1}\p_{x_{n+1}} \hat{u}(\xi,x_{n+1})
& = c_{\gamma,n} x_{n+1}^{-1} \p_{x_{n+1}} (\hat{f}(\xi) \phi_{\gamma}(|\xi|x_{n+1}))
 = c_{\gamma,n} x_{n+1}^{-1} \hat{f}(\xi)|\xi| \phi'_\gamma|_{|\xi| x_{n+1}}\\
&\stackrel{\eqref{eq:Bessel}}{=} \tilde{c}_{\gamma,n} |\xi|^2 \hat{f}(\xi) \phi_{\gamma-1}(|\xi|x_{n+1}).
\end{align*}
Abbreviating $N_{\gamma}(f)(x', x_{n+1}):= x_{n+1}^{1-2\gamma + 2\floor{\gamma}} \p_{x_{n+1}} ((x_{n+1}^{-1}\p_{x_{n+1}})^{\floor{\gamma}} u(x',x_{n+1}))$, we thus infer
\begin{align*}
\F_{x'}N(f)(\xi,x_{n+1})
& = C_{\gamma,n} |\xi|^{2 \floor{\gamma}} \hat{f}(\xi)|\xi| K_{\gamma-\floor{\gamma}-1}(|\xi|x_{n+1})x_{n+1}^{1-2\gamma+2\floor{\gamma}}(|\xi| x_{n+1})^{\gamma-\floor{\gamma}}\\
& = |\xi|^{2\gamma} \hat{f}(\xi) \left( C_{\gamma,n} K_{1-\gamma-\floor{\gamma}}(|\xi| x_{n+1}) (|\xi| x_{n+1})^{1-\gamma+\floor{\gamma}} \right) .
\end{align*}
As a consequence, for $\overline{C}_{\gamma,n}:=\tilde c_{\gamma,n} C_{\gamma,n}$,
\begin{align*}
&\|\tilde{c}_{\gamma,n}N_{\gamma}f(\cdot, x_{n+1}) -(-\D')^\gamma f(\cdot)\|_{H^{\mu-2\gamma}(\R^n)}\\
&= \|(1+|\xi|^2)^{\frac{\mu-2\gamma}{2}} \left( |\xi|^{2\gamma} \F_{x'}(f)(\xi) \left( \overline{C}_{\gamma,n} K_{1-\gamma-\floor{\gamma}}(|\xi| x_{n+1}) (|\xi| x_{n+1})^{1-\gamma+\floor{\gamma}} \right) - |\xi|^{2\gamma}\F_{x'}(f)(\xi) \right)\|_{L^2(\R^n)}\\
&\leq  \|(1+|\xi|^2)^{\frac{\mu-2\gamma}{2}} |\xi|^{2\gamma} \F_{x'}(f)(\xi) \left( \overline{C}_{\gamma,n} K_{1-\gamma-\floor{\gamma}}(|\xi| x_{n+1}) (|\xi| x_{n+1})^{1-\gamma+\floor{\gamma}}-1 \right)\|_{L^2(\{|\xi|<\epsilon x_{n+1}^{-1}\})}\\
& \quad + \|(1+|\xi|^2)^{\frac{\mu-2\gamma}{2}}|\xi|^{2\gamma} \F_{x'}(f)(\xi)\left( \overline{C}_{\gamma,n} K_{1-\gamma-\floor{\gamma}}(|\xi| x_{n+1}) (|\xi| x_{n+1})^{1-\gamma+\floor{\gamma}}-1 \right)\|_{L^2(\{|\xi|\geq \epsilon x_{n+1}^{-1}\})}\\
& \leq \|(1+|\xi|^2)^{\frac{\mu}{2}}\F_{x'}(f)(\xi)\|_{L^2(\R^n)} \sup\limits_{|z|\leq \epsilon} \left|\overline{C}_{\gamma,n} K_{1-\gamma-\floor{\gamma}}(z) z^{1-\gamma+\floor{\gamma}}-1 \right|\\
& \quad + \|(1+|\xi|^2)^{\frac{\mu}{2}} \F_{x'}(f)(\xi)\|_{L^2(\{|\xi|\geq \epsilon x_{n+1}^{-1}\})} \sup\limits_{|z|> \epsilon} \left|\overline{C}_{\gamma,n} K_{1-\gamma-\floor{\gamma}}(z) z^{1-\gamma+\floor{\gamma}}-1 \right|.
\end{align*}
Due to the asymptotics of the modified Bessel functions (see \eqref{eq:Bessel}) and the regularity of $f$ in the limit $x_{n+1}\rightarrow 0$ this implies the desired result for a proper choice of $\tilde c_{\gamma,n}$.
\end{proof}

\begin{cor}
\label{cor:eq_sing}
Let $f\in H^{\gamma}(\R^n)$ and let $u=E_{\gamma}f$ be the extension from Lemma ~\ref{lem:frac_higher}. Then we also have the following bulk estimates:
\begin{align}
\label{eq:bulk_CS}
\begin{split}
\| \F^{-1}_{x'}(|\xi|^{\gamma + \frac{1}{2}}\F_{x'} u)\|_{L^2(\R^{n+1}_+)}
& \leq C \|f\|_{H^{\gamma}(\R^n)},\\
\|x_{n+1}^{\frac{1-2(\gamma - \floor{\gamma})}{2}} \nabla u \|_{L^2(\R^{n+1}_+)}
& \leq C \|f\|_{H^{\gamma-\floor{\gamma}}(\R^n)},
 \end{split}
\end{align}
where $\F_{x'}$ denotes the tangential Fourier transform.
\end{cor}

\begin{proof}
In order to deduce the bulk estimates from \eqref{eq:bulk_CS} we note that
\begin{align*}
\|\F^{-1}_{x'}(|\xi|^{\gamma + \frac{1}{2}}\F_{x'} u)\|_{L^2(\R^{n+1}_+)}
&= \||\xi|^{\gamma + \frac{1}{2}} \F_{x'} u\|_{L^2(\R^{n+1}_+)}
= \||\xi|^{\gamma + \frac{1}{2}}c_{\gamma,n} \hat{f}(\xi) \phi_{\gamma}(|\xi| x_{n+1})\|_{L^2(\R^{n+1}_+)}\\
&= c_{\gamma,n}\| \phi_{\gamma}(z)\|_{L^2((0,\infty))}\||\xi|^{\gamma} \F_{x'}f\|_{L^2(\R^n)},
\end{align*}
where we used the change of coordinates $z=|\xi| x_{n+1}$.
Using an analogous change of coordinates, we also obtain
\begin{align*}
\|x_{n+1}^{\frac{1-2(\gamma-\floor{\gamma})}{2}}\nabla' u\|_{L^2(\R^{n+1}_+)}
&= c_{\gamma,n}\|(|\xi|x_{n+1})^{\frac{1-2(\gamma-\floor{\gamma})}{2}}|\xi|^{1-\frac{1-2(\gamma-\floor{\gamma})}{2}} (\F_{x'}f(\xi))\phi_{\gamma}(|\xi|x_{n+1})\|_{L^2(\R^{n+1}_+)}\\
& = c_{\gamma,n}\|z^{\frac{1-2(\gamma-\floor{\gamma})}{2}}\phi_{\gamma}(z)\|_{L^2((0,\infty))}\||\xi|^{\gamma-\floor{\gamma}} \F_{x'}f\|_{L^2(\R^n)},
\end{align*}
and
\begin{align*}
&\|x_{n+1}^{\frac{1-2(\gamma-\floor{\gamma})}{2}}\p_{x_{n+1}} u\|_{L^2(\R^{n+1}_+)}
= c_{\gamma,n}\|x_{n+1}^{\frac{1-2(\gamma-\floor{\gamma})}{2}}|\xi|(\F_{x'} f(\xi)) \phi_{\gamma}'|_{|\xi|x_{n+1}}\|_{L^2(\R^{n+1}_+)} \\
& = C_{\gamma,n}\|x_{n+1}^{\frac{1-2(\gamma-\floor{\gamma})}{2}}|\xi|(\F_{x'} f(\xi)) (|\xi|x_{n+1})\phi_{\gamma-1}(|\xi|x_{n+1})\|_{L^2(\R^{n+1}_+)} \\
&  = C_{\gamma,n}\|(|\xi| x_{n+1})^{\frac{1-2(\gamma-\floor{\gamma})}{2}}|\xi|^{1-\frac{1-2(\gamma-\floor{\gamma})}{2}}(\F_{x'} f(\xi)) (|\xi|x_{n+1})\phi_{\gamma-1}(|\xi|x_{n+1})\|_{L^2(\R^{n+1}_+)} \\
& = C_{\gamma,n}
\|z^{\frac{3-2(\gamma-\floor{\gamma})}{2}}\phi_{\gamma-1}(z)\|_{L^2((0,\infty))} \||\xi|^{\gamma-\floor{\gamma}} \F_{x'} f\|_{L^2(\R^n)}.
\end{align*} 
Here, in the passage from $\phi'_{\gamma}(t)$ to $t \phi_{\gamma-1}(t)$, we used the recurrence relations \eqref{eq:Bessel}.
This concludes the discussion of the mapping properties of $E_{\gamma}$ and provides the estimates from \eqref{eq:bulk_CS}.
\end{proof}

While the formulation \eqref{eq:frac_higher} already provides a convenient alternative \emph{local} characterisation of the fractional Laplacian as an iterated and weigthed Dirichlet-to-Neumann map for a second order equation in the upper half-plane, if $\gamma \notin (0,1)$ it is not immediately associated in a natural way with a \emph{finite} energy (the quantity $\|x_{n+1}^{\frac{1-2\gamma}{2}} \nabla u\|_{L^2(\R^{n+1}_+)}$ diverges in general). 

In order to remedy this, in the sequel, we recall that the fractional Laplacian is also related to a Dirichlet-to-Neumann map for a system (or, equivalently, a higher order equation) which can naturally be associated with a finite energy \cite{Y13}. This provides the natural functional analytic framework for our discussion of the unique continuation properties of the higher order fractional Laplacian and explains our focus on unique continuation properties for systems with Muckenhoupt weights in the earlier sections.

In order to derive the desired higher order equation for 
$u$, we begin by discussing the bulk equation:

\begin{lem}[Lemma 4.2 in \cite{Y13}]
\label{lem:bulk_higher}
Let $\gamma>0$ and let $u$ be a solution to the bulk equation in \eqref{eq:higher_frac}. Then, for $k\in\{0,\dots,\floor{\gamma}\}$ the function $w_k =(\D_b)^k u$ with $\D_b : = x_{n+1}^{-b} \nabla \cdot x_{n+1}^{b} \nabla$ and $b= 1-2\gamma + 2\floor{\gamma} \in (-1,1) $ satisfies
\begin{align}
\label{eq:k_equation}
\D w_k + \frac{1-2\gamma+2k}{x_{n+1}} \p_{x_{n+1}} w_k & = 0 \mbox{ in } \R^{n+1}_+.
\end{align}
In particular, 
\begin{align}
\label{eq:bulk}
(\D_b)^{\floor{\gamma}+1} u = 0 \mbox{ in } \R^{n+1}_+.
\end{align}
\end{lem}

The equation \eqref{eq:bulk} provides us with the bulk equation which we are working with in the sequel. For self-containedness, we recall the argument for Lemma ~\ref{lem:bulk_higher} from \cite{Y13}.

\begin{proof}
We show that if a function $w$ solves
\begin{align}
\label{eq:a}
\D w + \frac{a}{x_{n+1}} \p_{x_{n+1}} w & = 0 \mbox{ in } \R^{n+1}_+
\end{align}
with $a\in \R$,  then $w_1:=\D_b w$ solves
\begin{align*}
\D w_1 + \frac{a+2}{x_{n+1}} \p_{x_{n+1}} w_1 & = 0 \mbox{ in } \R^{n+1}_+.
\end{align*}
To this end, we observe that
\begin{align*}
\D w_1 
&= \D \D_b w
= \D \left(\D w + \frac{b}{x_{n+1}}\p_{x_{n+1}}w \right)\\
&\stackrel{\eqref{eq:a}}{=} 
\D \left( \frac{b-a}{x_{n+1}} \p_{x_{n+1}}w \right)
= \frac{b-a}{x_{n+1}} \D \p_{x_{n+1}} w + \frac{2(a-b)}{x_{n+1}^2} \p_{x_{n+1}}^2 w + \frac{2(b-a)}{x_{n+1}^3} \p_{x_{n+1}} w\\
&= (2+a) \frac{b-a}{x_{n+1}^2}\left( \frac{\p_{x_{n+1}}w}{x_{n+1}} - \p_{x_{n+1}}^2 w \right)
= -\frac{2+a}{x_{n+1}}\p_{x_{n+1}}\left(\frac{b-a}{x_{n+1}} \p_{x_{n+1}} w \right)\\
&\stackrel{\eqref{eq:a}}{=} -\frac{2+a}{x_{n+1}} \p_{x_{n+1}} w_1.
\end{align*}
This concludes the proof.
\end{proof}

Next we seek to complement \eqref{eq:bulk} with suitable boundary conditions.
To this end, we use the explicit form of $u$ which was deduced in the proof of Lemma ~\ref{lem:frac_higher}. It entails the validity of certain weighted Neumann conditions and provides tangential limits for the associated higher order Dirichlet data:

\begin{lem}
\label{lem:Neumann_higher}
Let $\gamma>0$, $\mu \in \R$, let $f\in H^{\mu}(\R^n)$ and let $u =E_{\gamma}f= G_{\gamma}\ast f$ be the solution to \eqref{eq:higher_frac} from Lemma ~\ref{lem:frac_higher}. For $x_{n+1}>0$ and $k\in\{0,\dots,\floor{\gamma}\}$, define $w_k:=(\D_b)^k u$, where $\D_b := x_{n+1}^{-b}\nabla \cdot x_{n+1}^b \nabla$ and $b= 1-2\gamma+2\floor{\gamma}\in(-1,1)$. Then we have that for all $k\in\{0,\dots,\floor{\gamma}\}$ it holds
\begin{align}
\label{eq:boundary_higher_frac}
 w_k(x',x_{n+1}) \rightarrow c_{n,\gamma,k}(-\D')^k u(x',0)
= c_{n,\gamma,k}(-\D')^k f(x') \mbox{ in } H^{\mu-2k}(\R^n) \mbox{ as } x_{n+1}\rightarrow 0,
\end{align}
for some constant $c_{n,\gamma,k}\neq 0$.
The functions $w_k$ are in $C^{\infty}_{loc}(\R^n \times (0,\infty))$ and for any $\nu \in \R$ they obey the bounds
\begin{align}
\label{eq:bulk_CS_1}
\begin{split}
\| \F^{-1}_{x'}(|\xi|^{\nu + \frac{1}{2}}\F_{x'} w_k)\|_{L^2(\R^{n+1}_+)} &\leq \|f\|_{\dot{H}^{\nu+2k}(\R^n)}, \\ 
\||\nabla'|^{\nu} x_{n+1}^{\frac{1-2(\gamma - \floor{\gamma})}{2}} \nabla w_k \|_{L^2(\R^{n+1}_+)}
& \leq C \|f\|_{\dot{H}^{2k+\gamma-\floor{\gamma}+\nu}(\R^n)}.
 \end{split}
\end{align}
 Moreover, for any $k\in\{0,\dots,\floor{\gamma}-1\}$
\begin{align}
\label{eq:boundary_higher}
 x_{n+1}^{1-2\gamma + 2\floor{\gamma}} \p_{x_{n+1}} w_k(x',x_{n+1}) \rightarrow 0 \mbox{ in } H^{\mu-2\gamma + 2\floor{\gamma}-2k}(\R^n) \mbox{ as } x_{n+1}\rightarrow 0,
\end{align}
and for some constant $c_{n,\gamma} \neq 0$
\begin{align}
\label{eq:boundary_higher_m}
 x_{n+1}^{1-2\gamma + 2\floor{\gamma}} \p_{x_{n+1}} w_{\floor{\gamma}}(x',x_{n+1}) \rightarrow c_{n,\gamma} (-\D')^{\gamma} f(x') \mbox{ in } H^{\mu-2\gamma}(\R^n) \mbox{ as } x_{n+1}\rightarrow 0.
\end{align}

\end{lem}

\begin{proof}
First, by induction, we show that 
\begin{align}
\label{eq:induct}
&\hat{w}_{k}(\xi, x_{n+1}) = c_{n,\gamma, k}|\xi|^{2k} \hat{f}(\xi)\phi_{\gamma-k}(|\xi| x_{n+1}).
\end{align}
For $k=0$ this is true by Lemma ~\ref{lem:frac_higher}. It thus suffices to prove the induction step. By Lemma ~\ref{lem:bulk_higher} and the equation for $w_k$ we have 
\begin{align*}
w_{k+1}= \D_b w_k = \frac{2\floor{\gamma}-2k}{x_{n+1}} \p_{x_{n+1}} w_k \mbox{ in } \R^{n+1}_+.
\end{align*} 
Using the claimed representation for $w_k$, i.e.  $\hat{w}_k(\xi, x_{n+1}) = c_{n,\gamma, k} |\xi|^{2k} \hat{f}(\xi)\phi_{\gamma-k}(|\xi| x_{n+1})$ and the asymptotic recurrence relations for modified Bessel functions ~\eqref{eq:Bessel}, this directly implies
\begin{align*}
\hat{w}_{k+1}(\xi,x_{n+1}) &= \frac{2\floor{\gamma}-2k}{x_{n+1}} \p_{x_{n+1}} \hat{w}_k(\xi,x_{n+1})\\
&= (2\floor{\gamma}-2k)c_{n,\gamma,k} |\xi|^{2k} \hat{f}(\xi)|\xi|^2 C_{\gamma,k}\phi_{\gamma-k-1}(|\xi|x_{n+1})\\
&= c_{n,\gamma,k+1}|\xi|^{2k+2} \hat{f}(\xi)\phi_{\gamma-k-1}(|\xi|x_{n+1}).
\end{align*}

The  representation from \eqref{eq:induct} together with the asymptotics from \eqref{eq:Bessel} then  directly entails \eqref{eq:boundary_higher_frac}, \eqref{eq:boundary_higher} and \eqref{eq:boundary_higher_m}. 
Here as in the proof of Lemma ~\ref{lem:frac_higher} all limits should be understood in the corresponding Sobolev spaces.
Indeed, for instance, if $k\in\{0,\dots,\floor{\gamma}-1\}$, by invoking \eqref{eq:Bessel}, we have
\begin{align*}
\p_{x_{n+1}} \hat{w}_k(\xi,x_{n+1})
& = c_{n,\gamma,k}|\xi|^{2k} \hat{f}(\xi) |\xi| \phi'_{\gamma-k}|_{|\xi|x_{n+1}}
= \tilde{c}_{n,\gamma,k}|\xi|^{2k+1} \hat{f}(\xi) (|\xi|x_{n+1})\phi_{\gamma-k-1}(|\xi|x_{n+1}).
\end{align*}
Thus, similarly as in the proof of Lemma ~\ref{lem:frac_higher}, we infer
\begin{align*}
&\tilde{c}_{n,\gamma,k}^{-1}\|x_{n+1}^{1-2\gamma+2\floor{\gamma}}\p_{x_{n+1}} \hat{w}_k(\xi,x_{n+1})\|_{H^{\mu-2\gamma+2\floor{\gamma}-2k}(\R^n)}\\
&= \|(1+|\xi|^2)^{\frac{\mu-2\gamma+2\floor{\gamma}-2k}{2}} |\xi|^{2k+2\gamma-2\floor{\gamma}} \F(f)(\xi) (|\xi|x_{n+1})^{2-2\gamma+2\floor{\gamma}}\phi_{\gamma-k-1}(|\xi| x_{n+1}) \|_{L^2(\R^n)}\\
& \leq \|(1+|\xi|^2)^{\frac{\mu}{2}}\F(f)(\xi)\|_{L^2(\R^n)} \sup\limits_{|z|\leq \epsilon} |z^{2-2\gamma-\floor{2\gamma}}\phi_{\gamma-k-1}(z)|\\
& \quad + \|(1+|\xi|^2)^{\frac{\mu}{2}} \F(f)(\xi)\|_{L^2(\{|\xi|\geq \epsilon x_{n+1}^{-1}\})} \sup\limits_{|z|> \epsilon} |z^{2-2\gamma-\floor{2\gamma}}\phi_{\gamma-k-1}(z)|.
\end{align*}
Relying on the asymptotics of the Bessel functions and  passing to the limit $x_{n+1}\rightarrow 0$ then leads to \eqref{eq:boundary_higher}. The remaining limits are obtained analogously.

Finally, the bounds from \eqref{eq:bulk_CS_1} are deduced as those from \eqref{eq:bulk_CS} in Corollary ~\ref{cor:eq_sing}.
\end{proof}

In analogy to the notation from Lemma ~\ref{lem:frac_higher}, we introduce the notation $E_{\gamma,k}f:=w_k$.

We next show that it is possible to obtain localised regularity estimates for the functions $w_k$ from Lemma ~\ref{lem:bulk_higher}. This is helpful in the discussion of the global unique continuation properties for the fractional Laplacian. These are particularly relevant in the analysis of associated fractional inverse problems.

\begin{lem}
\label{lem:localized}
Let $\nu \in \R$ and assume that $f\in H^{2k+\gamma-\floor{\gamma}+2N}(B_1')$ for some $N\in \N$ and ${k\in\{0,\dots,\floor{\gamma}\}}$. Then for all $r\in (0,1)$ we have that
\begin{align*}
\||\nabla'|^{2N}x_{n+1}^{\frac{1-2(\gamma-\floor{\gamma})}{2}} \nabla w_k\|_{L^2(B_{r}^+)} \leq C_{r} < \infty.
\end{align*}
\end{lem}

The argument for this relies on the pseudolocality of the operator at hand.

\begin{proof}
We split 
\begin{align*}
w_k = E_{\gamma,k}(f\eta) + E_{\gamma,k}(f(1-\eta)),
\end{align*}
where $\eta$ is a cut-off function that is one on $B_{r}'$ for some $r\in(0,1)$ and vanishes outside of $B'_1$.
For $E_{\gamma,k}(f\eta)$ the claim is a direct consequence of Lemma ~\ref{lem:Neumann_higher}. It hence suffices to study the regularity of $ E_{\gamma,k}(f(1-\eta))$. To this end, we argue as in \cite{RS17}. For convenience of notation we only prove the argument for $k=0$; the argument for $k\in \{1,\dots,\floor{\gamma}\}$ is analogous. Let $\psi$ be a second smooth cut-off function which is equal to one on the support of $\eta$ and vanishes outside of $B_{1}'$. Then,
\begin{align*}
(\psi u_2)(x',\epsilon) = \psi(x') (P_{\epsilon}^{\gamma}\ast ((1-\eta)f))(x')=:T_{\epsilon}f(x').
\end{align*}
By an explicit computation, we obtain that $P_{\epsilon}^{\gamma}(x'):=\F^{-1}_{x'} (\phi_{\gamma}(\epsilon |\cdot|))(x') = c_{\gamma,n} \frac{\epsilon^{2\gamma}}{(|x'|^2 + \epsilon^2)^{\frac{n+2\gamma}{2}}} $ (this exploits formula 9.6.25 in \cite{AS65}). Using the explicit form of $P_{\epsilon}^{\gamma}(x')$ or heat kernel estimates (as outlined in the next section and in \cite{CS16}), we obtain that for any $a>0$
\begin{align}
\label{eq:Pgamma}
\begin{split}
\int\limits_{|z|> a} |(\nabla')^{\alpha}P_\epsilon^{\gamma}(z)| dz& =\int\limits_{|z|> a} \epsilon^{-n}|(\nabla')^{\alpha}P_1^{\gamma}\Big(\frac z \epsilon\Big)| dz 
= \int\limits_{|y|>a/\epsilon} \epsilon^{-|\alpha|}|(\nabla')^{\alpha} P_1^{\gamma}(y)| dy\\
& \leq C_{n,\alpha,\gamma} \int\limits_{|y|>a/\epsilon} \epsilon^{-|\alpha|} |y|^{-n-2\gamma-|\alpha|} dy \leq C \epsilon^{2\gamma}.
\end{split}
\end{align}
As the convolution in the expression for $T_{\epsilon}$  is only active in regions in which $|x'-y'|>a$ for some suitable $a>0$, by virtue of Schur's lemma and an integration by parts, we then deduce that 
\begin{align*}
\|\langle \nabla' \rangle^{2N} T_{\epsilon}(\langle \nabla' \rangle^{2N} g)\|_{L^2(\R^n)} \leq C \epsilon^{2\gamma} \|g\|_{L^2(\R^n)},
\end{align*}
whence 
\begin{align*}
\|T_{\epsilon} f\|_{H^{2N}(\R^n)} \leq C_N \epsilon^{2\gamma} \|f\|_{H^{-2N}(\R^n)}. 
\end{align*}
Integrating in $x_{n+1}$, then implies 
that
\begin{align*}
\|x_{n+1}^{\frac{1-2(\gamma-\floor{\gamma})}{2}} (\nabla')^{2N} u_2 \|_{L^2(B_r' \times (0,1))}
&\leq C \|x_{n+1}^{\frac{1-2(\gamma-\floor{\gamma})}{2}}  (\nabla' )^{2N} (\psi u_2)\|_{L^2(B_r' \times (0,1))}\\
&\leq C \Big(\int\limits_{0}^{1} \epsilon^{{1-2(\gamma-\floor{\gamma})}}  \|T_{\epsilon}f\|_{H^{2N}(\R^n)}^2 d\epsilon\Big)^{\frac 1 2} \leq C_\gamma \|f\|_{H^{-2N}(\R^n)},
\end{align*}
which is the desired statement.

For the estimate of the normal derivative $x_{n+1}^{1-2(\gamma-\floor{\gamma})}\p_{n+1} u$ we notice that a short computation shows that
\begin{align*}
x_{n+1}^{1-2(\gamma-\floor{\gamma})}\p_{n+1} u
=
\left\{
\begin{array}{ll}
 c_{n,\gamma} x_{n+1}^{2-2(\gamma-\floor{\gamma})}  E_{\gamma-1}((-\D') f), &\ \mbox{ if } \gamma>1, \\
 c_{n,\gamma} E_{1-\gamma}((-\D')^{\gamma}f), & \ \mbox{ if } \gamma\in (0,1).
\end{array}
\right.
\end{align*}
Arguing as above then concludes the proof.
\end{proof}

We summarise the results from this section for $f\in H^{2\gamma}(\R^n)$:

\begin{prop}
\label{prop:H2gamma}
Let $\gamma>0$ and let $f\in H^{2\gamma}(\R^n)$. Then the function $u:=E_{\gamma}(f) \in C^{\infty}_{loc}(\R^{n+1}_+)\cap H^{1}(\R^{n+1}_+, x_{n+1}^{1+2\gamma - 2 \floor{\gamma}})$ is a solution to the scalar higher order problem
\begin{align*}
(\D_b)^{\floor{\gamma}+1} u & = 0 \mbox{ in } \R^{n+1}_+,\\
\lim\limits_{x_{n+1}\rightarrow 0} u &= f \mbox{ on } \R^n \times \{0\}, \\
\lim\limits_{x_{n+1}\rightarrow 0}  (\D_b)^k u &= c_{n,\gamma,k} \lim\limits_{x_{n+1}\rightarrow 0}(-\D')^k f \mbox{ on } \R^n \times \{0\} \mbox{ for } k \in\{1,\dots, \floor{\gamma}\},\\
\lim\limits_{x_{n+1}\rightarrow 0} x_{n+1}^{1-2\gamma + 2\floor{\gamma}} \p_{x_{n+1}} (\D_b)^{\floor{\gamma}} u &= c_{n,\gamma}(-\D')^\gamma f  \mbox{ on } \R^n \times \{0\}\\
\lim\limits_{x_{n+1}\rightarrow 0} x_{n+1}^{1-2\gamma + 2\floor{\gamma}} \p_{x_{n+1}} (\D_b)^k u &= 0  \mbox{ on } \R^n \times \{0\} \mbox{ for } k \in \{0,\dots, \floor{\gamma}-1\}.
\end{align*}
All limits $x_{n+1}\rightarrow 0$ are understood in an $L^{2}(\R^n)$ sense.

Setting $u_0:=u$ and defining the functions $u_{j+1}=\D_b u_{j}$ for $j\in\{0,\dots,\floor{\gamma}-1\}$, this can also be rewritten as the following system of second order equations 
\begin{align}
\label{eq:WUCP_syst}
\begin{split}
\D_b u_{m} & = 0 \mbox{ in } \R^{n+1}_+,\\
\D_b u_{j} & = u_{j+1} \mbox{ in } \R^{n+1}_+ \mbox{ for } j\in\{0,\dots,m-1\},\\
\lim\limits_{x_{n+1}\rightarrow 0} u_j & = c_{n,\gamma,j} (-\D')^{j}f \mbox{ on } \R^n \times \{0\} \mbox{ for } j \in \{0,\dots,m\},\\
\lim\limits_{x_{n+1}\rightarrow 0} x_{n+1}^b \p_{x_{n+1}} u_{m} & = c_{n,\gamma}(-\D')^{\gamma}f \mbox{ on } \R^n \times \{0\},\\
\lim\limits_{x_{n+1}\rightarrow 0} x_{n+1}^{b}\p_{x_{n+1}} u_j & = 0 \mbox{ on } \R^n \times \{0\} \mbox{ for } j \in \{0,\dots,m-1\},
\end{split}
\end{align}
where $m=\floor{\gamma}$. Again, all limits $x_{n+1}\rightarrow 0$ are understood in an $L^{2}(\R^n)$ sense.
\end{prop}

\subsection{The variable coefficient setting -- characterisation through a system of degenerate elliptic equations}
\label{sec:var_coef}

In this section, we derive analogous results to Proposition ~\ref{prop:H2gamma} in the presence of variable coefficients, i.e. we are now concerned with the operator $L^\gamma$, where $L=-\nabla \cdot \tilde{a}^{ij} \nabla$ and the coefficients $\tilde{a}^{ij}$ satisfy the conditions stated in \hyperref[cond:a1]{(A1)}-\hyperref[cond:a3]{(A3)} with $\mu = 2\floor{\gamma}$.
In contrast to the previous argument in which the Fourier transform diagonalised the tangential operator, we here rely on a spectral decomposition. We argue analogously as in \cite{RonSti16} and thus only present the arguments formally. For convenience of notation we set
\begin{align*}
L_b:= x_{n+1}^{-b}(\p_{x_{n+1}}x_{n+1}^{b} \p_{x_{n+1}} - x_{n+1}^b L),
\end{align*}
where $L$ is as above and $b=1-2\gamma+2\floor{\gamma}$. 

To this end, we recall that for the self-adjoint, positive operator $L$ we can carry out a spectral decomposition and obtain a unique associated resolution of the identity which is supported on the spectrum of $L$ with
\begin{align*}
( Lf, g )_{L^2(\R^n)} = \int\limits_{0}^{\infty} \lambda d E_{f,g}(\lambda) \mbox{ for all } f\in \mbox{Dom}(L), \ g \in L^2(\Omega). 
\end{align*}
Based on this we can define the action of the heat semigroup and of the fractional powers of $L$:
\begin{align*}
(e^{-tL}f,g)_{L^2(\R^n)}&= \int\limits_{0}^{\infty} e^{-t \lambda} d E_{f,g}(\lambda), \ f,g \in L^2(\Omega), \ t\geq 0, \\
(L^{\gamma} f,g)_{L^2(\R^n)}&= \int\limits_{0}^{\infty} \lambda^{\gamma} d E_{f,g}(\lambda), \ f\in \Dom(L^{\gamma}),g \in L^2(\Omega), \ t\geq 0,
\end{align*}
where $\Dom(L^{\gamma}):=\{f\in L^2(\Omega): \ \int\limits_{0}^{\infty} \lambda^{2\gamma} dE_{f,f}(\lambda)<\infty\}$.

Our main result in the context of variable coefficients mirrors the statement of the constant coefficient case:

\begin{prop}
\label{prop:H2gamma1}
Let $\gamma>0$ and let $f\in \Dom(L^{\gamma})$. Then the function 
\begin{align*}
u(x',x_{n+1}):= c_{\gamma} x_{n+1}^{2\gamma} \int\limits_{0}^{\infty} e^{-tL}f(x') e^{-\frac{x_{n+1}^{2}}{4t}} \frac{dt}{t^{1+\gamma}} \in C^{2\floor{\gamma}+2,1}_{loc}(\R^n \times (0,\infty))
\end{align*}
is a solution to the scalar higher order problem
\eqref{eq:higher_order_scalar}.
The function $u(x',x_{n+1})$ can also be represented as
\begin{align*}
u(x',x_{n+1}):= \tilde{c}_{\gamma}\int\limits_{0}^{\infty} e^{-t L} L^{\gamma} f(x') e^{-\frac{x_{n+1}^2}{4t}} \frac{dt}{t^{1-\gamma}}.
\end{align*}
Setting $u_0(x',x_{n+1}):=u(x',x_{n+1})$ and defining the functions $u_{j+1}(x',x_{n+1})=L_b u_{j}(x',x_{n+1})$ for $j\in\{1,\dots,\floor{\gamma}\}$, allows one to rewrite \eqref{eq:higher_order_scalar} as the the system \eqref{eq:WUCP_syst1},
where $m=\floor{\gamma}$. 
All boundary conditions hold in an $L^2(\R^n)$ sense. 
\end{prop}

\begin{rmk}
\label{rmk:Sob_reg}
With the systems representation being established, we also obtain that $u_k \in L^{2}_{loc}(\R^{n+1}_+,x_{n+1}^b)$. Then, direct energy estimates also yield that $u_k\in H^{1}_{loc}(\R^{n+1}_+, x_{n+1}^b)$ for all $k\in\{1,\dots,\floor{\gamma}\}$, if $f\in \Dom(L^{\gamma})$.  
\end{rmk}

\begin{proof}
The fact that $u(x',x_{n+1})=c_{\gamma} x_{n+1}^{2\gamma} \int\limits_{0}^{\infty} e^{-tL}f(x') e^{-\frac{x_{n+1}^{2}}{4t}} \frac{dt}{t^{1+\gamma}}$ solves the equation
\begin{align*}
\p_{x_{n+1}} x_{n+1}^{1-2\gamma} \p_{x_{n+1}} u -x_{n+1}^{1-2\gamma} L u &= 0 \mbox{ in } \R^{n+1}_+,\\
\lim\limits_{x_{n+1}\rightarrow 0} u & = f \mbox{ on } \R^{n}\times \{0\}.
\end{align*}
follows as in \cite{RonSti16, StingaTorrea10}. By interior elliptic regularity estimates and the assumed coefficient regularity, this also implies the claimed regularity result.

We discuss the attainment of the Dirichlet boundary conditions for the function $u$: To this end, we observe that
\begin{align*}
(u(\cdot, x_{n+1}),g(\cdot))
&=c_\gamma \int\limits_{0}^{\infty} (e^{-t L} f, g) \left( \frac{x_{n+1}^2}{t}\right)^{\gamma} e^{- \frac{x_{n+1}^2}{4t}} \frac{dt}{t}\\
& =c_\gamma \int\limits_{0}^{\infty} \int\limits_{0}^{\infty} e^{-t \lambda} \left( \frac{x_{n+1}^2}{t}\right)^{\gamma} e^{- \frac{x_{n+1}^2}{4t}} \frac{dt}{t} dE_{f,g}(\lambda)\\
& = -c_\gamma\int\limits_{0}^{\infty} \int\limits_{0}^{\infty} e^{- \lambda \frac{x_{n+1}^2}{z}} z^{\gamma} e^{- \frac z 4} \frac{dz}{z} dE_{f,g}(\lambda)
\end{align*}
Here we used the change of coordinates $z= \frac{x_{n+1}^2}{t}$. Thus, passing to the limit $x_{n+1}\rightarrow 0$ dominated convergence yields the claimed result.

Next, we seek to show that the function $u$ also satisfies the higher order equation \eqref{eq:higher_order_scalar} and the system  \eqref{eq:WUCP_syst1}.
To this end, we set $w_k(x',x_{n+1}):=L_b^k u(x',x_{n+1}) $ for $k\in \{0,\dots,\floor{\gamma}\}$. We claim that these functions solve the system \eqref{eq:WUCP_syst1}, where all boundary conditions hold in an $L^2$ sense. 

In order to observe this, analogously as in Lemma ~\ref{lem:bulk_higher}, we infer that the functions $w_k$ solve the bulk equation
\begin{align*}
(\p_{x_{n+1}}^2 + \frac{1-2\gamma +2k}{x_{n+1}} \p_{x_{n+1}} -L)w_k = 0 \mbox{ in } \R^{n+1}_+,
\end{align*}
whence
\begin{align*}
L_b w_k = \frac{c_{\gamma,k}}{x_{n+1}}\p_{x_{n+1}} w_k \mbox{ in } \R^{n+1}_+.
\end{align*}
We claim that 
\begin{align}
\label{eq:w_k}
\begin{split}
w_k(x',x_{n+1}) 
&= c_{\gamma,k} x_{n+1}^{2\gamma-2k}\int\limits_{0}^{\infty}e^{-t L} L^k f(x') e^{-\frac{x_{n+1}^2}{4t}} \frac{dt}{t^{1+\gamma-k}}\\
& = \tilde{c}_{\gamma,k} \int\limits_{0}^{\infty} e^{-t L} L^{\gamma} f(x') e^{-\frac{x_{n+1}^2}{4t}} \frac{dt}{t^{1-\gamma+k}}.
\end{split}
\end{align}
The latter in particular also shows the equivalent representation for $u$. 
We obtain the first representation for $w_k$ by induction and the following computation
\begin{align*}
(w_{k+1}(\cdot, x_{n+1}),g(\cdot)) 
&= (L_b w_k(\cdot, x_{n+1}),g(\cdot)) = \left(\frac{c_{\gamma,k}}{x_{n+1}}\p_{x_{n+1}} w_k(\cdot, x_{n+1}), g(\cdot) \right)\\
&= c_{\gamma,k} \int\limits_{0}^{\infty} (e^{-tL}L^k f, g) \frac{1}{x_{n+1}}\p_{x_{n+1}}\left( x_{n+1}^{2(\gamma-k)} e^{-\frac{x_{n+1}^2}{4t}} \right) \frac{dt}{t^{1+\gamma-k}}\\
&= c_{\gamma,k} \int\limits_{0}^{\infty} (e^{-tL}L^k f, g) x_{n+1}^{2\gamma-2k-2}\left( \frac{2(\gamma-k)}{t^{1+\gamma-k}} - \frac{x_{n+1}^{2}}{2t^{2+\gamma-k}}\right) e^{-\frac{x_{n+1}^2}{4t}}  dt\\
&= 2c_{\gamma,k}x_{n+1}^{2\gamma-2k-2} \int\limits_{0}^{\infty} (e^{-tL}L^k f, g) \p_{t}\left( t^{k-\gamma} e^{-\frac{x_{n+1}^2}{4t}} \right) dt\\
& = c_{\gamma,k+1}x_{n+1}^{2\gamma-2k-2} \int\limits_{0}^{\infty} (e^{-tL}L^{k+1}f,g) e^{-\frac{x_{n+1}^2}{4t}} \frac{dt}{t^{\gamma-k}}.
\end{align*}
Together with the representation for $u$ this yields the first identity in \eqref{eq:w_k}.
Arguing along the same lines as above (where the argument is detailed for $u$), this also immediately implies the claim on the Dirichlet data. 

We next deduce the alternative characterisation of $w_{k}(x',x_{n+1})$: We have
\begin{align*}
(w_{k}(\cdot, x_{n+1}),g(\cdot))
 & =c_{\gamma,k} \int\limits_{0}^{\infty} \int\limits_{0}^{\infty} e^{-t\lambda} \lambda^k \left( \frac{x_{n+1}^2}{t} \right)^{\gamma-k} e^{- \frac{x_{n+1}^2}{4t}} \frac{dt}{t} d E_{f,g}(\lambda)\\
 &= \tilde c_{\gamma,k}\int\limits_{0}^{\infty} \int\limits_{0}^{\infty} e^{-r \lambda} \lambda^k (r \lambda)^{\gamma-k} e^{-\frac{x_{n+1}^2}{4r}} \frac{dr}{r} d E_{f,g}(\lambda)\\
 & = \tilde c_{\gamma,k} \int\limits_{0}^{\infty} (e^{-r L} L^{\gamma} f, g) e^{-\frac{x_{n+1}^2}{4r}} \frac{dr}{r^{1-\gamma+k}}.
\end{align*}
Here we used the change of coordinates $r=\frac{x_{n+1}^2}{4t\lambda}$.

It remains to discuss the Neumann data:
\begin{align*}
(\p_{n+1} w_k(\cdot, x_{n+1}),g(\cdot))
& = -c_{\gamma,k} \int\limits_{0}^{\infty} (e^{-t L} L^{\gamma} f, g) r^{\gamma-k-1} \frac{x_{n+1}}{2} e^{-\frac{x_{n+1}^2}{4r}} \frac{dr}{r}.
\end{align*}
If $\gamma-k>1$, dominated convergence allows us to infer
\begin{align*}
&\lim\limits_{x_{n+1}\rightarrow 0}(x_{n+1}^{1-2\gamma+2\floor{\gamma}}\p_{n+1} w_k(\cdot, x_{n+1}),g(\cdot))\\
& = -\frac{c_{\gamma,k}}{2}\lim\limits_{x_{n+1}\rightarrow 0} x_{n+1}^{2-2\gamma+2\floor{\gamma}} \int\limits_{0}^{\infty} (e^{-t L} L^{\gamma} f, g) r^{\gamma-k-1}  e^{-\frac{x_{n+1}^2}{4r}} \frac{dr}{r} =0,
\end{align*}
as the $r$ powers are still integrable in zero.

For $\gamma-\floor{\gamma}<1$, we argue similarly as in \cite{StingaTorrea10}:
\begin{align*}
(x_{n+1}^{1-2\gamma+2\floor{\gamma}}\p_{n+1} w_k(\cdot, x_{n+1}),g(\cdot))
& = \tilde c_{\gamma,k} \int\limits_{0}^{\infty} \int\limits_{0}^{\infty} e^{-r \lambda} \lambda^{\gamma} r^{\gamma-k-1} x_{n+1}^{2-2\gamma+2\floor{\gamma}} e^{- \frac{x_{n+1}^2}{4r}} \frac{dr}{r} d E_{f,g}(\lambda)\\
& =\tilde c_{\gamma,k} \int\limits_{0}^{\infty}\int\limits_{0}^{\infty} e^{-\frac{z}{x_{n+1}}\lambda}\lambda^{\gamma} z^{1-\gamma+\floor{\gamma}} e^{-\frac{z}{4}} \frac{dz}{z} dE_{f,g}(\lambda).
\end{align*}
Here we used the change of coordinates $z = \frac{x_{n+1}^2}{r}$. By dominated convergence we may again pass to the limit and obtain 
\begin{align*}
\lim\limits_{x_{n+1}\rightarrow 0} (x_{n+1}^{1-2\gamma+2\floor{\gamma}}\p_{n+1} w_k(\cdot, x_{n+1}),g(\cdot)) = \tilde c_{\gamma,k} \int\limits_{0}^{\infty} z^{1-\gamma+\floor{\gamma}} e^{-z} \frac{dz}{z} (L^{\gamma}f,g).
\end{align*}
This concludes the argument.
\end{proof}

\bibliographystyle{alpha}
\bibliography{citationsHT}

\end{document}